\documentclass[letterpaper,12pt]{amsart}
\pdfoutput=1

\usepackage{amsfonts}
\usepackage[leqno]{amsmath}
\usepackage{amsthm}
\usepackage{graphicx}
\usepackage{setspace}
\usepackage[round]{natbib}
\usepackage{amssymb}
\usepackage[colorlinks,citecolor=blue,urlcolor=blue]{hyperref}
\usepackage{color}
\usepackage{multirow}
\usepackage{longtable}
\usepackage{rotating}

\newtheorem{thm}{Theorem}

\newtheorem{lem}[thm]{Lemma}

\newtheorem{assm}{Assumption}

\newtheorem{rem}{Remark}
\numberwithin{equation}{section}

\newcommand{\norm}[1]{\left\Vert#1\right\Vert}

\newcommand{\Pn}[0]{\frac{1}{n}\sum_{i=1}^n}
\newcommand{\alphahatnorm}[0]{\left\vert \widehat{\alpha}-\alpha_0 \right\vert_1}
\newcommand{\deltanorm}[0]{\left\vert \delta_0 \right\vert_1}
\newcommand{\fhatnorm}[0]{\left\Vert \widehat{f}-f_0\right\Vert_n}
\begin{document}

%%%%%%%%%%%%%%%%%%%%%%%%%%%%%%%%%%%%%%%%%%%%%%%%%%%%%%%%%
\title[The Lasso  with a Change-Point]{The Lasso for High-Dimensional Regression with a Possible  Change-Point}

\author[Lee]{Sokbae Lee}
\address{Department of Economics, Seoul National University, 1 Gwanak-ro, Gwanak-gu,
Seoul, 151-742, Republic of Korea, and
The Institute for Fiscal Studies, 7 Ridgmount Street, London, WC1E 7AE, UK.}
\email{sokbae@gmail.com}
\urladdr{https://sites.google.com/site/sokbae/.}
\author[Seo]{Myung Hwan Seo}
\address{Department of Economics, London School of Economics, Houghton
Street, London, WC2A 2AE, UK.}
\email{m.seo@lse.ac.uk}
\urladdr{http://personal.lse.ac.uk/SEO.}
\author[Shin]{Youngki Shin}
\address{Department of Economics, University of Western Ontario, 1151
Richmond Street N, London, ON N6A 5C2, Canada.}
\email{yshin29@uwo.ca}
\urladdr{http://publish.uwo.ca/\symbol{126}yshin29.}
\date{18 April 2014}
\thanks{We would like to thank Marine Carrasco,  Yuan Liao, 
Ya'acov Ritov,
two anonymous referees,
and seminar participants at various places for their helpful comments. 
This work was supported 
by the National Research Foundation of Korea Grant funded by the Korean Government (NRF-2012S1A5A8023573),
by the European Research Council (ERC-2009-StG-240910- ROMETA),
and by 
 the Social Sciences and Humanities Research Council of Canada (SSHRCC)}

\begin{abstract}
We consider a  high-dimensional regression model
with a possible change-point due to a covariate threshold
and develop the Lasso estimator of regression coefficients as well as
the threshold parameter.
Our Lasso estimator not only selects covariates but also selects a model 
between  linear and threshold regression models.
Under a sparsity assumption,  we derive non-asymptotic oracle inequalities for both the prediction risk and the $\ell_1$ estimation loss
for regression coefficients.
Since the Lasso estimator selects variables simultaneously, we show that oracle inequalities can be established without pretesting the existence of the threshold effect.
Furthermore, we establish conditions under which
the estimation error of the unknown threshold parameter can be
bounded by  a  nearly $n^{-1}$ factor 
even when the number of regressors can be much larger than the sample size ($n$).
We illustrate the usefulness of our proposed estimation method via Monte Carlo simulations
and an application to real data.
\\

\noindent \textsc{Key words.}
Lasso,
oracle inequalities,
sample splitting,
sparsity,
threshold models.

\end{abstract}

\maketitle

%%%%%%%%%%%%%%%%%%%%%%%%%%%%%%%%%%%%%%%%%%%%%%%%%%%%%%%%%%%%%%%%%%%%%%%%%%%
%\newpage

\onehalfspacing

%\doublespacing

\section{Introduction}

The Lasso and related methods have received rapidly increasing  attention in statistics since the seminal work of \citet{Tibshirani:96}.
For example, see  a timely monograph by \citet{BvdG:11} as well as  review articles by
\citet{Fan:Lv:10} and
\citet{Tibshirani:11} for general overview and recent developments.

In this paper, we develop a method for estimating a high-dimensional regression model with a possible
change-point due to a covariate threshold, while selecting relevant regressors from a set of many potential covariates.
In particular, we
propose the $\ell_1$ penalized least squares (Lasso) estimator of parameters, including
the unknown threshold parameter, and analyze its properties under a sparsity assumption when the number of possible covariates can be much larger than
the sample size.

To be specific, let $\{(Y_{i},X_{i},Q_{i}):i=1,\ldots ,n\}$ be a sample of independent
observations such that
\begin{equation}
Y_{i}=X_{i}^{\prime }\beta _{0}+X_{i}^{\prime }\delta _{0}1\{Q_{i}<\tau
_{0}\}+U_{i},\ \ \ i=1,\ldots ,n,  \label{model}
\end{equation}%
where for each $i$, $X_{i}$ is an $M\times 1$ deterministic vector, $Q_{i}$
is a deterministic scalar, $U_{i}$ follows $N(0,\sigma ^{2})$, and $1\{\cdot
\}$ denotes the indicator function.
The scalar variable $Q_i$ is the threshold variable and $\tau_0$ is the unknown threshold parameter.
Note that since $Q_i$ is a fixed variable in our setup, \eqref{model} includes
a regression model with a change-point at unknown time (e.g. $Q_i = i/n$).
Note that in this paper, we focus on the fixed design for $\{(X_i,Q_i):i=1,\ldots,n\}$
and independent normal errors $\{U_i: i=1,\ldots,n \}$. This setup has been extensively used in the literature \citep[e.g.][]{Bickel-et-al:09}.

A regression model such as \eqref{model} offers applied researchers a simple yet useful framework to model nonlinear relationships by splitting
the data into subsamples.
Empirical examples include cross-country growth models with multiple equilibria \citep{Durlauf:Johnson:95},
racial segregation \citep{Card:Mas:Rothstein:09}, and financial contagion \citep{Pes:Pick},
among many others.
Typically, the choice of the threshold variable is well motivated in applied work (e.g. initial per capita output
in \citet{Durlauf:Johnson:95}, and the
minority share in a neighborhood in \citet{Card:Mas:Rothstein:09}),
but selection of other covariates is subject to applied researchers' discretion.

However, covariate selection is important in identifying threshold effects (i.e., nonzero $\delta _{0}$) since a statistical model favoring threshold effects with a particular set of covariates
could be overturned by a linear model with a broader set of regressors.
Therefore, it seems natural to consider Lasso as a tool to estimate \eqref{model}.

The statistical problem we consider is to
estimate unknown parameters $(\beta _{0},\delta _{0},\tau _{0})\in \mathbb{R}%
^{2M+1}$ when $M$ is much larger than $n$.
For the classical setup (estimation of parameters without covariate selection when $M$ is smaller than $n$), estimation of \eqref{model} has been
well studied \citep[e.g.][]{Tong:90, chan:93, Hansen:00}.
Also, a general method for testing threshold effects in regression (i.e. testing $H_0: \delta_0 = 0$ in \eqref{model})
is available for the classical setup \citep[e.g.][]{lee2011testing}.

Although there are many  papers   on Lasso type methods
 and also equally many papers  on change points, sample splitting, and threshold models,
there seem to be only a handful of papers that intersect both topics.
\citet{Wu:08} proposed an information-based criterion for carrying out change point analysis and variable selection simultaneously in linear models with a possible change point; however, the proposed  method
in \citet{Wu:08}
would be infeasible in a sparse high-dimensional model.
In change-point models without covariates,
\citet{Harchaoui:Levy-Leduc:08, Harchaoui:Levy-Leduc:10} proposed a method for estimating the location of change-points in one-dimensional piecewise constant signals observed in white noise,
 using a penalized least-square criterion with an $\ell_1$-type penalty.
\citet{Zhang:Siegmund:12} developed Bayes Information Criterion (BIC)-like criteria for
determining the number of changes in the mean of multiple sequences of independent normal
observations  when the number of change-points can increase with the sample size.
 \citet{Ciuperca:12} considered a similar estimation problem as ours, but the corresponding analysis is restricted to the case when the number of potential covariates is small.

In this paper, we consider the Lasso estimator of regression coefficients as well as
the threshold parameter.
Since the change-point parameter $\tau_0$ does not enter additively in \eqref{model}, the resulting optimization problem in the Lasso estimation is non-convex.
We overcome this problem by comparing
the values of standard Lasso objective functions on a grid over the range of possible
values of $\tau_0$.

Theoretical properties of the Lasso and related methods for high-dimensional data are examined by
\citet{Fan:Peng:04},
\citet{Bunea-et-al:07b},
\citet{Candes:Tao:07},
\citet{Huang:Horowitz:Ma:08},
\citet{Huang:Ma:Zhang:08},
\citet{Kim:Choi:Oh:08},
\citet{Bickel-et-al:09},
and
\citet{Meinshausen:Yu:09},
among many others.
Most of the papers consider quadratic objective functions and linear or nonparametric models with an additive mean zero error.
There has been recent interest in extending this framework to generalized linear models \citep[e.g.][]{vdGeer:08,Fan:Lv:11},
to quantile regression models \citep[e.g.][]{Belloni:Chernozhukov:11,Bradic:Fan:Wang:12,Wang:Wu:Li:12},
and to hazards models \citep[e.g.][]{bradic2012regularization,Lin:Lv:12}.
%Some exceptions are, for example,
%\citet{vdGeer:08} who considered high-dimensional generalized linear models with Lipschitz loss functions;
%\citet{Fan:Lv:11} who developed nonconcave penalized likelihood estimator for high-dimensional generalized linear models;
%\citet{Bradic:Fan:Wang:12} who proposed a data-driven weighted linear combination of convex loss functions, together with weighted $L_1$-penalty for high-dimensional variable selection;
%\citet{Belloni:Chernozhukov:11} who developed the Lasso estimator of quantile regressions in high-dimensional sparse models;
%\citet{Wang:Wu:Li:12} who regularized quantile regression with a nonconvex penalty function;
%\citet{bradic2012regularization} who worked out nonconcave penalized methods for
%Cox's proportional hazards model with high-dimensional censored data;
%and
%\citet{Lin:Lv:12} who proposed a class of regularization methods for simultaneous variable selection
%and estimation in high-dimensional sparse additive hazards regression.
We contribute to this literature by
considering a regression model with a possible change-point
 and then
 deriving nonasymptotic oracle inequalities for both the prediction risk and the $\ell_1$ estimation loss
for regression coefficients under a sparsity scenario.

Our theoretical results build on \citet{Bickel-et-al:09}.
Since the Lasso estimator selects variables simultaneously, we show that oracle inequalities similar to those obtained in \citet{Bickel-et-al:09}
 can be established without pretesting the existence of the threshold effect.
In particular, when there is no threshold effect ($\delta _{0} = 0$),  we prove oracle inequalities that are basically equivalent to those in
\citet{Bickel-et-al:09}. 
Furthermore, when $\delta _{0} \neq 0$, we establish conditions under which
the estimation error of the unknown threshold parameter can be
bounded by  a  nearly $n^{-1}$ factor 
when the number of regressors can be much larger than the sample size.
To achieve this, we develop some sophisticated chaining arguments and provide sufficient regularity conditions under which we prove oracle inequalities.
The super-consistency result of $\widehat \tau$ is well known when the number of covariates is small \citep[see, e.g.][]{chan:93,Seijo:Sen:11a,Seijo:Sen:11b}.
To the best of our knowledge, our paper is the first work that demonstrates the possibility of a nearly $n^{-1}$ bound in the context of sparse high-dimensional regression models with a change-point.

The remainder of this paper is as follows.
In Section \ref{sec:est} we propose the Lasso estimator, and
in Section \ref{sec:emp} we give a brief illustration of our proposed estimation method using a real-data example in economics.
In Section \ref{sec:consist} we establish the prediction consistency of our Lasso estimator.
In Section \ref{sec:oracle} we establish sparsity oracle inequalities in terms of both the prediction loss
and the $\ell_1$ estimation loss for $(\alpha_0, \tau_0)$, while providing low-level sufficient conditions
for two possible cases of threshold effects.
In Section \ref{sec:MC} we present results of some simulation studies, and
Section \ref{sec:concl} concludes.
The appendices of the papr consist of 6 sections: Appendix \ref{sec:suff-ure}
provides sufficient conditions for one of our main assumptions,
Appendix \ref{rem-assump-discontinuity} gives some additional discussions on identifiability for $\tau_0$,
Appendices
\ref{sec:proofs:consist},  \ref{sec:proofs:oracle1}, and \ref{sec:proofs:oracle3}
contain all the proofs, and
Appendix \ref{sec:mc:extra} provides additional numerical results.
%Throughout the paper, let $a \vee b \equiv \max\{a,b\}$ and $a \wedge b \equiv %\min\{a,b\}$ for any real numbers $a$ and $b$.

\subsection*{Notation} We collect the notation used in the paper here. 
For  $\{(Y_i,X_i,Q_i): i=1,\ldots,n\}$ following \eqref{model}, let $\mathbf{X}_{i}(\tau )$ denote the $(2M\times 1)$ vector such that $\mathbf{X}_{i}(\tau )=(X_{i}^{\prime },X_{i}^{\prime }1\{Q_{i}<\tau\})^{\prime }$ and let $\mathbf{X}(\tau )$ denote the $(n\times 2M)$ matrix
whose $i$-th row is $\mathbf{X}_{i}(\tau )^{\prime }$.
For an $L$-dimensional vector $%
a,$ let $|a|_{p}$ denote the $\ell _{p}$ norm of $a,$ and $|J(a)|$ denote the
cardinality of $J(a)$, where $J(a)=\{j\in \{1,\ldots ,L\}:a_{j}\neq 0\}$. 
In addition, let $\mathcal{M}(a)$ denote the number of nonzero elements of $a$, i.e.\ $\mathcal{M}(a)=\sum_{j=1}^{L}1\{a_{j}\neq 0\}=|J(a)|$.
%The value $\mathcal{M}(\alpha _{0})$ characterizes the sparsity of the model \eqref{model1}. 
Let $a_{J}$ denote the vector in $\mathbb{R}^{L}$ that
has the same coordinates as $a$ on $J$ and zero coordinates on the
complement $J^{c}$ of $J$. For any $n$-dimensional vector $W=(W_{1},\ldots
,W_{n})^{\prime }$, define the empirical norm as
$\left\Vert W\right\Vert _{n}:=\left( n^{-1}\sum_{i=1}^{n}W_{i}^{2}\right)
^{1/2}$.
Let the superscript $^{\left( j\right) }$ denote the $j$-th element of a vector or the $j$-th column of a matrix depending on the context. 
Finally, define $f_{(\alpha ,\tau )}(x,q):=x^{\prime }\beta +x^{\prime }\delta
1\{q<\tau \}$, $f_{0}(x,q):=x^{\prime }\beta _{0}+x^{\prime }\delta
_{0}1\{q<\tau _{0}\}$, and $\widehat{f}(x,q):=x^{\prime }\widehat{\beta }%
+x^{\prime }\widehat{\delta }1\{q<\widehat{\tau }\}$.
Then, we define the prediction risk as $\left\Vert \widehat{f}-f_{0}\right\Vert _{n}
:=\left(\frac{1}{n}  \sum_{i=1}^n \left(  \widehat{f}(X_i,Q_i) -  f_{0}(X_i,Q_i) \right)^2\right)^{1/2}.$

\section{Lasso Estimation}\label{sec:est}

Let $\alpha
_{0}=(\beta _{0}^{\prime },\delta _{0}^{\prime })^{\prime }$. Then, using notation defined above, we can rewrite \eqref{model} as 
\begin{equation}
Y_{i}=\mathbf{X}_{i}(\tau _{0})^{\prime }\alpha _{0}+U_{i},\ \ \ i=1,\ldots
,n.  \label{model1}
\end{equation}

Let $\mathbf{y}\equiv (Y_{1},\ldots ,Y_{n})^{\prime }$. For any fixed $\tau \in \mathbb{T}$, 
where $\mathbb{T}\equiv [t_{0},t_{1}]$ is a parameter space for $\tau
_{0}$,
consider the residual sum of squares
\begin{align*}
S_{n}(\alpha ,\tau )& =n^{-1}\sum_{i=1}^{n}\left( Y_{i}-X_{i}^{\prime }\beta
-X_{i}^{\prime }\delta 1\{Q_{i}<\tau \}\right) ^{2} \\
& =\left\Vert \mathbf{y}-\mathbf{X}(\tau )\alpha \right\Vert _{n}^{2},
\end{align*}%
where $\alpha =(\beta ^{\prime },\delta ^{\prime })^{\prime }$.

We define the following $(2M\times 2M)$
diagonal matrix:
\begin{equation*}
\mathbf{D}(\tau ):=\text{diag}\left\{ \left\Vert \mathbf{X}^{(j)}(\tau
)\right\Vert _{n},\ \ j=1,...,2M\right\} .
\end{equation*}%
For each fixed $\tau \in \mathbb{T}$, define the  Lasso solution $\widehat{\alpha }(\tau )$
by
\begin{align}\label{Lasso-fixed-tau}
\widehat{\alpha }(\tau ):=\text{argmin}_{\alpha \in \mathcal{A} \subset \mathbb{R}^{2M}}\left\{
S_{n}(\alpha ,\tau )+\lambda \left\vert \mathbf{D}(\tau )\alpha \right\vert
_{1}\right\} ,
\end{align}%
where $\lambda $ is a tuning parameter that depends on $n$ and
$\mathcal{A}$ is a parameter space for $\alpha_0$.

It is important to note that the scale-normalizing factor $\mathbf{D}(\tau)$ depends on $\tau$ since different values of $\tau$ generate different dictionaries $\mathbf{X}(\tau)$. To see more clearly,  
define 
\begin{align}\label{x-notation-new}
\begin{split}
X^{\left( j\right) } &\equiv (X_{1}^{(j)},\ldots
,X_{n}^{(j)} )^{\prime },  \\
X^{\left( j\right) }(\tau ) &\equiv (X_{1}^{(j)}1\{Q_{1}<\tau \},\ldots
,X_{n}^{(j)}1\{Q_{n}<\tau \} )^{\prime }.
\end{split}
\end{align}
Then,
 for each $\tau \in \mathbb{T}$ and for each $j=1,\ldots,M$, we have
 $\left\Vert \mathbf{X}^{(j)}(\tau
)\right\Vert _{n} = \left\Vert X^{\left( j\right) }\right\Vert
_{n}$ 
and 
$\left\Vert \mathbf{X}^{(M+j)}(\tau
)\right\Vert _{n} = \left\Vert X^{\left( j\right) } (\tau) \right\Vert
_{n}$. 
Using this notation, we rewrite the $\ell_1$ penalty as 
\begin{align*}
\lambda \left\vert \mathbf{D}(\tau )\alpha \right\vert_{1}
&= \lambda \sum_{j=1}^{2M} \left\Vert \mathbf{X}^{(j)}(\tau
)\right\Vert _{n} \left\vert\alpha^{(j)}\right\vert \\
&= \lambda \sum_{j=1}^{M} \left[ \left\Vert X^{\left( j\right) }\right\Vert
_{n} \left\vert \alpha^{(j)} \right\vert +  \left\Vert X^{\left( j\right) }(\tau) \right\Vert
_{n} \left\vert \alpha^{(M+j)} \right\vert \right].
\end{align*}
Therefore, for each fixed $\tau \in \mathbb{T}$, $\widehat{\alpha }(\tau )$ is the weighted Lasso
that uses a data-dependent $\ell_1$ penalty to balance covariates adequately.

We now estimate $%
\tau _{0}$ by
\begin{equation}\label{tau-max}
\widehat{\tau }:=\text{argmin}_{\tau \in \mathbb{T}\subset \mathbb{R}%
}\left\{ S_{n}(\widehat{\alpha }(\tau ),\tau )+\lambda \left\vert \mathbf{D}%
(\tau )\widehat{\alpha }(\tau )\right\vert _{1}\right\}.
\end{equation}%
In fact, for any finite $n,$ $\widehat{\tau }$ is given by an
interval and we simply define the maximum of the interval as our estimator.
If we wrote the model using $1\left\{ Q_{i}>\tau \right\} ,$ then the
convention would be the minimum of the interval being the estimator. Then the
estimator of $\alpha _{0}$ is defined as $\widehat{\alpha }:=\widehat{\alpha
}(\widehat{\tau })$.
In fact, our proposed estimator of $(\alpha, \tau)$ can be viewed as the one-step minimizer such that:
\begin{align}\label{joint-max}
(\widehat{\alpha }, \widehat{\tau }) :=\text{argmin}_{\alpha \in \mathcal{A} \subset \mathbb{R}^{2M}, \tau \in \mathbb{T}\subset \mathbb{R}}\left\{
S_{n}(\alpha ,\tau )+\lambda \left\vert \mathbf{D}(\tau )\alpha \right\vert_{1}\right\}.
\end{align}%

It is worth noting that we penalize $\beta_0$ and $\delta_0$ in \eqref{joint-max},
where $\delta_0$ is the change of regression coefficients between two regimes. The model in \eqref{model} can be written as
\begin{align}\label{model2}
\begin{split}
Y_{i}&= X_{i}^{\prime }\beta _{0}+U_{i},\ \ \ \text{if  $Q_{i} \geq \tau_{0}$}, \\
Y_{i}&= X_{i}^{\prime }\beta _{1}+U_{i},\ \ \ \text{if  $Q_{i}<\tau_{0}$},
\end{split}
\end{align}%
where $\beta_1 \equiv \beta_0 + \delta_0$.
In view of \eqref{model2}, alternatively, one might penalize $\beta_0$ and $\beta_1$ instead of $\beta_0$ and $\delta_0$. We opted to penalize $\delta_0$ in this paper since the case of $\delta_0 = 0$  corresponds to the linear model.
If  $\widehat{\delta} = 0$,  
then this case amounts to selecting the linear model.

\section{Empirical Illustration}\label{sec:emp}

In this section, we apply the proposed Lasso method to growth regression models in economics. The neoclassical growth model predicts that economic growth rates converge in the long run. This theory has been tested empirically by looking at the negative relationship between the long-run growth rate and the initial GDP given other covariates (see \citet{barrosala} and \citet{durlauf2005growth} for literature reviews). Although empirical results confirmed the negative relationship between the growth rate and the initial GDP, there has been some criticism that the results depend heavily  on the selection of covariates. Recently, \citet{belloni2011high}  show that the Lasso estimation can help select the covariates in the \emph{linear} growth regression model and that the Lasso estimation results reconfirm the negative relationship between the long-run growth rate and the initial GDP.

We consider the growth regression model with a possible threshold. \citet{Durlauf:Johnson:95} provide the theoretical background of the existence of multiple steady states and estimate the model with two possible threshold variables. They check the robustness by adding other available covariates in the model, but it is not still free from the criticism of the \emph{ad hoc} variable selection. Our proposed Lasso method might be a good alternative in this situation.
Furthermore, as we will show later, our method works well even if there is no threshold effect in the model. Therefore, one might expect more robust results from our approach.

The regression model we consider has the following form: 
\begin{equation}\label{eq:emp01}
\mathit{gr}_{i}=\beta _{0}+\beta _{1}\mathit{lgdp60}_{i}+X_{i}^{\prime
}\beta _{2}+1\{Q_{i}<\tau \}\left( \delta _{0}+\delta _{1}\mathit{lgdp60}%
_{i}+X_{i}^{\prime }\delta _{2}\right) +\varepsilon _{i},
\end{equation}%
where $\mathit{gr}_{i}$ is the annualized GDP growth rate of country $i$
from 1960 to 1985, $\mathit{lgdp60}_{i}$ is the log GDP in 1960, and $Q_{i}$
is a possible threshold variable for which we use the initial GDP or the
adult literacy rate in 1960 following \citet{Durlauf:Johnson:95}. Finally, $%
X_{i}$ is a vector of additional covariates related to education, market
efficiency, political stability, market openness, and demographic
characteristics.
In addition, $X_i$ contains cross product terms between $\mathit{lgdp60}_i$ and education variables.
 Table \ref{tb:listVar} gives the list of all covariates
used and the description of each variable. We include as many covariates as
possible, which might mitigate the potential omitted variable bias. The data
set mostly comes from \citet{Barro:Lee:94}, and the additional adult
literacy rate is from \citet{Durlauf:Johnson:95}. Because of missing
observations, we have 80 observations with 46 covariates (including a
constant term) when $Q_{i}$ is the initial GDP ($n=80$ and $M=46$), and 70
observations with 47 covariates when $Q_{i}$ is the literacy rate ($n=70$
and $M=47$). It is worthwhile to note that the number of covariates in the
threshold models is bigger than the number of observations ($2M>n$ in our
notation). Thus, we cannot adopt the standard least squares method to
estimate the threshold regression model.

Table \ref{tb:resultM1} summarizes the model selection
and estimation results when $Q_{i}$ is the initial GDP.
In  Appendix \ref{sec:mc:extra} (see
Table \ref{tb:resultM2}), we report additional 
empirical results with  $Q_{i}$ being the literacy rate.
To compare different model specifications, we also
estimate a linear model, i.e.\ all $\delta$'s are zeros in \eqref{eq:emp01}, by the standard Lasso estimation.  
In each case, the regularization parameter $\lambda$ is chosen by the `leave-one-out'  cross validation method.
For the range $\mathbb{T}$ of the  threshold parameter, we consider an interval between  the 10\% and 90\% sample quantiles for each  threshold variable.
 
Main empirical findings are as follows. First, the marginal effect of 
$\mathit{lgdp60}_i$, which is given by%
\begin{equation*}
\frac{\partial \mathit{gr}_i}{\partial \mathit{lgdp60}_i}=\beta _{1}+ 
\mathit{educ}_i' \tilde{\beta}
_{2} +1\{Q_i<\gamma \}(\delta _{1}+\mathit{educ}_i' \tilde{\delta} _{2}),
\end{equation*}%
where $\mathit{educ}_i$ is a vector of education variables
and 
$\tilde{\beta}_{2}$ and $\tilde{\delta} _{2}$ are sub-vectors of  $\beta_{2}$ and $\delta_{2}$
corresponding to $\mathit{educ}_i$, is estimated to be negative for all the observed values of $\mathit{educ}_i$.
This confirms the theory of the neoclassical growth model. 
Second, some non-zero coefficients of interaction terms between \textit{%
lgdp60} and various education variables show the existence of threshold
effects in both threshold model specifications. 
This result implies that the  growth convergence rates can vary according to different levels of the initial GDP or the adult literacy rate in 1960. 
Specifically, in both threshold models, we have $\delta _{1}=0$, but some $%
\delta _{2}$'s are not zeros. 
Thus, conditional on other covariates, there exist different technological diffusion effects according to the threshold point. 
For example, a developing country (lower $Q$) with a higher education level will converge
faster perhaps by absorbing advanced technology more easily and more quickly. 
%Our result shows (especially, the \emph{lr} one) this effect is clear among underdeveloped countries, i.e.\ lower $Q_i$.
Finally, the Lasso with the threshold model specification selects a more
parsimonious model than that with the linear specification even though the
former doubles the number of potential covariates.

\section{Prediction Consistency of the Lasso Estimator}\label{sec:consist}

In this section, we consider the prediction consistency of the Lasso estimator.
We make the following assumptions. 

\begin{assm}\label{assumption-main-1}
(i) For the parameter space $\mathcal{A}$ for $\alpha_0$,  
any $\alpha \equiv (\alpha_1,\ldots,\alpha_{2M}) \in \mathcal{A}
\subset \mathbb{R}^{2M}$, including $\alpha_0$,  satisfies 
$\max_{j=1,\ldots,2M} \left\vert \alpha_j \right\vert \leq  C_1$  
for some constant $C_1 > 0$.
In addition, $\tau_0 \in \mathbb{T}\equiv [t_{0},t_{1}]$ that 
 satisfies $\min_{i=1,\ldots,n} Q_i < t_0  < t_1 < \max_{i=1,\ldots,n} Q_i$.
(ii) There exist universal constants $C_2 > 0$ and $C_3 > 0$ such that 
 $\left\Vert \mathbf{X}^{(j)}(\tau)\right\Vert _{n} \leq C_2$
uniformly in $j$ and $\tau \in \mathbb{T}$, 
and  $\left\Vert \mathbf{X}^{(j)}(t_0)\right\Vert _{n} \geq C_3$
uniformly in $j$, where   $j=1,...,2M$.
(iii) There is no $i\neq j$ such that $Q_{i}=Q_{j}.$
\end{assm}

Assumption \ref{assumption-main-1}(i) imposes the boundedness for each component of the parameter vector. The first part of 
Assumption \ref{assumption-main-1}(i) implies  that $\left\vert \alpha \right\vert_1 \leq 2 C_1 M$ for any $\alpha \in \mathcal{A}$, which seems
to be weak, since the sparsity assumption implies that  
$\left\vert \alpha_0 \right\vert_1$ is much smaller than $C_1 M$. 
Furthermore, in the literature on change-point and threshold models, it is common to assume
that the parameter space  is compact. 
For example, see \citet{Seijo:Sen:11a,Seijo:Sen:11b}.

The Lasso estimator in \eqref{joint-max} 
can be computed without knowing the value of $C_1$,
but  $\mathbb{T}\equiv [t_{0},t_{1}]$
has to be specified. 
In practice, researchers tend to choose some strict subset of the range of observed values of the threshold variable. 
Assumption \ref{assumption-main-1}(ii) imposes that
each covariate is of the same magnitude uniformly over $\tau$.
In view of the assumption that $\min_{i=1,\ldots,n} Q_i < t_0$, it is not stringent
to assume that   $\left\Vert \mathbf{X}^{(j)}(t_0)\right\Vert _{n}$ is bounded
away from zero.

Assumption \ref{assumption-main-1}(iii) imposes that
there is no tie among $Q_i$'s.
This is a convenient assumption such that 
we can always transform  general $Q_i$ to $Q_i = i/n$ without loss of generality. 
This holds with probability one  for the random design case if $Q_i$ is continuously distributed.

Define
\begin{align*}%\label{eq:r_n}
r_{n}:=\min_{1\leq j\leq M}\frac{\left\Vert X^{\left( j\right)
}(t_{0})\right\Vert _{n}^{2}}{\left\Vert X^{\left( j\right) }\right\Vert
_{n}^{2}},
\end{align*}
where $X^{\left( j\right) }$
and  $X^{\left( j\right) }(\tau )$ are  
defined in \eqref{x-notation-new}.
Assumption \ref{assumption-main-1}(ii) implies that $r_n$ is bounded away from zero. 
In particular, we have that $1 \geq r_n \geq C_3/C_2 > 0$.

Recall that
\begin{equation}
\left\Vert \widehat{f}-f_{0}\right\Vert _{n}
:=\left(\frac{1}{n}  \sum_{i=1}^n \left(  \widehat{f}(X_i,Q_i) -  f_{0}(X_i,Q_i) \right)^2\right)^{1/2}.
\end{equation}
where $\widehat{f}(x,q):=x^{\prime }\widehat{\beta }
+x^{\prime }\widehat{\delta }1\{q<\widehat{\tau }\}$ and $f_{0}(x,q):=x^{\prime }\beta _{0}+x^{\prime }\delta
_{0}1\{q<\tau _{0}\}$.
To establish theoretical results in the paper 
(in particular, oracle inequalities in Section \ref{sec:oracle}), 
let $(\widehat{\alpha },\widehat{\tau })$ be the Lasso estimator defined by \eqref{joint-max} with
\begin{align}\label{lambda-form}
\lambda =A \sigma \Big(\frac{\log 3M}{nr_{n}}\Big)^{1/2}
\end{align}%
for a constant $A>2\sqrt{2}/\mu$, where   $\mu \in (0,1)$ is a fixed constant.
We now present the first theoretical result of this paper.

\begin{thm}[Consistency of the Lasso]\label{theorem-main-1}
Let Assumption \ref{assumption-main-1} hold.  
Let $\mu$ be a constant such that $0 < \mu <1$, and let
$(\widehat{\alpha },\widehat{\tau })$ be the Lasso estimator defined by \eqref{joint-max} with $\lambda$ given by \eqref{lambda-form}.
Then, with probability at least $1-\left( 3M\right)^{1-A^{2}\mu ^{2}/8}$, 
we have
\begin{align*}
\left\Vert \widehat{f}-f_{0}\right\Vert _{n}
& \leq K_1 \sqrt{\lambda \mathcal{M}(\alpha _{0})},
\end{align*}
where  $K_1 \equiv \sqrt{ 2 C_{1} C_{2}(3+\mu)}  > 0$.
\end{thm}

The nonasymptotic upper bound on the prediction risk in Theorem \ref{theorem-main-1} can be  translated easily into asymptotic convergence.  Theorem \ref{theorem-main-1} implies the consistency of the Lasso, provided
that $n \rightarrow \infty$, $M \rightarrow \infty$, and 
$\lambda \mathcal{M}(\alpha _{0})\rightarrow 0$.
Recall that  $\mathcal{M}(\alpha _{0})$ represents the sparsity of the model %
\eqref{model1}. 
Note that in view of \eqref{lambda-form},
the  condition $\lambda \mathcal{M}(\alpha _{0})\rightarrow 0$ requires that
$\mathcal{M}(\alpha _{0}) = o ( \sqrt{nr_{n} / \log 3M } )$.
This implies that $\mathcal{M}(\alpha _{0})$ can increase with $n$.

\begin{rem}
Note that the prediction error increases as $A$ or $\mu$ increases; however,
the probability of correct recovery increases if  $A$ or $\mu$ increases.
Therefore, there exists a tradeoff between  the prediction error and the probability
of correct recovery. 
\end{rem}

\section{Oracle Inequalities}\label{sec:oracle}

In this section, we establish finite sample sparsity oracle inequalities in terms of both the prediction loss
and the $\ell_1$ estimation loss for unknown parameters.
First of all, we make the following assumption.

\begin{assm}[Uniform Restricted Eigenvalue (URE) $(s,c_{0},\mathbb{S})$]
\label{re-assump} For some integer $s$ such that $1\leq s\leq 2M$, a
positive number $c_{0}$, and 
some set $\mathbb{S} \subset \mathbb{R}$, the following condition holds:
\begin{equation*}
\kappa (s,c_{0},\mathbb{S}):=\min_{\tau \in \mathbb{S}}\min_{\substack{ J_{0}\subseteq
\{1,\ldots ,2M\},  \\ |J_{0}|\leq s}}\min_{\substack{ \gamma \neq 0,  \\ %
\left\vert \gamma _{J_{0}^{c}}\right\vert _{1}\leq c_{0}\left\vert \gamma
_{J_{0}}\right\vert _{1}}}\frac{|\mathbf{X}(\tau )\gamma |_{2}}{\sqrt{n}%
|\gamma _{J_{0}}|_{2}}>0.
\end{equation*}
\end{assm}

If $\tau _{0}$ were known, then
Assumption \ref{re-assump} is just a restatement of the restricted eigenvalue
assumption of \citet{Bickel-et-al:09} with $\mathbb{S} = \{ \tau _{0} \}$.
\citet{Bickel-et-al:09} provide sufficient conditions for the restricted eigenvalue condition.
In addition, \citet{vdGeer:Buhlmann:09} show the relations between the restricted eigenvalue condition and other conditions on the design matrix, and 
\citet{RWY:10} prove that restricted eigenvalue conditions hold with high probability for a large class of correlated Gaussian design matrices.

%It is straightforward to provide sufficient conditions for Assumption \ref{re-assump}  by extending those in \citet{Bickel-et-al:09}. For example, the URE $(s,c_{0},\mathbb{S})$ condition is satisfied if one of Assumptions 1--5 in Section 4 of \citet{Bickel-et-al:09} is satisfied uniformly over $\tau \in \mathbb{S}$.

If $\tau _{0}$ is unknown as in our setup, it seems necessary to assume that the restricted eigenvalue
condition holds uniformly over $\tau$. We consider separately two cases depending on whether $\delta_0 = 0$ or not.
On the one hand, if $\delta_0 = 0$ so that $\tau_0$ is not identifiable, then we need to assume that the URE condition holds uniformly on the whole parameter space, $\mathbb{T}$.
On the other hand, if $\delta_0 \neq 0$ so that $\tau_0$ is identifiable, then it suffices to impose the URE condition holds uniformly on a neighborhood of $\tau_0$.
In Appendix \ref{sec:suff-ure}, we provide two types of sufficient conditions for Assumption \ref{re-assump}.
One type is based on modifications of Assumption 2 of \citet{Bickel-et-al:09}
and the other type is in the same spirit as \citet[][Section 10.1]{vdGeer:Buhlmann:09}. 
Using the second type of results, we verify primitive sufficient conditions for the URE condition in the context of our simulation designs.
See Appendix \ref{sec:suff-ure} for details.

%This is distinguished
%from $X\left( \tau \right) ,$ which denotes the $\left( n\times M\right) $
%matrix whose $i$-th row is $X_{i}^{\prime }1\{Q_{i}<\tau \}.$

The URE condition is useful for us to improve the result in Theorem \ref{theorem-main-1}. Recall that in Theorem \ref{theorem-main-1},   the prediction risk  is bounded by a factor of $\sqrt{\lambda \mathcal{M}(\alpha _{0})}$.
This bound is too large to give  us an oracle inequality. We will show below that 
we can establish non-asymptotic oracle inequalities for the prediction risk
as well as the $\ell_1$ estimation loss, thanks to the URE condition.

The strength of the proposed Lasso method is that it is not necessary to know or pretest whether $\delta_0 = 0$ or not.
It is worth noting that we do not have to know whether there exists a threshold in 
the model
in order to establish oracle
inequalities for the prediction risk and the $\ell_1$ estimation loss for $\alpha_0$, although we divide our theoretical results into two cases below. This implies that we can make prediction 
and estimate $\alpha_0$ precisely
without knowing the presence of
threshold effect or without pretesting for it.

\subsection{Case I. No Threshold}\label{sec:oracle1}

We first consider the case that $\delta_0 = 0$. In other words, we estimate a threshold model via the Lasso method, but the true model is
simply a linear model $Y_{i}=X_{i}^{\prime }\beta _{0}+U_{i}$.
This is an important case to consider in applications,  because 
one may not be sure not only about covariates  selection but also about the existence of the threshold in the model.

Let $\phi _{\max }$ denote the supremum (over $\tau \in \mathbb{T}$) 
of the largest eigenvalue of 
$\mathbf{X}(\tau)^{\prime }\mathbf{X}(\tau )/n$. 
Then by definition, the largest eigenvalue of $\mathbf{X}(\tau
)^{\prime }\mathbf{X}(\tau )/n$ is bounded uniformly in $\tau \in \mathbb{T}$
by $\phi _{\max }$. The following theorem gives oracle inequalities 
for the first case.

\begin{thm}\label{main-thm-case1}
Suppose that $\delta _{0}=0$. Let  
Assumptions \ref{assumption-main-1} and \ref{re-assump} hold with $\kappa
=\kappa (s,\frac{1+\mu}{1-\mu },\mathbb{T})$ for $0 < \mu <1$, and $\mathcal{M}%
(\alpha _{0})\leq s\leq M$.
 %Let $U_{i}$ follow $N(0,\sigma ^{2})\ $and
Let $(\widehat{\alpha },\widehat{\tau })$ be the Lasso estimator defined by \eqref{joint-max} with $\lambda$ given by \eqref{lambda-form}.
%\begin{equation*}
%\lambda =A\sigma \Big(\frac{\log 3M}{nr_{n}}\Big)^{1/2}
%\end{equation*}%
%and $A>2\sqrt{2}/\mu .$
Then, with probability at least $1-\left( 3M\right)
^{1-A^{2}\mu ^{2}/8},$ we have
\begin{align*}
\left\Vert \widehat{f}-f_{0}\right\Vert _{n} 
&\leq  K_{2} \frac{\sigma }{\kappa }\left( \frac{\log 3M}{nr_{n}} s\right)
^{1/2}, \\
\left\vert \widehat{\alpha }-\alpha _{0}\right\vert _{1} &\leq 
K_{2} \frac{\sigma }{ \kappa ^{2}}
\left(\frac{\log 3M}{nr_{n}}\right)^{1/2}s, \\
\mathcal{M}(\widehat{\alpha })&\leq 
K_{2} \frac{\phi _{\max } }{\kappa ^{2}} s
\end{align*}%
for some universal constant $K_2 > 0$.
\end{thm}

To appreciate the usefulness of the  inequalities derived above, it is worth comparing inequalities in Theorem \ref{main-thm-case1} with those in Theorem 7.2 of \citet{Bickel-et-al:09}.
The latter corresponds to the case that $\delta_0 = 0$ is known \emph{a priori}
and  $\lambda = 2 A\sigma ({\log M}/{n})^{1/2}$ using our notation.
If we compare
Theorem \ref{main-thm-case1} with Theorem 7.2 of \citet{Bickel-et-al:09},
we can see that
the Lasso estimator in \eqref{joint-max}
gives qualitatively the same oracle inequalities as
the Lasso estimator in the linear model, even though
our model is much more overparametrized in that $\delta$ and $\tau$ are added to $\beta$ as parameters to estimate.

Also, as in \citet{Bickel-et-al:09}, there is no requirement on 
$\alpha_0$ such that the minimum value of nonzero components of $\alpha_0$ is bounded away from zero. In other words, there is no need to assume  
the minimum strength of the signals.
Furthermore, $\alpha_0$ is well estimated here even if $\tau_0$ is not identifiable at all.
Finally, note that
the  value of the constant $K_2$ is given  
 in the proof of Theorem \ref{main-thm-case1}
and that 
 Theorem 2 can be translated easily into asymptotic oracle results as well, since both $\kappa$ and $r_n$ are bounded away from zero by the URE condition and Assumption \ref{assumption-main-1}, respectively.

\subsection{Case II. Fixed Threshold}\label{sec:oracle3}

This subsection explores the case where the threshold effect is well-identified
and discontinuous. We begin with the following additional assumptions to reflect
this.

\begin{assm}[Identifiability under Sparsity and Discontinuity of Regression]
\label{A-discontinuity}
For a given $s\geq \mathcal{M}\left( \alpha
_{0}\right) ,$ and for any $\eta$ and $\tau $ such that $\left\vert \tau -\tau
_{0}\right\vert > \eta \geq  \min_{i}\left\vert Q_{i}-\tau_0\right\vert $
and $\alpha \in \left\{ \mathcal{\alpha }:\mathcal{M}\left( \alpha \right)
\leq s \right\} $, there exists a  constant $c>0$ such that%
\begin{equation*}
\left\Vert f_{\left( \alpha ,\tau \right) }-f_0 \right\Vert _{n}^{2}  > c\eta.
\end{equation*}
\end{assm}

Assumption \ref{A-discontinuity} implies, among other things, that
for some $s\geq \mathcal{M}\left( \alpha _{0}\right) ,\ $and for
any $\alpha \in \left\{ \mathcal{\alpha }:\mathcal{M}\left( \alpha \right)
\leq s\right\} $ and $\tau $ such that $\left( \alpha ,\tau \right) \neq
\left( \alpha _{0},\tau _{0}\right) $,
\begin{align}\label{A-Id}
\left\Vert f_{\left( \alpha ,\tau \right) }-f_0 \right\Vert _{n}\neq 0.
\end{align}
This condition can be regarded as identifiability of $\tau_0$.
If $\tau _{0}$ were known, then a sufficient condition for the identifiability under the sparsity would be that $URE\left( s,c_{0}, \{\tau_0\} \right)$ holds for some $c_{0}\geq 1$.
 Thus, the main point in \eqref{A-Id} is that there is no
sparse representation that is equivalent to $f_{0}$ when the sample is
split by $\tau \neq \tau _{0}.$
In fact,
Assumption \ref{A-discontinuity} is stronger than just the identifiability of $\tau_0$ as it specifies the rate of
deviation in $f$ as $\tau $ moves away from $\tau _{0},$ which in turn dictates the bound for the estimation error of $\widehat \tau$.
We provide further discussions on Assumption \ref{A-discontinuity} in Appendix \ref{rem-assump-discontinuity}.

\begin{rem}
The restriction $\eta \geq \min_{i}\left\vert Q_{i}-\tau_0\right\vert $ in Assumption \ref{A-discontinuity}
is necessary since we consider the fixed design for both $X_i$ and $Q_i$. Throughout this section, we implicitly assume that
the sample size $n$ is large enough such that $\min_{i}\left\vert Q_{i}-\tau_0\right\vert$ is very small, implying that
the restriction $\eta \geq \min_{i}\left\vert Q_{i}-\tau_0\right\vert $
 never binds in any of inequalities below.
This is typically true for the random design case if $Q_i$ is continuously distributed.
\end{rem}

\begin{assm}[Smoothness of Design]
\label{A-smoothness}
For any $\eta >0,$ there exists a constant $C<\infty $ such that%
\begin{equation*}
\sup_{j}\sup_{\left\vert \tau -\tau _{0}\right\vert <\eta } \frac{1%
}{n}\sum_{i=1}^{n}\left\vert X_{i}^{\left( j\right) }\right\vert ^{2}\left\vert
1\left( Q_{i}<\tau _{0}\right) -1\left( Q_{i}<\tau \right)
\right\vert \leq C\eta .
\end{equation*}
\end{assm}

Assumption \ref{A-smoothness} has been assumed
in the classical setup with a fixed number of stochastic
regressors to
exclude cases like $Q_{i}$ has a point mass at $\tau _{0}\ $or $\mathbb{E}\left( X_{i}|Q_{i}=\tau _{0}\right) $ is unbounded.
In our setup, Assumption \ref{A-smoothness} amounts to a deterministic version of
some smoothness assumption for the distribution of the threshold variable $%
Q_{i}$.
When $(X_i,Q_i)$ is a random vector, it is satisfied
under the standard assumption that $Q_{i}$ is continuously distributed and $%
\mathbb{E}( \vert X_{i}^{\left( j\right) }\vert ^{2}|Q_{i}=\tau )$ is continuous and bounded in a neighborhood of $%
\tau _{0}$ for each $j$.

To simplify notation, in the following theorem, we assume without loss of generality that $Q_{i}=i/n$. Then $\mathbb{T} = [t_0,t_1] \subset (0,1)$.
%For some constant $\eta > 0$, define  $h_{n}\left( \eta \right) :=\left( \left( 2n\eta %\right) ^{-1}\sum_{i=%
%\left[ n\left( \tau _{0}-\eta \right) \right] }^{\left[ n\left( \tau
%_{0}+\eta \right) \right] }\left( X_{i}^{\prime }\delta _{0}\right)
%^{2}\right) ^{1/2}$, where $[\cdot]$ denotes an integer part of any real number. 
In addition, let $\eta_0 = \max \left\{ n^{-1}, K_1 \sqrt{\lambda \mathcal{M}(\alpha_0)} \right\}$, where $K_1$ is the same constant in Theorem \ref{theorem-main-1}.

\begin{assm}[Well-defined Second Moments]
\label{h-regular}
For any  $\eta$ such that $1/n \leq \eta \leq \eta_0$, 
%$\tau_0 + \eta < t_1$, and $\tau_0 - \eta > t_0$, 
 $h_n^2(\eta)$ is bounded, where 
\begin{align*}
h_{n}^2 \left( \eta \right)
 := \frac{1}{  2n\eta } \sum_{i=%
\max\{1, \left[ n\left( \tau _{0}-\eta \right) \right]\} }^{\min\{ \left[ n\left( \tau
_{0}+\eta \right) \right], n\} }\left( X_{i}^{\prime }\delta _{0}\right)
^{2}
\end{align*} 
and $[\cdot]$ denotes an integer part of any real number. 
\end{assm}

Assumption \ref{h-regular} assumes that $h_n^2(\eta)$ is well defined
for  any  $\eta$ such that $1/n \leq \eta \leq \eta_0$. 
Assumption \ref{h-regular} amounts to 
some weak regularity condition on the second moments of the fixed design.  
Assumption \ref{A-discontinuity} implies that $\delta_0 \neq 0$
and that 
 $h_n^2(\eta)$ is  bounded away from zero.
Hence,  Assumptions \ref{A-discontinuity} and  \ref{h-regular} imply that 
 $h_n^2(\eta)$ is bounded and bounded away from zero. 

To present the theorem below, it is necessary to make one additional
technical assumption (see Assumption \ref{tech-cond}  in  Appendix 
\ref{sec:proofs:oracle3}). We opted not to show Assumption \ref{tech-cond} here, since we believe this is just a sufficient condition that does not add much to our understanding of the main result. However, we would like to point out that  
Assumption \ref{tech-cond} can hold  for all sufficiently large $n$, provided that $s \lambda  \left\vert \delta _{0}\right\vert_{1} \rightarrow 0$, as $n \rightarrow 0$.
See Remark \ref{tech-remark} in  Appendix 
\ref{sec:proofs:oracle3} for details.

We now give the main result of this section.

\begin{thm}\label{main-text-thm-fixed-threshold} 
Suppose that Assumptions \ref{assumption-main-1} and \ref{re-assump}
hold with $\mathbb{S} = \left\{ \left\vert \tau -\tau _{0}\right\vert \leq
\eta_0 \right\} $, $\kappa =\kappa (s,\frac{2+\mu}{1-\mu },\mathbb{S})$
for $0<\mu <1,$ and $\mathcal{M}(\alpha _{0})\leq s\leq M$. Furthermore,  Assumptions \ref{A-discontinuity}, \ref{A-smoothness}, and \ref{h-regular}  hold
and let $n$ be large enough so that Assumption \ref{tech-cond}  
in  Appendix 
\ref{sec:proofs:oracle3}
holds. 
Let $(\widehat{\alpha },\widehat{\tau })$ be the Lasso estimator defined by \eqref{joint-max} with $\lambda$ given by \eqref{lambda-form}.
%\begin{equation*}
%\lambda =A\sigma \Big(\frac{\log 3M}{nr_{n}}\Big)^{1/2}
%\end{equation*}%
%and $A>2\sqrt{2}/\mu $.
Then, 
%there exists a sequence of constants $\eta
%_{1},...,\eta _{m^{\ast }}$ for some finite $m^{\ast }$ such that
% $h_{n}\left( \eta_j  \right) > 0$ for each $j=1,\ldots,m^{\ast}$,
%and 
with
probability at least $1-\left( 3M\right) ^{1-A^{2}\mu
^{2}/8}-
C_4 \left( 3M\right) ^{-C_5/r_{n}}$ for some positive constants $C_4$ and $C_5$, we have
\begin{align*}
\left\Vert \widehat{f}-f_{0}\right\Vert _{n}& \leq 
K_3 \frac{\sigma }{\kappa }\left( \frac{\log 3M}{nr_{n}}s\right) ^{1/2}, \\
\left\vert \widehat{\alpha }-\alpha _{0}\right\vert _{1}& \leq 
K_3 \frac{\sigma }{\kappa^2 }
\left( {\frac{\log 3M}{nr_{n}}}\right) ^{1/2} s, \\
\left\vert \hat{\tau}-\tau _{0}\right\vert &\leq 
K_3 \frac{\sigma^2 }{\kappa^2 }
\frac{\log 3M}{nr_{n}}s, \\
\mathcal{M}\left( \hat{\alpha}\right) &\leq 
K_3 \frac{\phi _{\max }}{\kappa ^{2}} s
\end{align*}
for some universal constant $K_3 > 0$.
\end{thm}

Theorem \ref{main-text-thm-fixed-threshold} gives the same  inequalities (up to constants)
as those in Theorem \ref{main-thm-case1}
 for the prediction risk as well as the $\ell_1$ estimation loss for $\alpha_0$.
It is important to note that
$\left\vert \hat{\tau}-\tau _{0}\right\vert$ is bounded by a constant times
$s \log 3M/(nr_{n})$, whereas  $\left\vert \widehat{\alpha }-\alpha _{0}\right\vert _{1}$ is bounded by a constant times $s [\log 3M/({nr_{n}})]^{1/2}$.
This can be viewed as a nonasymptotic version of the super-consistency of $\widehat \tau$ to $\tau_0$.
As noted at the end of Section \ref{sec:oracle1}, since both $\kappa$ and $r_n$ are bounded away from zero by the URE condition and Assumption \ref{assumption-main-1}, respectively, 
Theorem \ref{main-text-thm-fixed-threshold} 
implies asymptotic rate results immediately. 
The values of constants $C_4$, $C_5$ and  $K_3$ are given  
 in the proof of Theorem \ref{main-text-thm-fixed-threshold}.

The main contribution of this section  is that we have extended the well-known super-consistency result of $\widehat \tau$ when $M < n$ \citep[see, e.g.][]{chan:93,Seijo:Sen:11a,Seijo:Sen:11b}
to the high-dimensional setup ($M \gg n$).
In both cases, the main reason we achieve the super-consistency for the threshold parameter is that the least squares objective function behaves locally linearly around the true threshold parameter value rather than locally quadratically, as in regular estimation problems.  
An interesting remaining research question is to investigate whether it would be possible
to obtain the super-consistency result of $\widehat \tau$
under a weaker condition, perhaps without a restricted eigenvalue condition. 

\section{Monte Carlo Experiments}\label{sec:MC}

In this section we conduct some simulation studies and check the properties of the proposed Lasso estimator. The baseline model is \eqref{model}, where $X_i$ is an $M$-dimensional vector generated from $N(0,I)$, $Q_i$ is a scalar generated from the uniform distribution on the interval of $(0,1)$, and the error term $U_i$ is generated from $N(0,0.5^2)$. The threshold parameter is set to $\tau_0=0.3,0.4,$ and $0.5$ depending on the simulation design, and the coefficients are set to  $\beta_0=(1,0,1,0,\ldots,0)$, and $\delta_0=c\cdot(0,-1,1,0,\ldots,0)$ where $c=0$ or $1$. Note that there is no threshold effect when $c=0$. The number of observations is set to $n=200$. Finally, the dimension of $X_i$ in each design is set to $M=50,100, 200$ and $400$, so that the total number of regressors are 100, 200, 400 and 800, respectively. The range of $\tau$ is $\mathbb{T} = [0.15,0.85]$.

We can estimate the parameters by the standard LASSO/LARS algorithm of \citet{efron2004least} without much modification. Given a regularization parameter value $\lambda$, we estimate the model for each grid point of $\tau$ that spans over 71 equi-spaced points on $\mathbb{T}$. This procedure can be conducted by using the standard linear Lasso. Next, we plug-in the estimated parameter $\widehat{\alpha}(\tau):=\left(\widehat{\beta}(\tau)', \widehat{\delta}(\tau)'\right)'$ for each $\tau$ into the objective function and choose $\widehat{\tau}$ by \eqref{tau-max}. Finally, $\widehat{\alpha}$ is estimated by $\widehat{\alpha}(\widehat{\tau})$.
%\begin{align}
%\widehat{\tau}:=\arg \min_ { \tau \in \mathbb{T} \subset \mathbb{R} } \left\{ \widehat{S}\left( \widehat{ \alpha}\left(\tau \right), \tau \right)  + \lambda \left|D\left(\tau\right)^{1/2}\widehat{\alpha}\left(\tau\right) \right|_1\right\}
%\end{align}
%and $\widehat{\alpha}:=\widehat{\alpha}(\widehat{\tau})$.
The regularization parameter $\lambda$ is chosen by \eqref{lambda-form}
%\begin{align}
%\lambda := A \times \sigma \sqrt{\frac{\log{\left(3M\right)}}{n r_n} }
%\end{align}
where $\sigma=0.5$ is assumed to be known. For the constant $A$, we use four different values: $A=2.8, 3.2, 3.6,$ and $4.0$.

Table \ref{tb:M50} and Figures \ref{fg:indepX1}--\ref{fg:indepX2} summarize these simulation results. To compare the performance of the Lasso estimator, we also report the estimation results of the least squares estimation (Least Squares) available only when $M=50$ and two oracle models (Oracle 1 and Oracle 2, respectively). Oracle 1 assumes that the regressors with non-zero coefficients are known. In addition to that, Oracle 2 assumes that the true threshold parameter $\tau_0$ is known. Thus, when $c\neq0$, Oracle 1 estimates $(\beta^{(1)}, \beta^{(3)},\delta^{(2)},\delta^{(3)})$ and $\tau$ using the least squares estimation while Oracle 2 estimates only $(\beta^{(1)}, \beta^{(3)},\delta^{(2)},\delta^{(3)})$. When $c=0$, both Oracle 1 and Oracle 2 estimate only $(\beta^{(1)}, \beta^{(3)})$. All results are based on 400 replications of each sample.

The reported mean-squared prediction error ($PE$) for each sample is calculated numerically as follows. For each sample $s$, we have the estimates $\widehat{\beta}_{s}$, $\widehat{\delta}_{s}$, and $\widehat{\tau}_{s}$. Given these estimates, we generate a new data $\{Y_j,X_j,Q_j\}$ of 400 observations and calculate the prediction error as

\begin{align}
\widehat{PE}_{s} = \frac{1} {400} \sum_{j=1}^{400} \left(f_0(x_j,q_j) - \widehat{f}(x_j,q_j) \right)^2.
\end{align}
The mean, median, and standard deviation of the prediction error are calculated from the 400 replications,  $\{\widehat{PE}_{s}\}_{s=1}^{400}$. We also report the mean of $\mathcal{M}(\widehat{\alpha})$ and $\ell_1$-errors for $\alpha$ and $\tau$. Table  \ref{tb:M50} reports the simulation results of $M=50$. For simulation designs with $M>50$, Least Squares is not available, and we summarize the same statistics only for the Lasso estimation in Figures \ref{fg:indepX1}--\ref{fg:indepX2}.

When $M=50$, across all designs, the proposed Lasso estimator performs better than Least Squares in terms of mean and median prediction errors, the mean of $\mathcal{M}(\widehat{\alpha})$, and the $\ell_1$-error for $\alpha$.
The performance of the Lasso estimator becomes much better when there is no threshold effect, i.e.\ $c=0$. 
This result confirms the robustness of the Lasso estimator for whether  or not there exists a threshold effect.
However, Least Squares performs better than the Lasso estimator in terms of estimation of $\tau_0$ when $c=1$,
although the difference here is much smaller than the differences in prediction errors and the $\ell_1$-error for $\alpha$.

From Figures \ref{fg:indepX1}--\ref{fg:indepX2}, we can reconfirm the robustness of the Lasso estimator when $M=100, 200$, and $400$. As predicted by the theory developed in previous sections, the prediction error and $\ell_1$ errors for $\alpha$ and $\tau$ increase slowly as $M$ increases. The graphs also show that the results are quite uniform across different regularization parameter values except $A=4.0$.

In  Appendix \ref{sec:mc:extra}, we report additional simulation results, while
allowing correlation between covariates.
Specifically, the $M$-dimensional vector $X_i$ is  generated from a multivariate normal $N(0,\Sigma)$ with $(\Sigma)_{i,j}=\rho^{|i-j|}$, where $(\Sigma)_{i,j}$ denotes the (i,j) element of the $M \times M$ covariance matrix $\Sigma$. 
All other random variables are the same as above. 
We obtained  very similar results as previous cases: Lasso outperforms Least Squares, and the prediction error, the mean of $\mathcal{M}(\widehat{\alpha})$, and  $\ell_1$-errors increase very slowly as $M$ increases. See further details in
 Appendix \ref{sec:mc:extra}, which also reports satisfactory simulation results
regarding
 frequencies of selecting true parameters when both $\rho=0$ and $\rho=0.3$.

In sum, the simulation results confirm the theoretical results developed earlier and show that the proposed Lasso estimator will be useful for the high-dimensional threshold regression model.

\section{Conclusions}\label{sec:concl}

We have considered a  high-dimensional regression model
with a possible change-point due to a covariate threshold
and have developed the Lasso method. We have derived nonasymptotic oracle inequalities
and have illustrated the usefulness of our proposed estimation method via simulations and a real-data application.

We conclude this paper by providing some areas of future research. 
First, it would be interesting to extend other penalized estimators (for example, the adaptive Lasso of \citet{Zou:06} and the smoothly clipped absolute deviation (SCAD) penalty of \citet{Fan:Li:01}) to our setup and to see whether we would be able to improve the performance of our estimation method.
Second, an extension to multiple change points is also an important research topic. 
There has been some advance to this direction, especially regarding key issues like computational cost and the determination of the number of change points (see, for example, \citet{Harchaoui:Levy-Leduc:10} and \citet{frick2013multiscale}). 
However, they are confined to a single regressor case, and the extension to a large number of regressors would be highly interesting. 
Finally, it would be also an interesting research topic to investigate the minimax lower bounds of the proposed estimator and its prediction risk as \citet{raskutti2011minimax, raskutti2012minimax} did in high-dimensional linear regression setups.

%\textbf{[Youngki, could you please write up discussions on minimax lower bounds and multiple change points here. A short paragraph would suffice. We can iterate this after your turn - SL]}

%\bibliographystyle{chicago}
%\bibliography{LeeSeoShin_13Mar2014}

\begin{table}[htbp]

\begin{center}
\caption{List of Variables}
\label{tb:listVar}
\small
%\normalsize
\begin{tabular}{p{1.2in}p{5in}}
\hline \hline
Variable Names & Description \\
\hline
\multicolumn{2}{l}{\underline{\emph{Dependent Variable}}}\\
$\textit{gr}$ & Annualized GDP growth rate in the period of 1960--85 \\
& \\
\multicolumn{2}{l}{\underline{\emph{Threshold Variables}}}\\
\textit{gdp60} & Real GDP per capita in 1960 (1985 price)\\
\textit{lr} & Adult literacy rate in 1960 \\
& \\
\multicolumn{2}{l}{\underline{\emph{Covariates}}}\\
\textit{lgdp60} &	Log GDP per capita in 1960 (1985 price)\\
\textit{lr} & Adult literacy rate in 1960 (only included when $Q=lr$)\\
$\textit{ls}_k$	& Log(Investment/Output) annualized over 1960-85; a proxy for the log physical savings rate\\
$\textit{lgr}_{pop}$ &Log population growth rate annualized over 1960--85\\
\textit{pyrm60} &	Log average years of primary schooling in the male population in 1960\\
\textit{pyrf60}	&Log average years of primary schooling in the female population in 1960\\
\textit{syrm60} &	Log average years of secondary schooling in the male population in 1960\\
\textit{syrf60}	&Log average years of secondary schooling in the female population in 1960\\
\textit{hyrm60} &	Log average years of higher schooling in the male population in 1960\\
\textit{hyrf60} &	Log average years of higher schooling in the female population in 1960\\
\textit{nom60} &	Percentage of no schooling in the male population in 1960\\
\textit{nof60}	&Percentage of no schooling in the female population in 1960\\
\textit{prim60}&	Percentage of primary schooling attained in the male population in 1960\\
\textit{prif60}&	Percentage of primary schooling attained in the female population in 1960\\
\textit{pricm60}	&Percentage of primary schooling complete in the male population in 1960\\
\textit{pricf60}&	Percentage of primary schooling complete in the female population in 1960\\
\textit{secm60}&	Percentage of secondary schooling attained in the male population in 1960\\
\textit{secf60}	&Percentage of secondary schooling attained in the female population in 1960\\
\textit{seccm60}&	Percentage of secondary schooling complete in the male population in 1960\\
\textit{seccf60}&	Percentage of secondary schooling complete in the female population in 1960\\
\textit{llife} &	Log of life expectancy at age 0 averaged over 1960--1985\\
\textit{lfert}&	Log of fertility rate (children per woman) averaged over 1960--1985\\
\textit{edu/gdp} &	Government expenditure on eduction per GDP averaged over 1960--85\\
\textit{gcon/gdp}&	Government consumption expenditure net of defence and education per GDP averaged over 1960--85\\
\textit{revol} & The number of revolutions per year over 1960--84\\
\textit{revcoup} &	The number of revolutions and coups per year over 1960--84\\
\textit{wardum}  &	Dummy for countries that participated in at least one external war over 1960--84\\
\textit{wartime}  &	The fraction of time over 1960-85 involved in external war\\
\textit{lbmp}	& Log(1+black market premium averaged over 1960--85)\\
\textit{tot}&	The term of trade shock\\
$\textit{lgdp60} \times \textit{`educ'}$	& {Product of two covariates (interaction of \textit{lgdp60} and education variables from \textit{pyrm60} to \textit{seccf60}); total 16 variables} \\
\hline
\end{tabular}
\end{center}
\end{table}

\begin{table}[htbp]
\small
\begin{center}
\caption{Model Selection and Estimation Results with $Q={gdp60}$}
\label{tb:resultM1}
\begin{tabular}{cccccc}
\\
\hline \hline
 && \multirow{2}{*}{Linear Model}  && \multicolumn{2}{c}{Threshold Model}\\
 && & &\multicolumn{2}{c}{$\widehat{\tau}=2898$}\\
\hline
 && & & $ \widehat{\beta}$ & $ \widehat{\delta}$ \\
\cline{5-6}
\textit{const.} 			&& -0.0923 	&& -0.0811 & - \\
\textit{lgdp60} 		&& -0.0153 	&& -0.0120 & - \\
$\textit{ls}_k$			&& 0.0033 	&& 0.0038 & - \\
$\textit{lgr}_{pop}$	&&0.0018	&& - & - \\
\textit{pyrf60}		&&0.0027	&& - & - \\
\textit{syrm60}		&&0.0157	&& - & - \\
\textit{hyrm60}		&&0.0122	&& 0.0130 & - \\
\textit{hyrf60}		&&-0.0389	&& - & -0.0807\\
\textit{nom60}		&& - 			&& - & $2.64 \times 10^{-5}$ \\
\textit{prim60}		&&-0.0004	&& -0.0001 & - \\
\textit{pricm60}		&&0.0006	&&$-1.73 \times 10^{-4}$ & $-0.35 \times 10^{-4}$\\
\textit{pricf60}		&&-0.0006	&& - & - \\
\textit{secf60}		&&0.0005	&& - & - \\
\textit{seccm60}		&&0.0010	&& - & 0.0014\\
\textit{llife}			&&0.0697	&& 0.0523& - \\
\textit{lfert}			&&-0.0136	&& -0.0047& - \\
\textit{edu/gdp}		&&-0.0189	&& - & - \\
\textit{gcon/gdp}		&&-0.0671	&& -0.0542& - \\
\textit{revol}			&&-0.0588	&& - & - \\
\textit{revcoup} 		&&0.0433	&& - & - \\
\textit{wardum}		&&-0.0043	&& - & -0.0022\\
\textit{wartime}		&&-0.0019	&& -0.0143 &-0.0023 \\
\textit{lbmp}			&&-0.0185	&& -0.0174 & -0.0015\\
\textit{tot}			&&0.0971	&& - & 0.0974\\
$\textit{lgdp60} \times\textit{pyrf60}$         && - &&  $-3.81\times 10^{-6}$ &  - \\
$\textit{lgdp60} \times\textit{syrm60} $ 	&& - && - & 0.0002\\
$\textit{lgdp60} \times\textit{hyrm60}$ 	&& - && - & 0.0050\\
 $\textit{lgdp60} \times\textit{hyrf60}$ 	&& - && -0.0003 & - \\
$\textit{lgdp60} \times\textit{nom60}$	&& - && - & $8.26\times 10^{-6}$\\
$\textit{lgdp60} \times\textit{prim60}$	&&$-6.02 \times 10^{-7}$   && - & - \\
$\textit{lgdp60} \times\textit{prif60}$		&&$-3.47 \times 10^{-6}$   && - & $-8.11 \times 10^{-6}$\\
$\textit{lgdp60} \times\textit{pricf60}$	&&$-8.46 \times 10^{-6}$  && - & - \\
$\textit{lgdp60} \times\textit{secm60}$	&& -0.0001  && - & - \\
$\textit{lgdp60} \times\textit{seccf60}$	&& -0.0002   && $-2.87 \times 10^{-6}$ & -\\
\hline
$\lambda$ & & $0.0004$  & & \multicolumn{2}{c}{$0.0034$}\\
$\mathcal{M}(\widehat{\alpha})$ &&28 && \multicolumn{2}{c}{26}\\
$\#\ \textit{of covariates}$ & &46 & & \multicolumn{2}{c}{92}\\
$\#\  \textit{of obsesrvations}$ && 80 && \multicolumn{2}{c}{80}\\
%$R^2$ && 0.85 && \multicolumn{2}{c}{0.80}\\
%$\widetilde{R}^2$ && 0.89 && \multicolumn{2}{c}{0.86}\\
%$adj.~R^2$ && 0.77 && \multicolumn{2}{c}{0.70}\\
\hline
\\
\multicolumn{6}{p{.8\textwidth}}{\footnotesize \emph{Note: }The regularization parameter $\lambda$ is chosen by the `leave-one-out'  cross validation method. $\mathcal{M}(\widehat{\alpha})$ denotes the number of covariates to be selected by the Lasso estimator, and `-' indicates that the regressor is not selected. Recall that $\widehat{\beta}$ is the coefficient when $Q \ge \widehat{\gamma}$ and that $\widehat{\delta}$ is the change of the coefficient value when $Q < \widehat{\gamma}$.}
\end{tabular}
\end{center}
\end{table}

%% M=50
\begin{table}[htbp]
\footnotesize
\begin{center}
\caption{Simulation Results with $M=50$}
\label{tb:M50}
\begin{tabular}{lp{2cm}crrrrrr}
\hline \hline
{Threshold}                 &  {Estimation}   &Constant          &   \multicolumn{3}{c}{Prediction Error (PE)}  & \multirow{2}{*}{$\mathbb{E}\left[\mathcal{M}\left(\widehat{\alpha}\right)\right]$}&   \multirow{2}{*}{$\mathbb{E}\left|\widehat{\alpha}-\alpha_0\right|_1$} &    \multirow{2}{*}{$\mathbb{E}\left|\widehat{\tau}-\tau_0\right|_1$}  \\
\cline{4-6}
Parameter &  Method & for $\lambda$ & Mean & Median & SD & & \\\hline
& & & & & & &\\
\multicolumn{9}{c}{\underline{Jump Scale: $c=1$}} \\
& & & & & & &\\
\multirow{7}{*}{$\tau_0=0.5$} & {Least Squares}  & None & 0.285  &  0.276  &  0.074  & 100.00& 7.066 &0.008 \\
 & \multirow{4}{*}{Lasso}  &$A=2.8$ & 0.041  &  0.030  &  0.035  & 12.94 &0.466 & 0.010 \\
 & & $A=3.2$ & 0.048  &  0.033  &  0.049  &10.14 &0.438  &0.013 \\
 & & $A=3.6$ & 0.067  &  0.037  &  0.086  &8.44 &0.457  & 0.024\\
 & & $A=4.0$ & 0.095  &  0.050  &  0.120  &7.34  &0.508   &0.040 \\
 & Oracle 1  & None & 0.013  &  0.006  &  0.019  &4.00 &0.164  &0.004  \\
 & Oracle 2  & None & 0.005  &  0.004  &  0.004  & 4.00&0.163 &0.000 \\
& & & & & & &\\
\hline
& & & & & & &\\
\multirow{7}{*}{$\tau_0=0.4$} & {Least Squares}  & None & 0.317  &  0.304  &  0.095  & 100.00 &7.011 &0.008\\
 & \multirow{4}{*}{Lasso}  &$A=2.8$ & 0.052  &  0.034  &  0.063  &13.15 &0.509  &0.016 \\
 & & $A=3.2$ & 0.063  &  0.037  &  0.083  & 10.42 &0.489 & 0.023 \\
 & & $A=3.6$ & 0.090  &  0.045  &  0.121  &8.70 &0.535 & 0.042\\
 & & $A=4.0$ & 0.133  &  0.061  &  0.162  &7.68  &0.634 &0.078  \\
 & Oracle 1  & None & 0.014  &  0.006  &  0.022  &4.00 &0.163 &0.004 \\
 & Oracle 2  & None & 0.005  &  0.004  &  0.004  & 4.00 &0.163 &0.000 \\
& & & & & & &\\
\hline
& & & & & & &\\
\multirow{7}{*}{$\tau_0=0.3$} & {Least Squares}  & None & 2.559  &  0.511  &  16.292  &100.00 & 12.172&0.012\\
 & \multirow{4}{*}{Lasso}  &$A=2.8$ & 0.062  &  0.035  &  0.091  &13.45 & 0.602 &0.030  \\
 & & $A=3.2$ & 0.089  &  0.041  &  0.125  &10.85 & 0.633 &0.056 \\
 & & $A=3.6$ & 0.127  &  0.054  &  0.159  &9.33 & 0.743 &0.099 \\
 & & $A=4.0$ & 0.185  &  0.082  &  0.185  &8.43 & 0.919 &0.168 \\
 & Oracle 1  & None & 0.012  &  0.006  &  0.017  & 4.00&0.177 & 0.004\\
 & Oracle 2  & None & 0.005  &  0.004  &  0.004  & 4.00&0.176  &0.000\\
& & & & & & &\\
\hline
& & & & & & &\\
\multicolumn{9}{c}{\underline{Jump Scale: $c=0$}} \\
& & & & & & &\\
\multirow{6}{*}{N/A} & {Least Squares}  & None & 6.332  &  0.460  &  41.301  &100.00&20.936 & \multirow{6}{*}{ N/A} \\
 & \multirow{4}{*}{Lasso}  &$A=2.8$ & 0.013  &  0.011  &  0.007  &9.30 & 0.266  & \\
 & & $A=3.2$ & 0.014  &  0.012  &  0.008   & 6.71 &  0.227 & \\
 & & $A=3.6$ & 0.015  &  0.014  &  0.009   & 4.95 &0.211 & \\
 & & $A=4.0$ & 0.017  &  0.016  &  0.010   &3.76 & 0.204  & \\
 & Oracle 1 \& 2  & None & 0.002  &  0.002  &  0.003   &2.00 & 0.054 & \\
% & Oracle 2  & None & 0.002  &  0.002  &  0.003  & && \\
& & & & & & &\\
\hline
\\
\multicolumn{9}{p{\textwidth}}{\footnotesize \emph{Note: }$M$ denotes the column size of $X_i$ and $\tau$ denotes the threshold parameter. Oracle 1 \& 2 are estimated by the least squares when sparsity is known and when sparsity and $\tau_0$ are known, respectively. All simulations are based on 400 replications of a sample with 200 observations. }
\end{tabular}
\end{center}

\end{table}

\begin{figure}[p]
\caption{Mean Prediction Errors and Mean $\mathcal{M}(\widehat{\alpha})$}
\begin{center}

%\includegraphics[width = 2.8in]{plot_PE_M50.pdf}
%\includegraphics[width = 2.8in]{plot_Malpha_M50.pdf} \\[0pt]
%$M=50$\\[0pt]
\includegraphics[width = 2.8in]{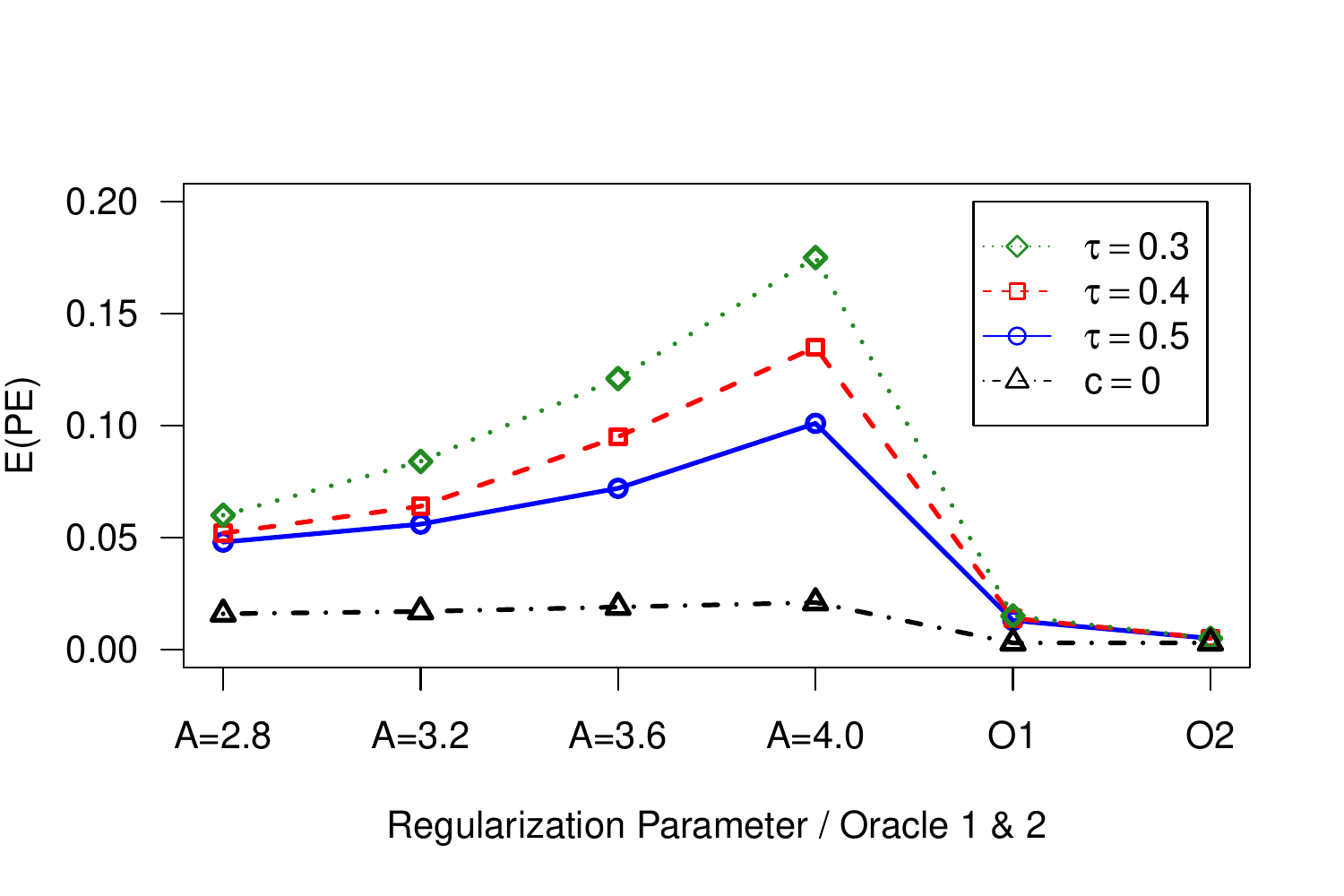}
\includegraphics[width = 2.8in]{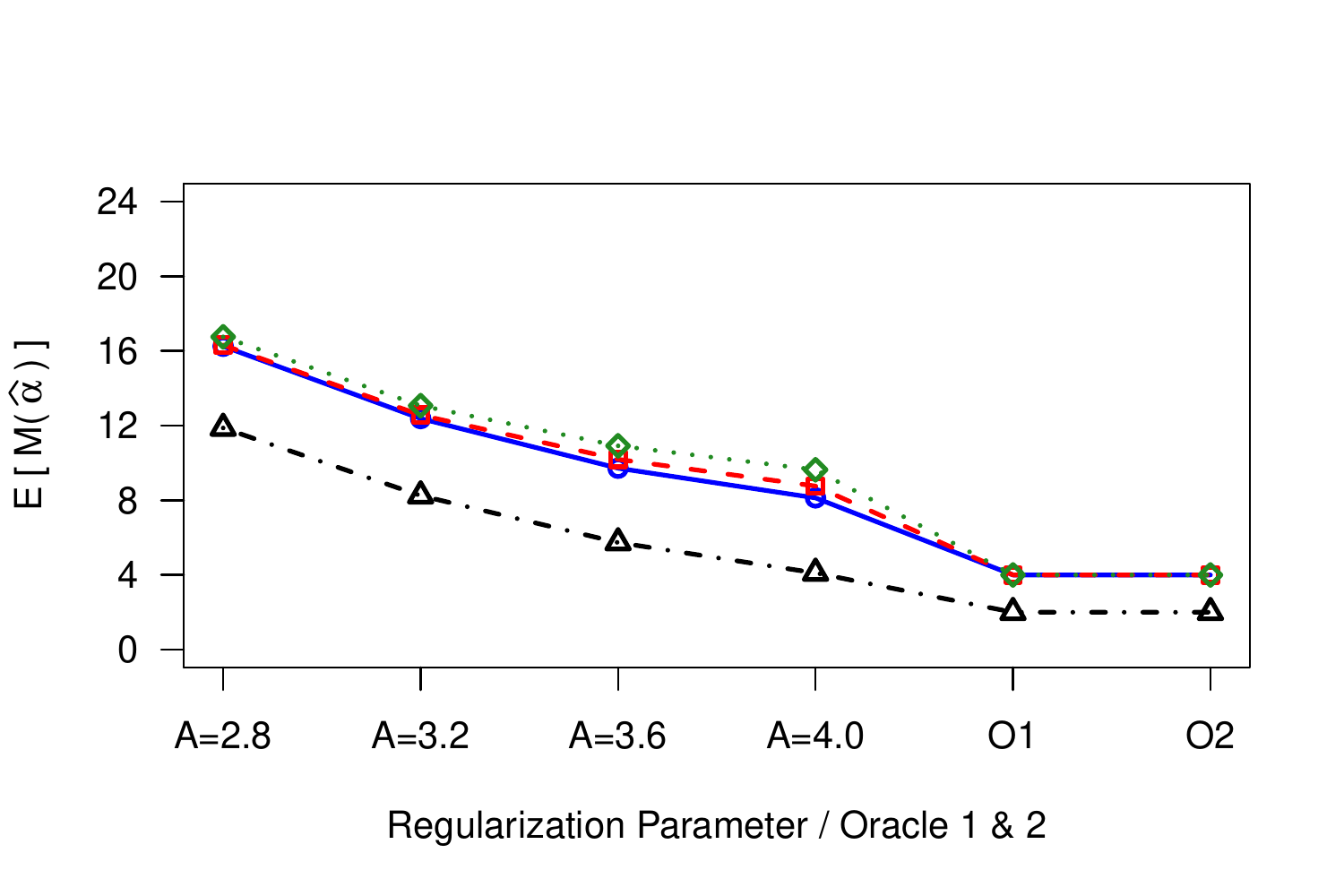} \\[0pt]
$M=100$\\[0pt]
\includegraphics[width = 2.8in]{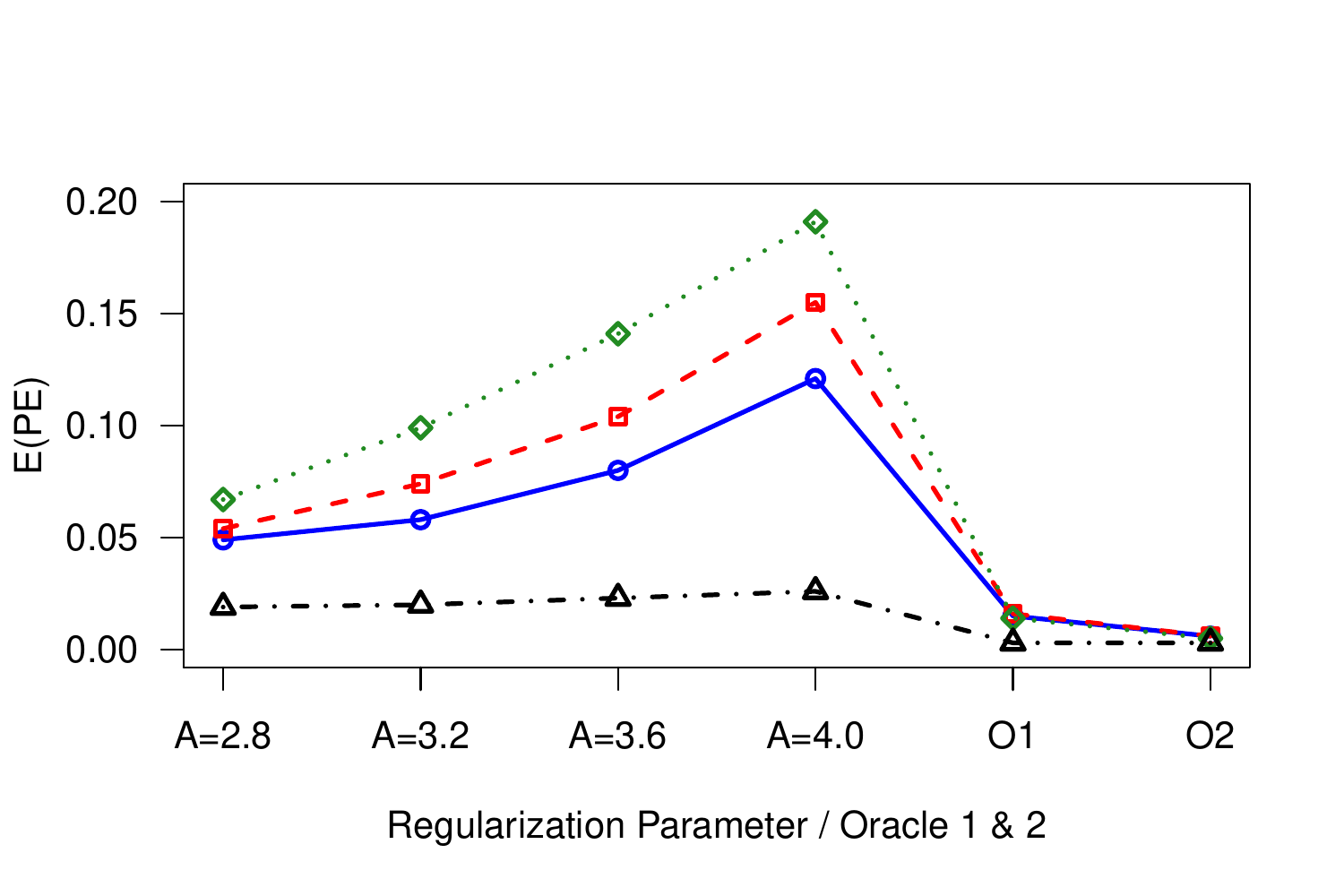}
\includegraphics[width = 2.8in]{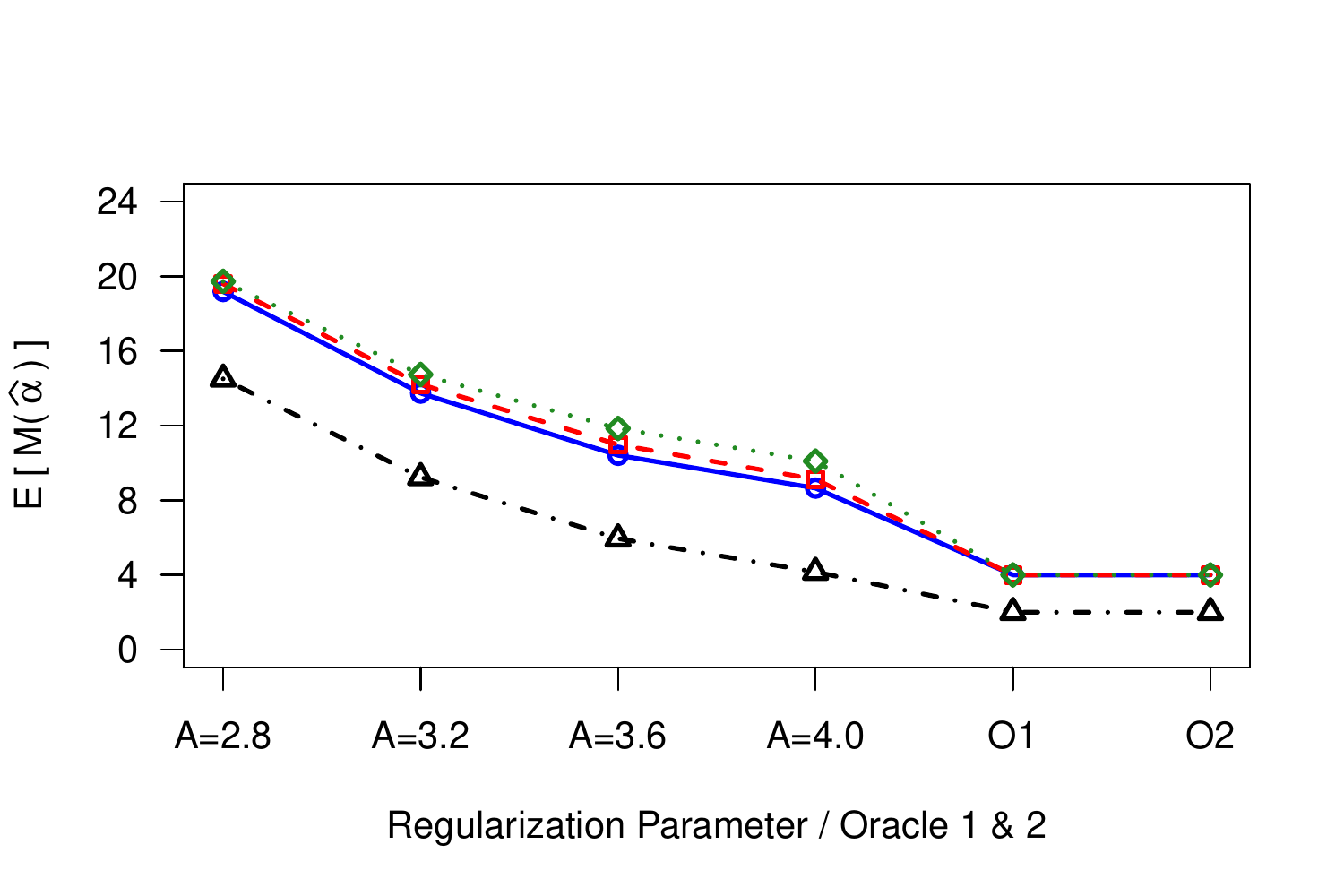} \\[0pt]
$M=200$\\[0pt]
\includegraphics[width = 2.8in]{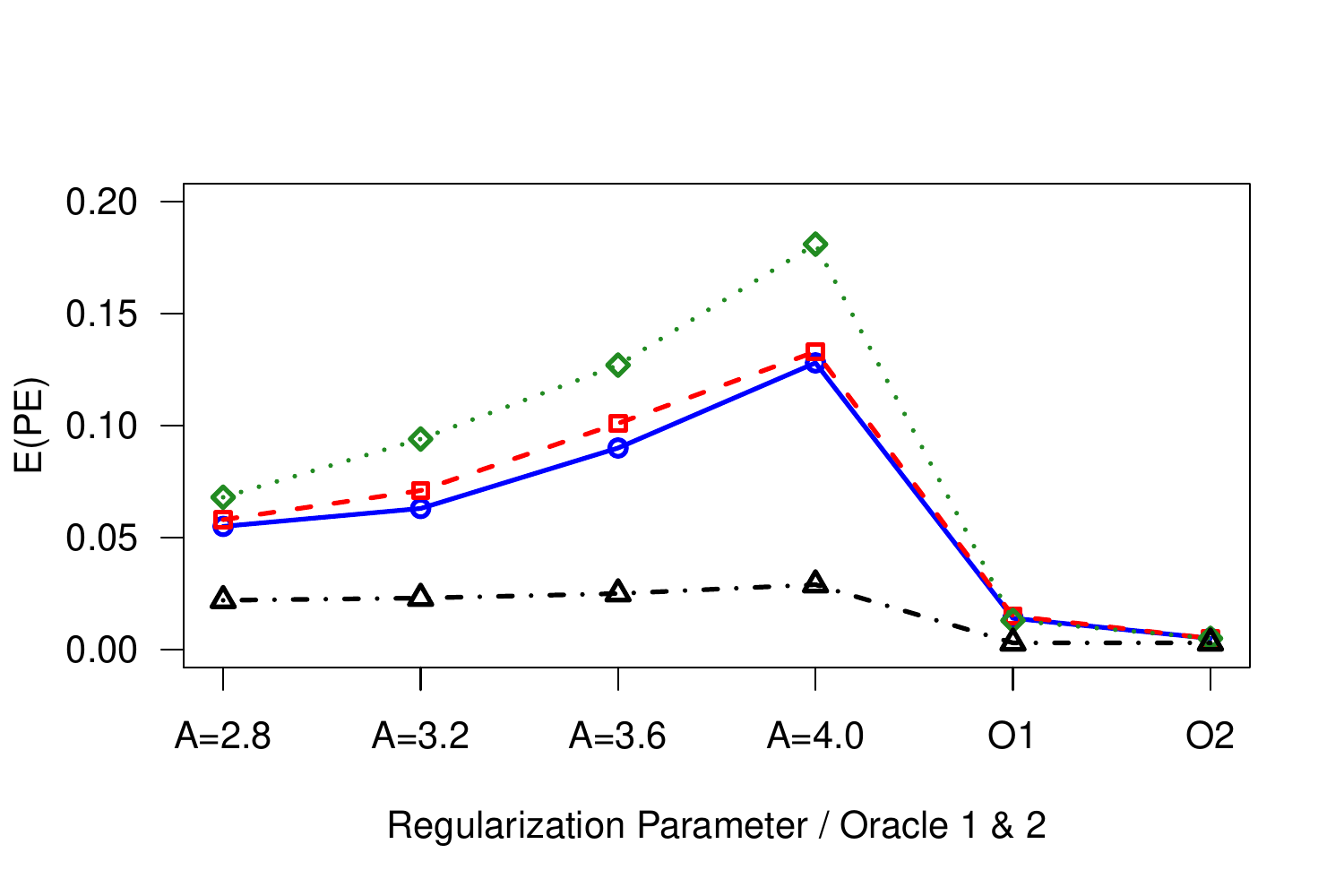}
\includegraphics[width = 2.8in]{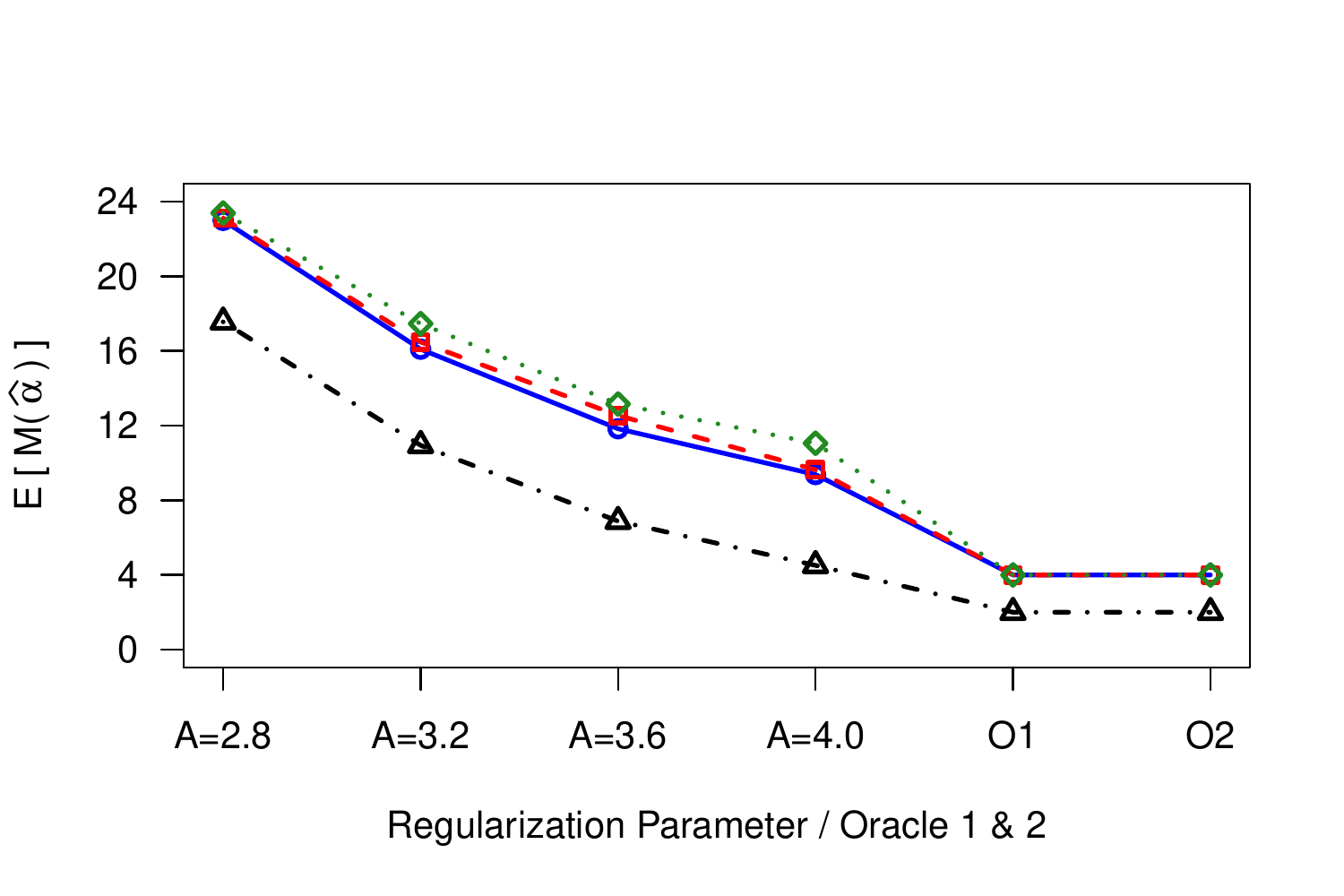} \\[0pt]
$M=400$\\[0pt]
\end{center}
\par
\label{fg:indepX1}
\end{figure}

\begin{figure}[tbp]
\caption{Mean $\ell_1$-Errors for $\alpha$ and $\tau$}
\begin{center}

%\includegraphics[width = 2.8in]{plot_l1_M50.pdf}
%\includegraphics[width = 2.8in]{plot_l1_tau_M50.pdf} \\[0pt]
%$M=50$\\[0pt]
\includegraphics[width = 2.8in]{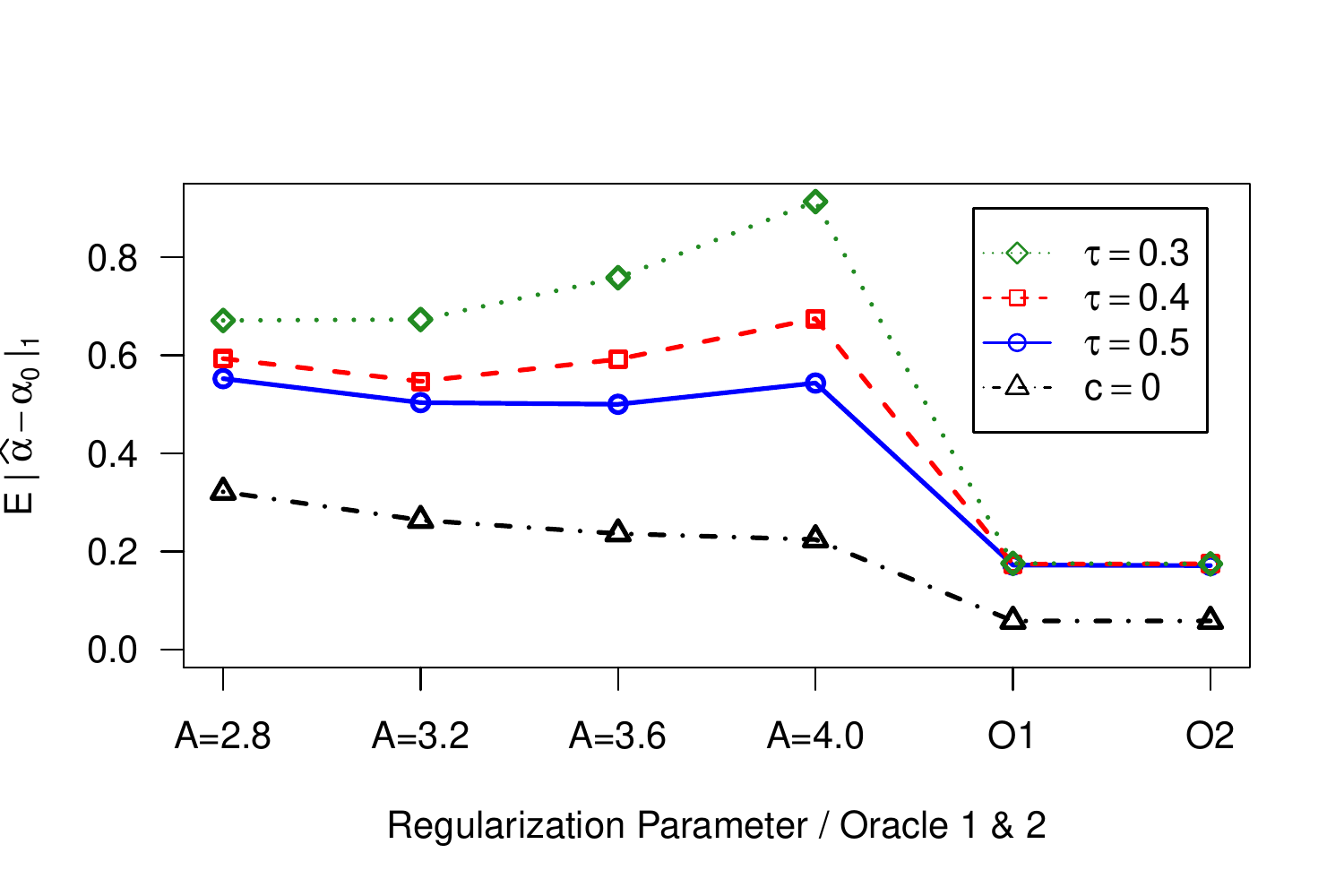}
\includegraphics[width = 2.8in]{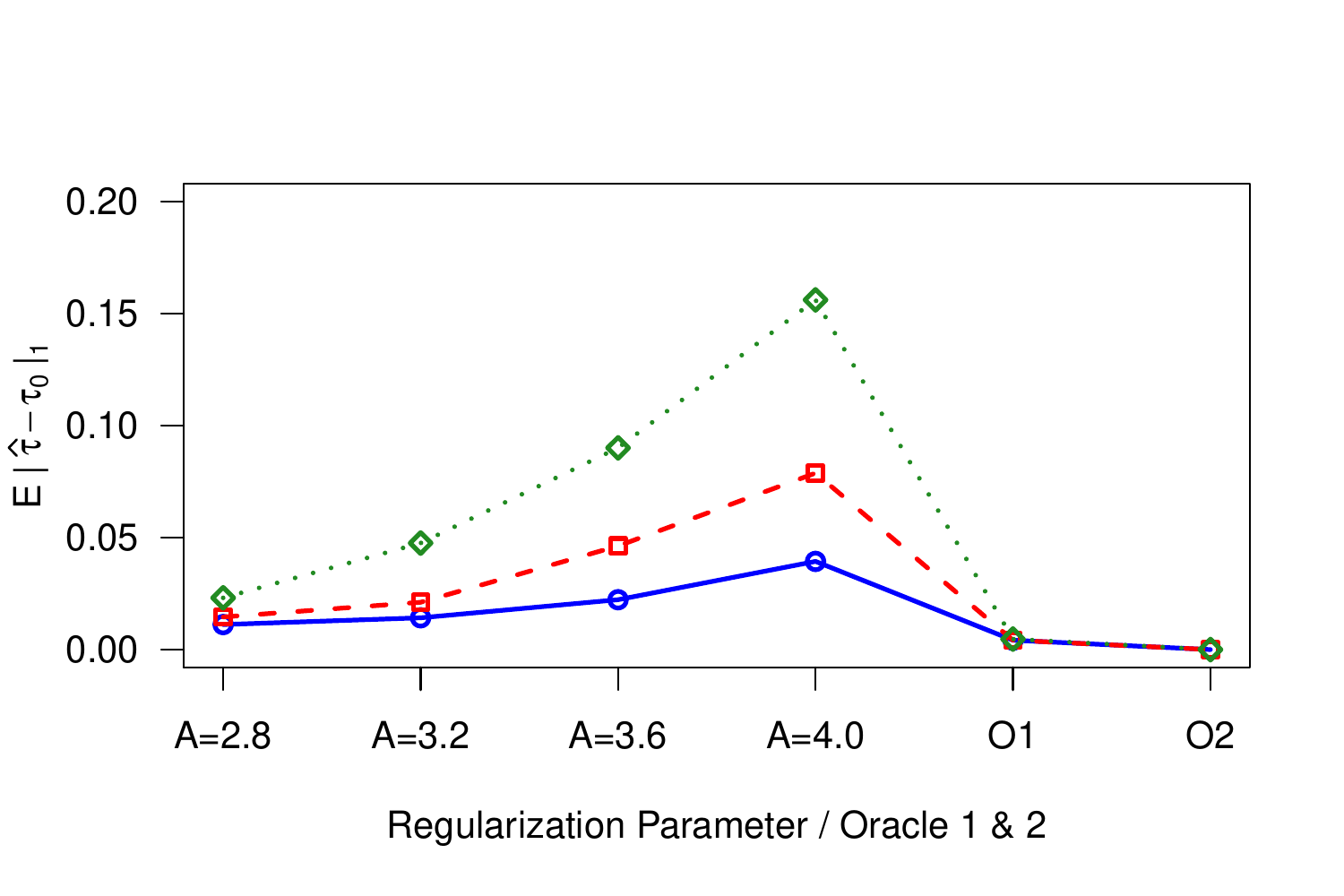} \\[0pt]
$M=100$\\[0pt]
\includegraphics[width = 2.8in]{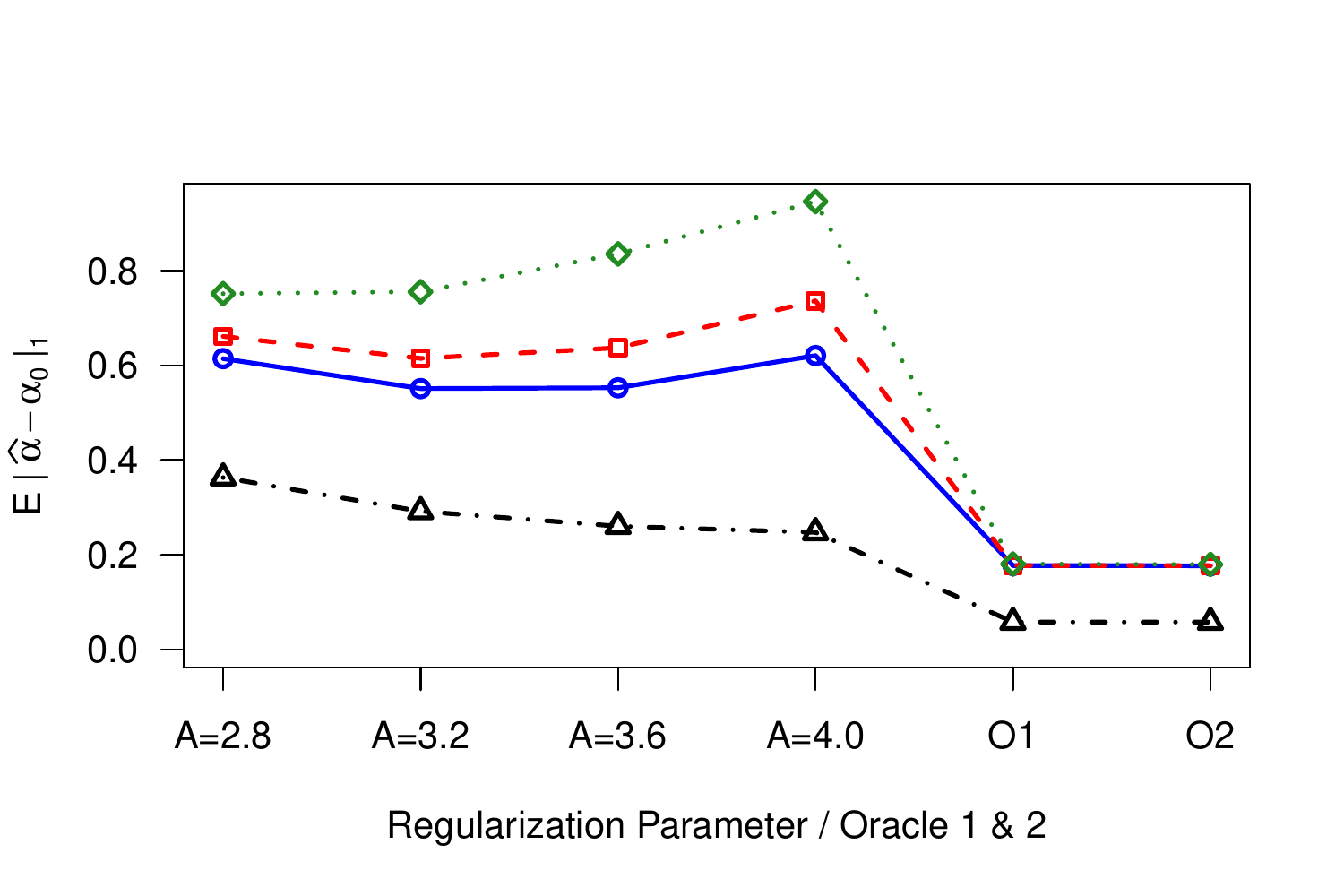}
\includegraphics[width = 2.8in]{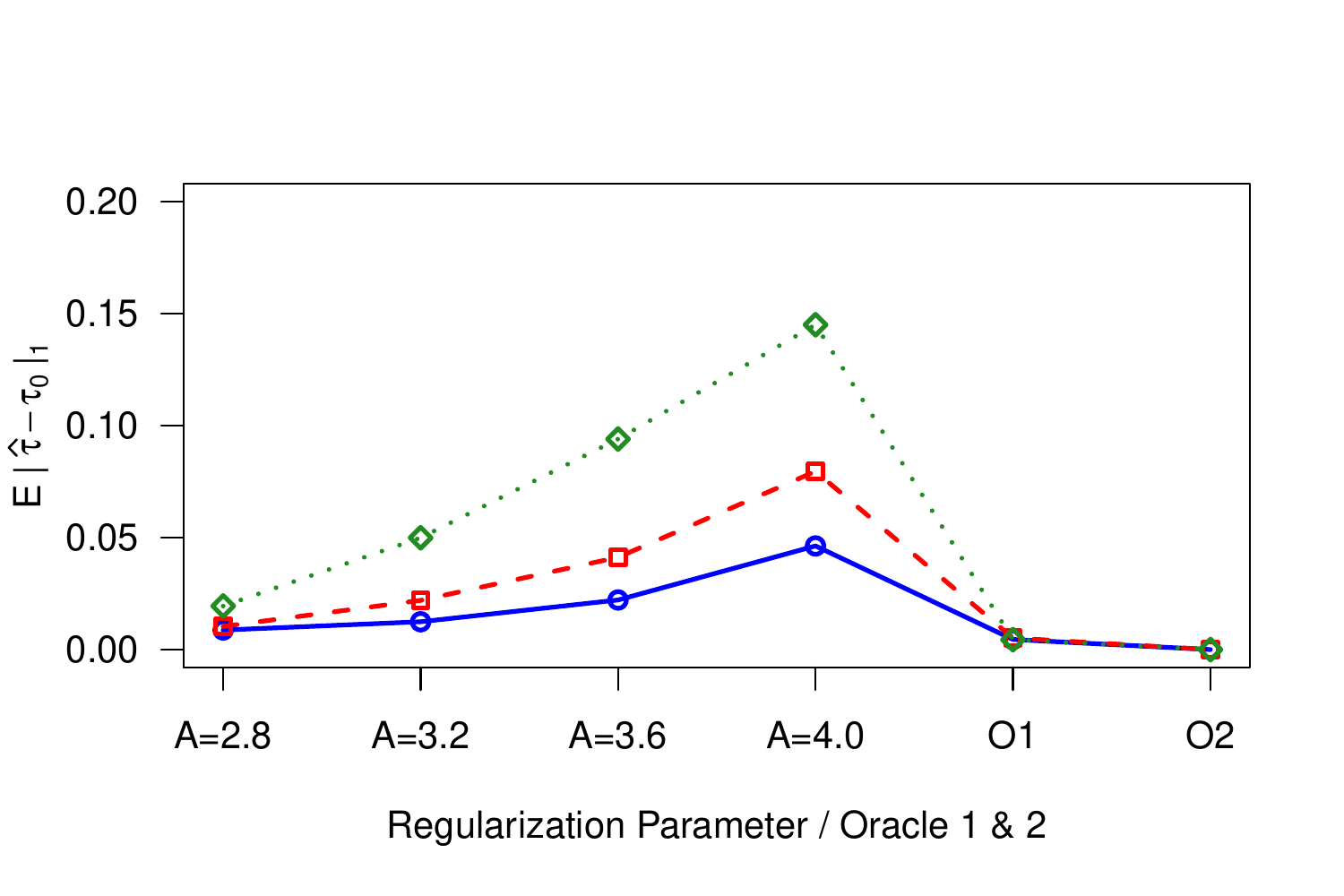} \\[0pt]
$M=200$\\[0pt]
\includegraphics[width = 2.8in]{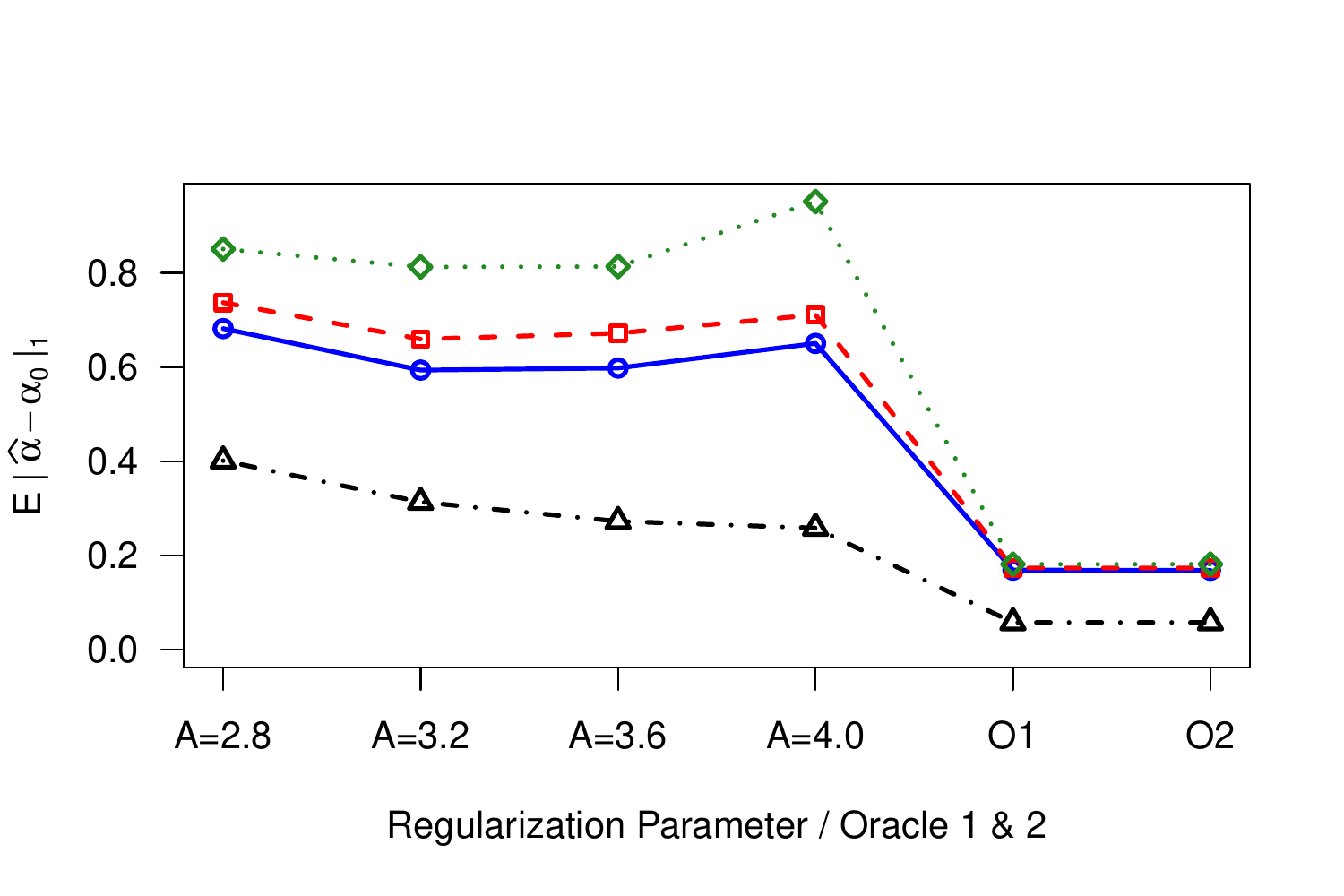}
\includegraphics[width = 2.8in]{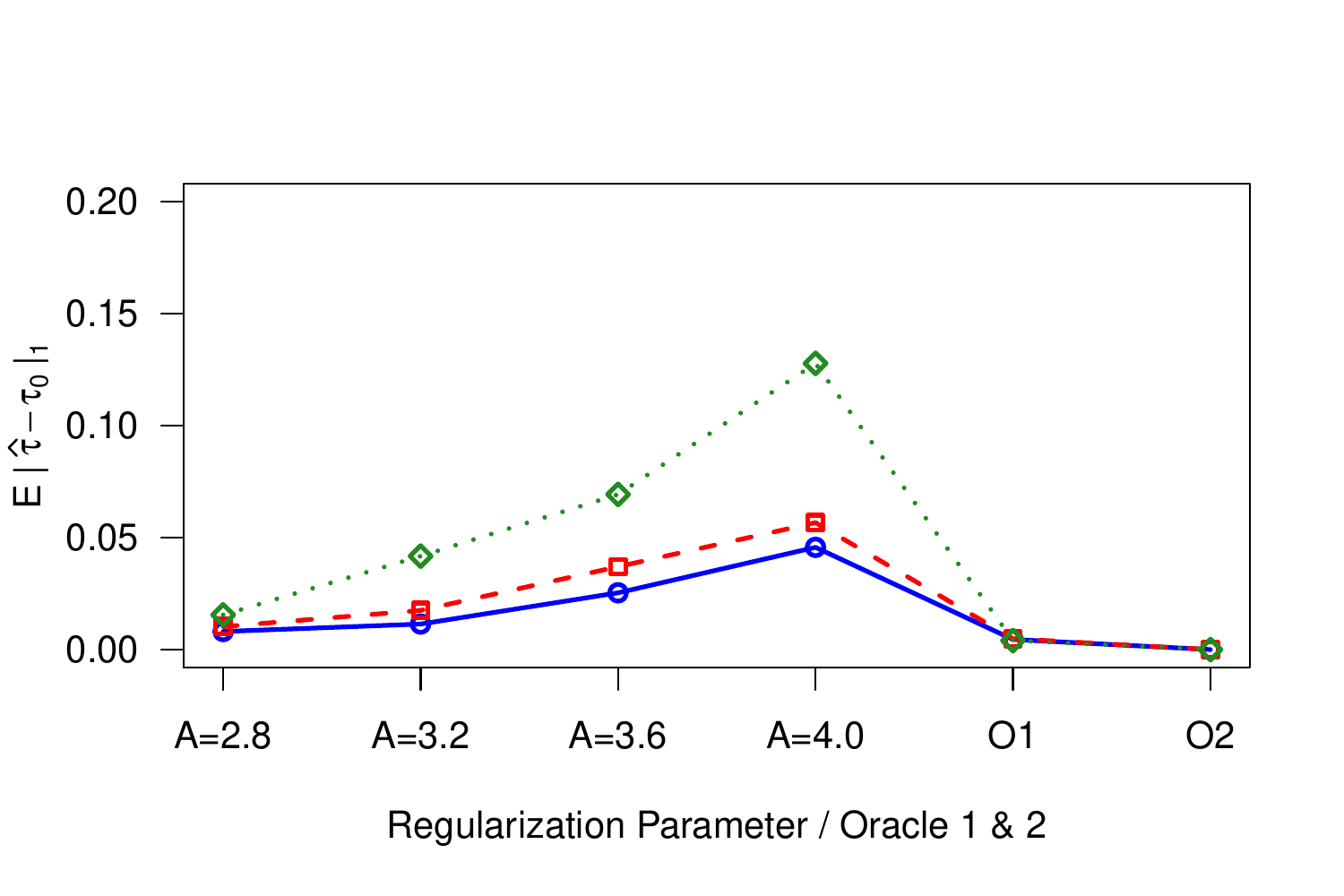} \\[0pt]
$M=400$\\[0pt]
\end{center}
\par
\label{fg:indepX2}
\end{figure}

\clearpage

\renewcommand\thepage{A-\arabic{page}} \setcounter{page}{1}

\appendix

\section*{Appendices}
We first define some notation used in the appendices. Let $a \vee b \equiv \max\{a,b\}$ and $a \wedge b \equiv \min\{a,b\}$ for any real numbers $a$ and $b$. 
For two (positive semi-definite) matrices $\mathbf{V}_1$ and $\mathbf{V}_2$, define the supremum distance $d_\infty(\mathbf{V}_1, \mathbf{V}_2) :=\max_{j,k}|(\mathbf{V}_1)_{j,k} -  (\mathbf{V}_2)_{j,k}|$.
Let $\widehat{\mathbf{D}}=%
\mathbf{D}(\widehat{\tau })$ and $\mathbf{D=D}\left( \tau _{0}\right) ,$ and similarly, let $\widehat{S}_{n}=S_{n}(\widehat{\alpha },\widehat{\tau })$ and $S_{n}=S_{n}\left( \alpha _{0},\tau _{0}\right)$.
Recall that $\mathbf{X}(\tau )$ denotes the $(n\times 2M)$ matrix whose $i$-th row is $\mathbf{X}_{i}(\tau )^{\prime }=(X_{i}^{\prime },X_{i}^{\prime }1\{Q_{i}<\tau\})^{\prime }$.  Define $X_{\max }
:=\max_{\tau, j} \big\{ \left\Vert \mathbf{X}^{(j)}(\tau)\right\Vert _{n},\ \ j=1,...,2M, \tau \in \mathbb{T} \big\}$ and $X_{\min }:=\min_{j} \big\{ \left\Vert \mathbf{X}^{(j)}(t_0)\right\Vert _{n},\ \ j=1,...,2M \big\}$ where $t_0$ is from $\mathbb{T}\equiv[t_0,t_1]$. Also,$\ $let $\alpha _{\max} $ denote the maximum value that all the elements of $\alpha $ can take in absolute value.

\section{Sufficient Conditions for the Uniform Restricted Eigenvalue Assumption}\label{sec:suff-ure}

In this section of the appendix, 
we provide two sets of sufficient conditions for Assumption \ref{re-assump}. 

\subsection{The First Sufficient Condition}

The first approach is based on modifications of Assumption 2 of \citet{Bickel-et-al:09}.
We first write
$\mathbf{X}\left( \tau \right) = (\mathbf{\tilde{X}}, \mathbf{\tilde{X}}\left( \tau \right))$
where  $\mathbf{\tilde{X}}$ is the $\left( n\times M\right) $
matrix whose $i$-th row is $X_{i}^{\prime }$,
and
$\mathbf{\tilde{X}}\left( \tau \right)$
is
the $\left( n\times M\right) $
matrix whose $i$-th row is $X_{i}^{\prime }1\{Q_{i}<\tau \}$, respectively.
Define the following Gram matrices:
\begin{align*}
\Psi_n(\tau) &:= n^{-1} \mathbf{X}\left( \tau\right) ^{\prime }\mathbf{X}\left( \tau \right), \\
\Psi_{n,+}(\tau) &:= n^{-1} \mathbf{\tilde{X}}\left( \tau \right) ^{\prime }\mathbf{\tilde{X}}\left( \tau \right), \\
\Psi_{n,-}(\tau) &:=  n^{-1} \left[\mathbf{\tilde{X}} - \mathbf{\tilde{X}}\left( \tau \right)\right]^{\prime }\left[\mathbf{\tilde{X}} - \mathbf{\tilde{X}}\left( \tau \right)\right],
\end{align*}
and define the following restricted eigenvalues:
\begin{align*}
\phi _{\min }\left( u,\tau \right) &:= \min_{x\in \mathbb{R}^{2M}:1\leq
\mathcal{M}\left( x\right) \leq u}\frac{x^{\prime } \Psi_n(\tau) x}{x^{\prime }x},\ %\text{and }%
\phi _{\max}\left( u,\tau \right) :=\max_{x\in \mathbb{R}^{2M}:1\leq
\mathcal{M}\left( x\right) \leq u}\frac{x^{\prime } \Psi_n(\tau) x}{x^{\prime }x},\\
\phi _{\min,+ } \left( u,\tau \right)  &:=\min_{x\in \mathbb{R}^{M}:1\leq
\mathcal{M}\left( x\right) \leq u}\frac{x^{\prime }\Psi_{n,+}(\tau) x}{x^{\prime }x},
\phi _{\max,+ } \left( u,\tau \right)  :=\max_{x\in \mathbb{R}^{M}:1\leq
\mathcal{M}\left( x\right) \leq u}\frac{x^{\prime }\Psi_{n,+}(\tau) x}{x^{\prime }x}
\end{align*}%
and $\phi _{\min,- } \left( u,\tau \right)$ and $\phi _{\max,- } \left( u,\tau \right)$ are defined analogously with
$\Psi_{n,-}(\tau)$.
Let%
\begin{align*}
\kappa _{2}\left( s,m,c_{0},\tau \right) &:=\sqrt{\phi _{\min }\left( s+m,\tau
\right) }\left( 1-c_{0}\sqrt{\frac{s\phi _{\max }\left( m,\tau \right) }{%
m\phi _{\min }\left( s+m,\tau \right) }}\right), \\
\psi &:= \min_{\tau \in \mathbb{S}} \frac{ \phi _{\max,-}\left( 2m,\tau
\right) \wedge \phi _{\max,+ }\left( 2m,\tau \right) }{ \phi
_{\max,- } \left( 2m,\tau \right) \vee \phi _{\max,+ }\left( 2m,\tau
\right) }.
\end{align*}

\begin{lem}\label{lm-suffURE}
Assume that the following holds uniformly in $\tau \in \mathbb{S}$:
\begin{align}\label{brt-assumption2}
\begin{split}
m\phi _{\min,+} \left( 2s+2m,\tau \right)  &> c_1^{2} s\phi _{\max,+} \left( 2m,\tau \right),  \\
m\phi _{\min,-} \left( 2s+2m,\tau \right)  &> c_1^{2} s\phi _{\max,-} \left( 2m,\tau \right)
\end{split}
\end{align}%
for some integers $s,m$ such that $1 \leq s \leq M/4$, $m\geq s$ and $2s+2m \leq M
$ and a constant $c_{1} > 0$.
Also, assume that $\psi >0.$
Then, Assumption \ref{re-assump}
is satisfied with $c_0 = c_1 \sqrt{\psi/(1+\psi)}$ and  $\kappa \left( s,c_{0},\mathbb{S}\right) =\min_{\tau \in \mathbb{S}}\kappa _{2}\left(
s,m,c_{0},\tau \right).$
\end{lem}

Conditions in \eqref{brt-assumption2} are modifications of Assumption 2 of \citet{Bickel-et-al:09}.
Note that for each $\tau \in  \mathbb{S}$, data are split into two subsamples with corresponding Gram matrices $\Psi_{n,+}(\tau)$
$\Psi_{n,-}(\tau)$, respectively. Hence, conditions in \eqref{brt-assumption2} are equivalent to stating that
 Assumption 2 of \citet{Bickel-et-al:09} holds with a universal constant $c_0$ for each subsample of all possible sample splitting induced by different values of $\tau  \in  \mathbb{S}$.
 As discussed by \citet{Bickel-et-al:09}, if we take $s + m = s \log n$ and assume that $\phi _{\max,+}(\cdot,\cdot)$ and  $\phi _{\max,-}(\cdot,\cdot)$ are uniformly bounded by a constant,
conditions in Lemma \ref{lm-suffURE} are equivalent to
\begin{align*}
\min_{\tau \in \mathbb{S}} \log n \left[ \phi _{\min,+}( 2s \log n, \tau) \wedge \phi _{\min,-}( 2s \log n, \tau ) \right] > c_{URE},
\end{align*}
where  $c_{URE} > 0$ is a constant.

\begin{proof}[\textbf{Proof of Lemma \ref{lm-suffURE}}]
Define $\bar{\mathbf{X}}\left( \tau \right) :=(\mathbf{\tilde{X}}-\mathbf{\tilde{X}}\left( \tau \right), \mathbf{\tilde{X}}\left( \tau \right)).$
For any $y=\left( y_{1}^{\prime },y_{2}^{\prime }\right) ^{\prime}$
such that $y_1, y_2  \in \mathbb{R}^{M} \setminus \{ 0 \}$, let $x_{1}=y_{1}/\sqrt{y^{\prime }y},x_{2}=y_{2}/%
\sqrt{y^{\prime }y}.$ Then $x_{1}^{\prime }x_{1}+x_{2}^{\prime
}x_{2}=1$.
Furthermore, since  $[\mathbf{\tilde{X}}-\mathbf{\tilde{X}}\left( \tau \right)]'\mathbf{\tilde{X}}\left( \tau \right) = 0$, we have
\begin{equation*}
\frac{y^{\prime } n^{-1} \bar{\mathbf{X}}\left( \tau \right) ^{\prime }\bar{\mathbf{X}}\left( \tau
\right) y}{y^{\prime }y}=
\frac{x_{1}^{\prime } \Psi_{n,-}(\tau)  x_{1}}{x_{1}^{\prime }x_{1}}x_{1}^{\prime }x_{1}+\frac{%
x_{2}^{\prime }\Psi_{n,+}(\tau)  x_{2}}{x_{2}^{\prime }x_{2}}%
x_{2}^{\prime }x_{2}.
\end{equation*}
Also, note that $\mathcal{M}\left( x_{1}\right) $ and $\mathcal{M}\left( x_{2}\right) $
are smaller than or equal to $\mathcal{M}\left( y\right) .$

Since any selection of $s$ column vectors in $\mathbf{X}\left( \tau \right) $
can be represented by a linear transformation of a selection of $2s$ column
vectors of $\bar{\mathbf{X}}\left( \tau \right) $,
the minimum restricted eigenvalue of dimension $2s$ for $\bar{\mathbf{X}}\left( \tau \right)$  can be smaller than
that of dimension $s$ for $\mathbf{X}\left( \tau \right) $. Likewise,
the maximum restricted eigenvalue of dimension $2s$ for $\bar{\mathbf{X}}\left( \tau \right)$  can be larger than
that of dimension $s$ for $\mathbf{X}\left( \tau \right) $.
Thus, with $u=2s+2m$,%
\begin{eqnarray*}
m\min_{y\in \mathbb{R}^{2M}:1\leq \mathcal{M}\left( y\right) \leq s+m}\frac{%
y^{\prime } n^{-1} \mathbf{X}\left( \tau \right) ^{\prime }\mathbf{X}\left( \tau
\right) y}{y^{\prime }y} &\geq &m\min_{y\in \mathbb{R}^{2M}:1\leq \mathcal{M}%
\left( y\right) \leq u}\frac{y^{\prime } n^{-1} \bar{\mathbf{X}}\left( \tau \right) ^{\prime }%
\bar{\mathbf{X}}\left( \tau \right) y}{y^{\prime }y} \\
&\geq &m\left( \phi _{\min,-}\left( u,\tau \right) \wedge \phi _{\min,+}\left( u,\tau \right) \right)  \\
&>&c_{1}^{2}s\left( \phi _{\max,- } \left( 2m,\tau \right) \wedge \phi
_{\max,+ }\left( 2m,\tau \right) \right)  \\
&>&c_{1}^{2}s\frac{\psi }{1+\psi }\max_{y\in \mathbb{R}^{2M}:1\leq \mathcal{M%
}\left( y\right) \leq 2m}\frac{y^{\prime } n^{-1} \bar{\mathbf{X}}\left( \tau \right)
^{\prime }\bar{\mathbf{X}}\left( \tau \right) y}{y^{\prime }y} \\
&\geq &c_{1}^{2}s\frac{\psi }{1+\psi }\max_{y\in \mathbb{R}^{2M}:1\leq
\mathcal{M}\left( y\right) \leq m}\frac{y^{\prime } n^{-1} \mathbf{X}\left( \tau
\right) ^{\prime }\mathbf{X}\left( \tau \right) y}{y^{\prime }y}.
\end{eqnarray*}%
This implies that \citet{Bickel-et-al:09}'s Assumption 2 hold for $\mathbf{X}%
\left( \tau \right) $ with $c_{0}=c_{1}\sqrt{\psi /\left( 1+\psi \right) }.$
Then, it follows from their Lemma 4.1 that Assumption \ref{re-assump} is
satisfied with $\kappa \left( s,c_{0}\right) =\min_{\tau \in \mathbb{S}%
}\kappa _{2}\left( s,m,c_{0},\tau \right) .$
\end{proof}

\subsection{The Second Sufficient Condition}\label{subsec:suff:URE}

The second approach is in the spirit of Section 10.1 of  \citet{vdGeer:Buhlmann:09}.
In this subsection, we provide primitive sufficient conditions for our simulation designs in Section \ref{sec:MC} and Appendix \ref{sec:mc:extra}. 
In our simulation designs,
$X_i$ is independent and identically distributed (i.i.d.) as $N(0,\Sigma)$.
The independent design case is with $\Sigma = I_M$ and the dependent case is  with $(\Sigma)_{i,j} = \rho^{|i-j|}$, where $(\Sigma)_{i,j}$ denotes the (i,j) element of the $M \times M$ covariance matrix $\Sigma$. 
Also,  $Q_i$ is independent of $X_i$ and i.i.d. from $\text{Unif}(0,1)$. 

Define
$\hat{\mathbf{V}}(\tau) := \mathbf{X}(\tau)'\mathbf{X}(\tau)/n$
and $\mathbf{V}(\tau) := \mathbb{E}[\mathbf{X}_i(\tau)  \mathbf{X}_i(\tau)']$.
In our simulation designs,
$\mathbf{V}(\tau) = \Omega \otimes \Sigma$, where 
\begin{align*}
\Omega &\equiv 
\left(
\begin{array}{cc}
   1 &  \tau    \\
 \tau &  \tau     \\ 
\end{array}
\right),
\end{align*}
since $Q_i$ and $X_i$ are independent of each other and $\mathbb{P}[Q_i <\tau] = \tau$.

For a positive semi-definite, $2M \times 2M$ matrix $\mathbf{V}$, define
\begin{equation*}
\kappa (\mathbf{V}; s,c_{0},\mathbb{S}):=\min_{\tau \in \mathbb{S}}\min_{\substack{ J_{0}\subseteq
\{1,\ldots ,2M\},  \\ |J_{0}|\leq s}}\min_{\substack{ \gamma \neq 0,  \\ %
\left\vert \gamma _{J_{0}^{c}}\right\vert _{1}\leq c_{0}\left\vert \gamma
_{J_{0}}\right\vert _{1}}}\frac{(\gamma' \mathbf{V} \gamma)^{1/2} }{
|\gamma _{J_{0}}|_{2}}.
\end{equation*}
As in \citet{vdGeer:Buhlmann:09}, 
define the supremum distance:
\begin{align*}
d_\infty(\mathbf{V}_1, \mathbf{V}_2) :=
\max_{j,k}|(\mathbf{V}_1)_{j,k} -  (\mathbf{V}_2)_{j,k}|
\end{align*}
for two (positive semi-definite) matrices $\mathbf{V}_1$ and $\mathbf{V}_2$.

%Assume that 
%$\kappa (\mathbf{V}(\tau); s,c_{0},\mathbb{S}) > 0$ and that 
%$\sup_{\tau \in \mathbb{T}} d_\infty(\hat{\mathbf{V}}(\tau), \mathbf{V}(\tau)) \leq \tilde{r}$.

Note that for any $2M$-dimensional, nonzero vector $\gamma$ such that
$|\gamma_{J_0^c}|_1 \leq c_0  |\gamma_{J_0}|_1$, we have that 
\begin{align*}
|\gamma' (\hat{\mathbf{V}}(\tau) - \mathbf{V}(\tau)) \gamma| 
&\leq \sup_{\tau \in \mathbb{T}} d_\infty(\hat{\mathbf{V}}(\tau), \mathbf{V}(\tau))  |\gamma|_1^2  \\
&\leq \sup_{\tau \in \mathbb{T}} d_\infty(\hat{\mathbf{V}}(\tau), \mathbf{V}(\tau))  (1+c_0)^2|\gamma_{J_0}|_1^2 \\
&\leq \sup_{\tau \in \mathbb{T}} d_\infty(\hat{\mathbf{V}}(\tau), \mathbf{V}(\tau))  (1+c_0)^2 s|\gamma_{J_0}|_2^2, 
%\\
%&\leq \tilde{r} (1+c_0)^2   s \frac{\gamma' \mathbf{V}(\tau) \gamma }{
%\kappa^2 (\mathbf{V}(\tau); s,c_{0},\mathbb{S})}
\end{align*}
which implies that
\begin{align*}
\kappa (\hat{\mathbf{V}}(\tau); s,c_{0},\mathbb{S}) 
\geq \kappa (\mathbf{V}(\tau); s,c_{0},\mathbb{S})
- (1+c_0) \sqrt{s \times \sup_{\tau \in \mathbb{T}} d_\infty(\hat{\mathbf{V}}(\tau), \mathbf{V}(\tau)) }.
\end{align*}
Hence, our simulation design satisfies
Assumption \ref{re-assump} with probability approaching one as $n \rightarrow \infty$ if we establish the following two steps:
\begin{enumerate}
\item[\textbf{Step 1}.] $\kappa (\mathbf{V}(\tau); s,c_{0},\mathbb{S}) > 0$ (a population version of the URE condition),

\item[\textbf{Step 2}.] $s \times \sup_{\tau \in \mathbb{T}} d_\infty(\hat{\mathbf{V}}(\tau), \mathbf{V}(\tau)) \rightarrow_p 0$  as $n \rightarrow \infty$ (uniform convergence in probability of the sample covariance matrix with a rate faster than $s^{-1}$). 
\end{enumerate}

\begin{proof}[Proof of Step 1]
All the eigenvalues of the Kronecker product of $\Omega$ and $\Sigma$
can be written as the product between eigenvalues of $\Omega$ and those of $\Sigma$.
First, note that as long as $\tau$ belongs to a strict compact interval between 0 and 1, as in our simulations ($\mathbb{T} \in [0.15, 0.85]$), we have strictly positive eigenvalues for  $\Omega$.
Second, if $\Sigma$ is an identity matrix, then all eigenvalues are 1's; 
if $\Sigma$ is a Toeplitz matrix such that $(\Sigma)_{i,j} = \rho^{|i-j|}$, then 
the smallest eigenvalue of $\Sigma$ is $1-\rho$, independent of the dimension $M$.
Hence, in both cases, the smallest eigenvalue of $\mathbf{V}(\tau)$ is bounded away from zero
uniformly in $\tau$. Thus,  it is clear that $\kappa (\mathbf{V}(\tau); s,c_{0},\mathbb{S}) > 0$ holds.
\end{proof}

To prove the second step, it is sufficient to assume that as $n \rightarrow \infty$,
we have that 
$M \rightarrow \infty$ and that
\begin{align}\label{URE-suff-mc}
s = o\left( \sqrt{\frac{n}{\log nM}} \; \right).
\end{align}

\begin{proof}[Proof of Step 2]
For any $j,k=1,\ldots,M$, define 
\begin{align*}
\tilde{V}_{j,k} &:= \frac{1}{n} \sum_{i=1}^n X_i^{(j)}X_i^{(k)}  - 
\mathbb{E}[X_i^{(j)}X_i^{(k)}],\\
 \hat{V}_{j,k}(\tau) &:= \frac{1}{n} \sum_{i=1}^n X_i^{(j)}X_i^{(k)} 1\{ Q_i < \tau \} - 
\mathbb{E}[X_i^{(j)}X_i^{(k)} 1\{ Q_i < \tau \}].
\end{align*}
Note that 
$s \times \sup_{\tau \in \mathbb{T}} d_\infty(\hat{\mathbf{V}}(\tau), \mathbf{V}(\tau))$
is bounded by the maximum between 
$s \times \max_{j,k}|\tilde{V}_{j,k}|$ and  
$s \times \max_{j,k}\sup_{\tau \in \mathbb{T}}
|\hat{V}_{j,k}(\tau)|$. The former is already shown to be $o_p(1)$
by \citet[][Section 10.1]{vdGeer:Buhlmann:09} under the restriction 
$s = o\left[ (n/\log M)^{1/2} \right]$.
Thus, it suffices to show that
\begin{align}\label{need-to-prove}
s \times \max_{j,k}\sup_{\tau \in \mathbb{T}}
|\hat{V}_{j,k}(\tau)| = o_p(1).
\end{align}

Since  $X_i$ is i.i.d. as $N(0,\Sigma)$,
$Q_i$ i.i.d. as $\text{Unif}(0,1)$, and 
 $Q_i$ is independent of $X_i$ in our simulation designs, 
there exists a universal constant $C < \infty$ such that 
\begin{align*}
\max_{j,k} \sup_{\tau \in \mathbb{T}} \mathbb{E} \left\{ X_i^{(j)}X_i^{(k)} 1\{ Q_i < \tau \} - 
\mathbb{E}[X_i^{(j)}X_i^{(k)} 1\{ Q_i < \tau \}] \right\}^2 < C.
\end{align*}
Suppose that $\hat{V}_{j,k}'(\tau)$ denotes an independent copy of $\hat{V}_{j,k}(\tau)$.
Note that for each $(j,k)$ pair and every $\tau \in \mathbb{T}$, by Chebyshev's inequality, we have 
\begin{align*}
\mathbb{P} \left[ |\hat{V}_{j,k}'(\tau)| \geq  \frac{\varepsilon}{2} \right] \leq 
\frac{4C}{n \varepsilon^2}
\end{align*}
for every $\varepsilon > 0$.
Let $\epsilon_1,\ldots,,\epsilon_n$ be a Rademacher sequence independent of the original data.
For each $(j,k)$ pair, 
\begin{align*}
\mathbb{P} \left[ \sup_{\tau \in \mathbb{T}} |\hat{V}_{j,k}(\tau)| \geq \varepsilon \right] 
&\leq \left[ 1-  \frac{4C}{n \varepsilon^2} \right]^{-1} P \left[ \sup_{\tau \in \mathbb{T}} |\hat{V}_{j,k}(\tau) - \hat{V}_{j,k}'(\tau)| \geq \frac{\varepsilon}{2} \right]  \\
&\leq 2\left[ 1-  \frac{4C}{n \varepsilon^2} \right]^{-1} \mathbb{P} \left[ \sup_{\tau \in \mathbb{T}} \left| 
\frac{1}{n} \sum_{i=1}^n \epsilon_i X_i^{(j)}X_i^{(k)} 1\{ Q_i < \tau \}
\right|\geq \frac{\varepsilon}{4} \right],
\end{align*}
where the first inequality follows from Pollard's first symmetrization lemma 
\citep[][Lemma II.3.8, page 14]{Pollard:84}, 
and
the second inequality comes from Pollard's second symmetrization 
\citep[][page 15]{Pollard:84}.  

Note that
 for all $j,k = 1,\ldots,M$,
$\mathbb{E} \epsilon_i X_i^{(j)}X_i^{(k)} 1\{ Q_i < \tau \} = 0$
and that 
there exists  some universal constant $K < \infty$ such that 
\begin{align*}
%\frac{1}{n} \sum_{i=1}^n 
\mathbb{E} \left[ \left| \epsilon_i X_i^{(j)}X_i^{(k)} 1\{ Q_i < \tau \} \right|^m \right] \leq \frac{m!}{2} K^{m-2}
\end{align*}
for all $m \geq 2$, for all $\tau \in \mathbb{T}$, and  for all $j,k = 1,\ldots,M$.
Then by Bernstein's inequality
\citep[see e.g. Lemma 14.9. of][]{BvdG:11}, we have, for each $(j,k)$ pair and each
$\tau \in \mathbb{T}$,
\begin{align*}
\mathbb{P} \left[ 
\frac{1}{n} \sum_{i=1}^n   \epsilon_i X_i^{(j)}X_i^{(k)} 1\{ Q_i < \tau \}
\geq K t + \sqrt{2t} \right]
\leq \exp( - nt)
\end{align*}
for any $t > 0$.
Since $K$ is a universal constant over $(j,k)$ pairs and 
over $\tau \in \mathbb{T}$ as well, an application of Boole's inequality yields that 
\begin{align*}
\mathbb{P} \left[ 
\max_{j,k} \sup_{\tau \in \mathbb{T}} \frac{1}{n} \sum_{i=1}^n   \epsilon_i X_i^{(j)}X_i^{(k)} 1\{ Q_i < \tau \} \geq K t + \sqrt{2t} \right]
\leq n M^2 \exp( - nt)
\end{align*}
for any $t > 0$.
Combining all the results above yields
\begin{align*}
\mathbb{P} \left[ \max_{j,k} \sup_{\tau \in \mathbb{T}} |\hat{V}_{j,k}(\tau)| \geq 4(K t + \sqrt{2t}) \right] 
&\leq 4 \left[ 1-  \frac{4C}{n (K t + \sqrt{2t})^2} \right]^{-1} n M^2 \exp( - nt)
 \end{align*}
for any $t > 0$.
Then, under the restriction \eqref{URE-suff-mc},
we obtain the desired result by taking 
$t = \tilde{C} \times {\log n M}/{n}$ for a sufficiently large 
$\tilde{C} > 0$. 
\end{proof}

\section{Discussions on Assumption \ref{A-discontinuity}}\label{rem-assump-discontinuity}

We provide further discussions on  Assumption \ref{A-discontinuity}.
Assumption \ref{A-discontinuity} is stronger than just the identifiability of $\tau_0$ as it specifies the rate of
deviation in $f$ as $\tau $ moves away from $\tau _{0}.$ The linear rate
here is sharper than the quadratic one that is usually observed in 
regular M-estimation problems, and it reflects the fact that the limit
criterion function, in the classical setup with a fixed number of stochastic
regressors, has a kink at the true $\tau _{0}.$

For instance,
suppose that $\{(Y_{i},X_{i},Q_{i}):i=1,\ldots ,n\}$ are independent and identically distributed,
and  consider the case where only the intercept is included in $X_i$. Assuming that $Q_i$ has a
density function that is continuous and positive everywhere (so that $%
\mathbb{P}\left( \tau \leq Q_i<\tau _{0}\right) $ and $\mathbb{P}\left( \tau _{0}\leq Q_i<\tau
\right) $ can be bounded below by $c_{1}\left\vert \tau -\tau
_{0}\right\vert $ for some $c_{1}>0$), we have that
\begin{eqnarray*}
&&\mathbb{E}\left( Y_{i}-f_{i}\left( \alpha ,\tau \right) \right) ^{2}-%
\mathbb{E}\left( Y_{i}-f_{i}\left( \alpha _{0},\tau _{0}\right) \right) ^{2}
\\
&=&\mathbb{E}\left( f_{i}\left( \alpha _{0},\tau _{0}\right) -f_{i}\left(
\alpha ,\tau \right) \right) ^{2} \\
&=&\left( \alpha _{1}-\alpha _{10}\right) ^{2}\mathbb{P}\left( Q_i<\tau \wedge \tau
_{0}\right) +\left( \alpha _{2}-\alpha _{20}\right) ^{2}\mathbb{P}\left( Q_i\geq \tau
\vee \tau _{0}\right) \\
&&+\left( \alpha _{2}-\alpha _{10}\right) ^{2}\mathbb{P}\left( \tau \leq Q_i<\tau
_{0}\right) +\left( \alpha _{1}-\alpha _{20}\right) ^{2}\mathbb{P}\left( \tau
_{0}\leq Q_i<\tau \right) \\
&\geq &c\left\vert \tau -\tau _{0}\right\vert ,
\end{eqnarray*}%
for some $c>0,$ where
$f_{i}\left( \alpha ,\tau \right) = X_i'\beta + X_i'\delta 1\{Q_i < \tau\}$,
$\alpha _{1}=\beta +\delta \ $and $\alpha _{2}=\beta ,$
unless $\left\vert \alpha _{2}-\alpha _{10}\right\vert $ is too small when $%
\tau <\tau _{0}$ and $\left\vert \alpha _{1}-\alpha _{20}\right\vert $ is
too small when $\tau >\tau _{0}.$ However, when $\left\vert \alpha
_{2}-\alpha _{10}\right\vert $ is small, say smaller than $\varepsilon ,$ $%
\left\vert \alpha _{2}-\alpha _{20}\right\vert $ is bounded above zero due
to the discontinuity that $\alpha _{10}\neq \alpha _{20}$ and $\mathbb{P}\left( Q_i\geq
\tau \vee \tau _{0}\right) =\mathbb{P}\left( Q_i\geq \tau _{0}\right) $ is also bounded
above zero. This implies the inequality still holds. Since the same
reasoning applies for the latter case, we can conclude our discontinuity
assumption holds in the standard discontinuous threshold regression setup.
In other words, the previous literature has typically imposed conditions
sufficient enough to render this condition.

\subsection{Verification of Assumption \ref{A-discontinuity} for the Simulation Design of Section \ref{sec:MC}}

In this subsection, we may provide more 
primitive discussions for our simulation design in Section \ref{sec:MC},
where $X_{i}\sim N\left( 0,I_{M}\right) $ and $Q_{i}\sim \text{Unif}(0,1)$
independent of $X_{i}$ and $U_{i}\sim N\left( 0,\sigma ^{2}\right) $
independent of $\left( X_{i},Q_{i}\right) .$ 
For simplicity, suppose that $\beta _{0}=0$ and 
$\delta _{0}=\left( c_{0},0,...,0\right) ^{\prime }$ for $c_{0}\neq 0$ and $%
\tau _{0}=0.5.$ 
Recall that $\mathbb{T} = [0.15,0.85]$ in our simulation design.
As our theoretical framework is deterministic design, we may
check if Assumption \ref{A-discontinuity} is satisfied with probability
approaching one as $n\rightarrow \infty .$

We only consider the  case of $\tau < \tau _{0}$ explicitly below.
The other case is similar. 
Note that when $\tau < \tau _{0}$, 
\begin{align*}
\left\Vert f_{\alpha ,\tau }-f_{0}\right\Vert _{n}^{2}& =\frac{1}{n}%
\sum_{i=1}^{n}\left( X_{i}^{\prime }1\left\{ Q_{i} < \tau \right\}
\left( \beta + \delta -
\beta_0 - \delta _{0}\right)
\right) ^{2} \\
&+\frac{1}{n}\sum_{i=1}^{n}\left(
X_{i}^{\prime }1\left\{ Q_{i} \geq \tau_0 \right\} \left( \beta -\beta _{0}\right) \right) ^{2} \\
& +\frac{1}{n}\sum_{i=1}^{n}\left( X_{i}^{\prime }1\left\{ \tau
 \leq Q_{i} < \tau_0 \right\} \left( \beta   -\beta _{0} -\delta_0 \right)
\right) ^{2}.
\end{align*}%
 Then, under our specification of the data generating process,
\begin{align*}
\lefteqn{\mathbb{E}\left\Vert f_{\alpha ,\tau }-f_{0}\right\Vert _{n}^{2}} \\
&=  \tau \left( \beta + \delta -
\beta_0 - \delta _{0}\right)'\left( \beta + \delta -
\beta_0 - \delta _{0}\right) 
+ 
(1-\tau_0) \left( \beta  -
\beta_0 \right)' \left( \beta  -
\beta_0 \right)  \\
&+(\tau_0 - \tau) \left( \beta  -
\beta_0 - \delta_0 \right)'\left( \beta  -
\beta_0 - \delta_0 \right) \\
&\geq (1-\tau_0) \left( \beta  -
\beta_0 \right)' \left( \beta  -
\beta_0 \right) +(\tau_0 - \tau) \left( \beta  -
\beta_0 - \delta_0 \right)'\left( \beta  -
\beta_0 - \delta_0 \right) \\
&= (1-\tau_0)  \beta '  \beta   +(\tau_0 - \tau) \left( \beta   - \delta_0 \right)'\left( \beta  -
 \delta_0 \right) \\
&=  c_{0}^{2}\left( \tau_0 -\tau \right) 
+\beta ^{\prime }\beta \left( 1 -\tau \right) -2\beta' \delta
_{0}\left( \tau_0 -\tau \right).
\end{align*}%
If $\beta ^{\prime }\beta \left( 1-\tau \right) -2\beta' \delta
_{0}\left( \tau_0 -\tau \right) \geq 0$, Assumption \ref{A-discontinuity} is satisfied with probability
approaching one as $n\rightarrow \infty$. 

Suppose not.
Then, we must have that 
$\beta ^{\prime }\beta \left( 1-\tau \right)  < 2\beta' \delta
_{0}\left( \tau_0 -\tau \right) = 2 \beta_1 c_0 \left( \tau_0 -\tau \right)$
for some nonzero $\beta_1$, which is the first element of $\beta$.
Hence, 
\begin{align*}
\beta ^{\prime }\beta   < 2 |\beta_1| |c_0| \frac{\tau_0 -\tau}{ 1-\tau}
\end{align*}
Now note that
\begin{equation*}
\left\vert \beta _{1}\right\vert \leq \frac{\beta ^{\prime }\beta }{%
\left\vert \beta _{1}\right\vert }\leq 2 |c_0| \frac{\tau_0 -\tau}{ 1-\tau}
\leq  |c_0| \frac{0.7}{ 0.85} \;
\text{ for any $\tau \in \mathbb{T} = [0.15,0.85]$},
\end{equation*}%
which implies that 
\begin{align*}
\mathbb{E} \left\Vert f_{\alpha ,\tau }-f_{0}\right\Vert _{n}^{2}
&\geq \left(
c_{0}-\beta _{1}\right) ^{2}\left( \tau_0 -\tau \right) +
\beta ^{\prime}\beta (1-\tau) - \beta_1^2 \left( \tau_0 -\tau \right) \\
&\geq \left( \frac{0.15}{0.85}\right) ^{2}c_{0}^{2}\left( \tau
-\tau _{0}\right),
\end{align*}%
where the last inequality follows from the simple fact that $\beta ^{\prime}\beta (1-\tau) \geq \beta_1^2 \left( \tau_0 -\tau \right)$.

\section{Proofs for Section \ref{sec:consist}}\label{sec:proofs:consist}

In this section of the appendix, we prove the prediction consistency of our Lasso estimator. Let
\begin{align*}
V_{1j}& :=\left( n\sigma \left\Vert X^{(j)}\right\Vert _{n}\right)
^{-1}\sum_{i=1}^{n}U_{i}X_{i}^{(j)},  \\
V_{2j}(\tau )& :=\left( n\sigma \left\Vert X^{(j)}(\tau )\right\Vert
_{n}\right) ^{-1}\sum_{i=1}^{n}U_{i}X_{i}^{(j)}1\{Q_{i}<\tau \}.
\end{align*}%
For a constant $\mu \in (0,1)$, define the events
\begin{align*}
\mathbb{A}& :=\bigcap_{j=1}^{M}\left\{ 2|V_{1j}|\leq \mu \lambda /\sigma
\right\} , \\
\mathbb{B}& :=\bigcap_{j=1}^{M}\left\{ 2\sup_{\tau \in \mathbb{T}%
}|V_{2j}(\tau )|\leq \mu \lambda /\sigma \right\} ,
\end{align*}%
Also define $J_{0}:=J(\alpha _{0})$ and $%
R_{n}:=R_{n}(\alpha _{0},\tau _{0}),$ where%
\begin{equation*}
R_{n}(\alpha ,\tau ):=2n^{-1}\sum_{i=1}^{n}U_{i}X_{i}^{\prime }\delta
\left\{ 1(Q_{i}<\widehat{\tau })-1(Q_{i}<\tau )\right\} .
\end{equation*}

The following lemma gives some useful inequalities  that provide a basis for all our  theoretical results.

\begin{lem}[Basic Inequalities]
\label{lemma1} Conditional on the events $\mathbb{A}$ and $\mathbb{B}$, we
have
\begin{align}
\left\Vert \widehat{f}-f_{0}\right\Vert _{n}^{2}+\left( 1-\mu \right)
\lambda \left\vert \widehat{\mathbf{D}}(\widehat{\alpha }-\alpha
_{0})\right\vert _{1}& \leq 2\lambda \left\vert \widehat{\mathbf{D}}(%
\widehat{\alpha }-\alpha _{0})_{J_{0}}\right\vert _{1}  \label{imp-ineq-lem}
\\
& +\lambda \left\vert \left\vert \widehat{\mathbf{D}}{\alpha _{0}}%
\right\vert _{1}-\left\vert \mathbf{D}{\alpha _{0}}\right\vert
_{1}\right\vert +R_{n}  \notag
\end{align}%
and%
\begin{equation}
\left\Vert \widehat{f}-f_{0}\right\Vert _{n}^{2}+\left( 1-\mu \right)
\lambda \left\vert \widehat{\mathbf{D}}(\widehat{\alpha }-\alpha
_{0})\right\vert _{1}\leq 2\lambda \left\vert \widehat{\mathbf{D}}(\widehat{%
\alpha }-\alpha _{0})_{J_{0}}\right\vert _{1}+\left\Vert f_{(\alpha _{0},%
\widehat{\tau })}-f_{0}\right\Vert _{n}^{2}.  \label{imp-ineq-2-lem}
\end{equation}
\end{lem}

The basic inequalities in Lemma \ref{lemma1} involve more terms than
that of the linear model \citep[e.g. Lemma 6.1 of][]{BvdG:11} because our model in \eqref{model} includes the unknown threshold parameter $\tau_0$
and the weighted Lasso is considered in \eqref{Lasso-fixed-tau}.
Also, it helps prove our main results to have different upper bounds in
\eqref{imp-ineq-lem} and \eqref{imp-ineq-2-lem}
for the same lower bound.

\begin{proof}[\textbf{Proof of Lemma \protect\ref{lemma1}}]
Note that
\begin{align}\label{joint-Lasso-def}
\widehat{S}_{n}+\lambda \left\vert \widehat{\mathbf{D}}\widehat{\alpha }%
\right\vert _{1}\leq S_{n}(\alpha ,\tau )+\lambda \left\vert \mathbf{D}(\tau
)\alpha \right\vert _{1}
\end{align}
for all $(\alpha ,\tau )\in \mathbb{R}^{2M}\times \mathbb{T}.$ Now write
\begin{align*}
\lefteqn{\widehat{S}_{n}-S_{n}(\alpha ,\tau )} \\
& =n^{-1}\left\vert \mathbf{y}-\mathbf{X}(\widehat{\tau })\widehat{\alpha }%
\right\vert _{2}^{2}-n^{-1}\left\vert \mathbf{y}-\mathbf{X}(\tau )\alpha
\right\vert _{2}^{2} \\
& =n^{-1}\sum_{i=1}^{n}\left[ U_{i}-\left\{ \mathbf{X}_{i}(\widehat{\tau }%
)^{\prime }\widehat{\alpha }-\mathbf{X}_{i}(\tau _{0})^{\prime }\alpha
_{0}\right\} \right] ^{2}-n^{-1}\sum_{i=1}^{n}\left[ U_{i}-\left\{ \mathbf{X}%
_{i}(\tau )^{\prime }\alpha -\mathbf{X}_{i}(\tau _{0})^{\prime }\alpha
_{0}\right\} \right] ^{2} \\
& =n^{-1}\sum_{i=1}^{n}\left\{ \mathbf{X}_{i}(\widehat{\tau })^{\prime }%
\widehat{\alpha }-\mathbf{X}_{i}(\tau _{0})^{\prime }\alpha _{0}\right\}
^{2}-n^{-1}\sum_{i=1}^{n}\left\{ \mathbf{X}_{i}(\tau )^{\prime }\alpha -%
\mathbf{X}_{i}(\tau _{0})^{\prime }\alpha _{0}\right\} ^{2} \\
& \;\;\;-2n^{-1}\sum_{i=1}^{n}U_{i}\left\{ \mathbf{X}_{i}(\widehat{\tau }%
)^{\prime }\widehat{\alpha }-\mathbf{X}_{i}(\tau )^{\prime }\alpha \right\}
\\
& =\left\Vert \widehat{f}-f_{0}\right\Vert _{n}^{2}-\left\Vert f_{(\alpha
,\tau )}-f_{0}\right\Vert _{n}^{2} \\
& \;\;\;-2n^{-1}\sum_{i=1}^{n}U_{i}X_{i}^{\prime }(\widehat{\beta }-\beta
)-2n^{-1}\sum_{i=1}^{n}U_{i}\left\{ X_{i}^{\prime }\widehat{\delta }1(Q_{i}<%
\widehat{\tau })-X_{i}^{\prime }\delta 1(Q_{i}<\tau )\right\} .
\end{align*}%
Further, write the last term above as
\begin{align*}
\lefteqn{n^{-1}\sum_{i=1}^{n}U_{i}\left\{ X_{i}^{\prime }\widehat{\delta }%
1(Q_{i}<\widehat{\tau })-X_{i}^{\prime }\delta 1(Q_{i}<\tau )\right\} } \\
& =n^{-1}\sum_{i=1}^{n}U_{i}X_{i}^{\prime }(\widehat{\delta }-\delta
)1(Q_{i}<\widehat{\tau })+n^{-1}\sum_{i=1}^{n}U_{i}X_{i}^{\prime }\delta
\left\{ 1(Q_{i}<\widehat{\tau })-1(Q_{i}<\tau )\right\} .
\end{align*}%
Hence, (\ref{joint-Lasso-def}) can be written as
\begin{equation*}
\begin{split}
\left\Vert \widehat{f}-f_{0}\right\Vert _{n}^{2}& \leq \left\Vert f_{(\alpha
,\tau )}-f_{0}\right\Vert _{n}^{2}+\lambda \left\vert \mathbf{D}(\tau
)\alpha \right\vert _{1}-\lambda \left\vert \widehat{\mathbf{D}}\widehat{%
\alpha }\right\vert _{1} \\
& +2n^{-1}\sum_{i=1}^{n}U_{i}X_{i}^{\prime }(\widehat{\beta }-\beta
)+2n^{-1}\sum_{i=1}^{n}U_{i}X_{i}^{\prime }(\widehat{\delta }-\delta
)1(Q_{i}<\widehat{\tau }) \\
& +2n^{-1}\sum_{i=1}^{n}U_{i}X_{i}^{\prime }\delta \left\{ 1(Q_{i}<\widehat{%
\tau })-1(Q_{i}<\tau )\right\} .
\end{split}%
\end{equation*}%
Then on the events $\mathbb{A}$ and $\mathbb{B}$, we have
\begin{equation}
\begin{split}
\left\Vert \widehat{f}-f_{0}\right\Vert _{n}^{2}& \leq \left\Vert f_{(\alpha
,\tau )}-f_{0}\right\Vert _{n}^{2}+\mu \lambda \left\vert \widehat{\mathbf{D}%
}(\widehat{\alpha }-\alpha )\right\vert _{1} \\
& +\lambda \left\vert \mathbf{D}(\tau )\alpha \right\vert _{1}-\lambda
\left\vert \widehat{\mathbf{D}}\widehat{\alpha }\right\vert
_{1}+R_{n}(\alpha ,\tau )
\end{split}
\label{basic-ineq}
\end{equation}%
for all $(\alpha ,\tau )\in \mathbb{R}^{2M}\times \mathbb{T}.$

Note  the fact that
\begin{align}\label{sparsity-useful-fact}
\left\vert \widehat{\alpha }^{(j)}-\alpha _{0}^{(j)}\right\vert +\left\vert
\alpha _{0}^{(j)}\right\vert -\left\vert \widehat{\alpha }^{(j)}\right\vert
=0\text{ for $j\notin J_{0}$}.
\end{align}%
On the one hand, by \eqref{basic-ineq} (evaluating at $(\alpha ,\tau
)=(\alpha _{0},\tau _{0})$), on the events $\mathbb{A}$ and $\mathbb{B}$,
\begin{equation*}
\begin{split}
\lefteqn{\norm{ \widehat{f} - f_{0} }_{n}^{2}+\left( 1-\mu \right) \lambda
\left\vert \mathbf{\hat{D}}(\widehat{\alpha }-\alpha _{0})\right\vert _{1}}
\\
& \leq \lambda \left( \left\vert \widehat{\mathbf{D}}(\widehat{\alpha }%
-\alpha _{0})\right\vert _{1}+\left\vert \widehat{\mathbf{D}}\alpha
_{0}\right\vert _{1}-\left\vert \widehat{\mathbf{D}}\widehat{\alpha }%
\right\vert _{1}\right) \\
& +\lambda \left\vert \left\vert \widehat{\mathbf{D}}{\alpha _{0}}%
\right\vert _{1}-\left\vert \mathbf{D}{\alpha _{0}}\right\vert
_{1}\right\vert +R_{n}(\alpha _{0},\tau _{0}) \\
& \leq 2\lambda \left\vert \widehat{\mathbf{D}}(\widehat{\alpha }-\alpha
_{0})_{J_{0}}\right\vert _{1}+\lambda \left\vert \left\vert \widehat{\mathbf{%
D}}{\alpha _{0}}\right\vert _{1}-\left\vert \mathbf{D}{\alpha _{0}}%
\right\vert _{1}\right\vert +R_{n}(\alpha _{0},\tau _{0}),
\end{split}%
\end{equation*}%
which proves $\left( \ref{imp-ineq-lem}\right) $. On the other hand, again
by \eqref{basic-ineq} (evaluating at $(\alpha ,\tau )=(\alpha _{0},\widehat{%
\tau })$), on the events $\mathbb{A}$ and $\mathbb{B}$,
\begin{equation*}
\begin{split}
\lefteqn{\norm{ \widehat{f} - f_{0} }_{n}^{2}+\left( 1-\mu \right) \lambda
\left\vert \mathbf{\hat{D}}(\widehat{\alpha }-\alpha _{0})\right\vert _{1}}
\\
& \leq \lambda \left( \left\vert \widehat{\mathbf{D}}(\widehat{\alpha }%
-\alpha _{0})\right\vert _{1}+\left\vert \widehat{\mathbf{D}}\alpha
_{0}\right\vert _{1}-\left\vert \widehat{\mathbf{D}}\widehat{\alpha }%
\right\vert _{1}\right) +\left\Vert f_{(\alpha _{0},\widehat{\tau }%
)}-f_{0}\right\Vert _{n}^{2} \\
& \leq 2\lambda \left\vert \widehat{\mathbf{D}}(\widehat{\alpha }-\alpha
_{0})_{J_{0}}\right\vert _{1}+\left\Vert f_{(\alpha _{0},\widehat{\tau }%
)}-f_{0}\right\Vert _{n}^{2},
\end{split}%
\end{equation*}%
which proves $\left( \ref{imp-ineq-2-lem}\right) $.
\end{proof}

We now establish conditions under which $\mathbb{A}\cap \mathbb{B}$ has
probability close to one with a suitable choice of $\lambda $. 
Let $\Phi $ denote the cumulative distribution function of the standard normal.

\begin{lem}[Probability of $\mathbb{A}\cap \mathbb{B}$]
\label{lemma5}
Let $\{U_i: i=1,\ldots, n \}$ be independent and identically distributed as  $\mathbf{N}(0,\sigma ^{2})$.
Then
\begin{equation*}
\mathbb{P}\{\mathbb{A}\cap \mathbb{B}\}\geq 1-6M\Phi \left( -\frac{\mu \sqrt{%
nr_{n}}}{2\sigma }\lambda \right) .
\end{equation*}
\end{lem}

Recall that $r_n$ depends on the lower bound $t_0$ of the parameter space for $\tau_0$.
Suppose that $t_0$ is taken such that $t_0 < \min_{i=1,\ldots,n} Q_i$. Then
$\left\Vert X^{\left( j\right)}(t_{0})\right\Vert _{n} = 0$, and therefore, $r_n = 0$.
In this case, Lemma \ref{lemma5} reduces to $\mathbb{P}\{\mathbb{A}\cap \mathbb{B}\} \geq 1 - 3M$ regardless of $n$ and $\lambda$, hence resulting in a useless bound.
This illustrates a need for restricting the parameter space for $\tau_0$ (see Assumption \ref{assumption-main-1}).
%In practice, researchers tend to choose a sufficiently strict subset of the range of %observed values of the threshold variable.

\begin{proof}[\textbf{Proof of Lemma \ref{lemma5}}]
Since $U_i  \sim \mathbf{N}(0,\sigma ^{2})$,
\begin{equation*}
\mathbb{P}\{\mathbb{A}^{c}\}\leq \sum_{j=1}^{M}\mathbb{P}\left\{ \sqrt{n}%
|V_{1j}|>\mu \sqrt{n}\lambda /(2\sigma )\right\}
= 2M\Phi \left( -\frac{\mu
\sqrt{n}}{2\sigma}\lambda \right)
\leq 2M\Phi \left( -\frac{\mu
\sqrt{r_n n}}{2\sigma}\lambda \right) ,
\end{equation*}
where the last inequality follows from $0 < r_n \leq 1$.

Now consider the event $\mathbb{B}$. 
For the  simplify of notation, we assume
without loss of generality that $Q_{i}=i/n$
since  there is no tie among $Q_i$'s. 
Note that $\left\Vert X^{\left(
j\right) }(\tau )\right\Vert _{n}$ is monotonically increasing in $\tau $ and $%
\sum_{i=1}^{n}U_{i}X_{i}^{\left( j\right) }1\left\{ Q_{i}<\tau \right\} $
can be rewritten as a partial sum process by the rearrangement of $i$
according to the magnitude of $Q_{i}.$ 
To see the latter, given $\left\{ Q_{i}\right\} $, let $\ell $ be the
index $i$ such that $Q_{i}$ is the $\ell $-th smallest of $\left\{
Q_{i}\right\}.$
 Since $\left\{ U_{i}\right\} $ is an independent and identically distributed (i.i.d.) sequence and $%
Q_{i} $ is deterministic, $\left\{ U_{\ell }\right\} _{\ell =1,...,n}$ is
also an i.i.d. sequence. Furthermore, $U_{\ell }X_{\ell }^{\left( j\right) }$
is a sequence of independent and symmetric random variables as $U_{\ell }$
is Gaussian and $X$ is a deterministic design. Thus, it satisfies the
conditions for L\'{e}vy's
inequality \citep[see e.g. Proposition A.1.2 of][]{VW}.
Then, by L\'{e}vy's
inequality,%
\begin{align*}
\mathbb{P}\left\{ \sup_{\tau \in \mathbb{T}}\sqrt{n}\left\vert V_{2j}(\tau
)\right\vert >\mu \sqrt{n}\lambda /(2\sigma )\right\} & \leq \mathbb{P}\left\{ \sup_{1\leq s\leq n}\left\vert \frac{1}{\sigma \sqrt{n}}\sum_{i=1}^{s}U_{i}X_{i}^{\left( j\right) }\right\vert >\left\Vert X^{\left(j\right) }\left( t_{0}\right) \right\Vert _{n}\frac{\mu \sqrt{n}}{2\sigma }\lambda \right\} \\
& \leq 2\mathbb{P}\left\{ \sqrt{n}|V_{1j}|>\frac{\left\Vert X^{\left(j\right) }(t_{0})\right\Vert _{n}}{\left\Vert X^{\left( j\right)}\right\Vert _{n}}\frac{\mu \sqrt{n}}{2\sigma }\lambda \right\} .
\end{align*}
Therefore, we have
\begin{align*}
\mathbb{P}\{\mathbb{B}^{c}\}& \leq \sum_{j=1}^{M}\mathbb{P}\left\{
\sup_{\tau \in \mathbb{T}}\sqrt{n}|V_{2j}(\tau )|>\mu \sqrt{n}\lambda
/(2\sigma )\right\} \\
& \leq 4M\Phi \left( -\frac{\mu \sqrt{r_{n}n}}{2\sigma }\lambda \right) .
\end{align*}%
Since $\mathbb{P}\{\mathbb{A}\cap \mathbb{B}%
\}\geq 1-\mathbb{P}\{\mathbb{A}^{c}\}-\mathbb{P}\{\mathbb{B}^{c}\}$, we have
proved the lemma.
\end{proof}

We are ready to establish the prediction consistency of the Lasso estimator.
Recall that $X_{\max }
:=\max_{\tau, j} \big\{ \left\Vert \mathbf{X}^{(j)}(\tau)\right\Vert _{n},\ \ j=1,...,2M, \tau \in \mathbb{T} \big\}$ and 
$X_{\min }:=\min_{j} \big\{ \left\Vert \mathbf{X}^{(j)}(t_0)\right\Vert _{n},\\ \ \  j=1,...,2M \big\}$,
that $\alpha _{\max
} $ denotes the maximum value that all the elements of $\alpha $ can take in
absolute value, and  that Assumption \ref{assumption-main-1} implies that $r_n > 0$ and also $X_{\min} > 0$.

\begin{lem}[Consistency of the Lasso]\label{lemma2}
Let $(\widehat{\alpha },\widehat{\tau })$ be the Lasso estimator defined by \eqref{joint-max} with $\lambda$ given by \eqref{lambda-form}.  Then, with probability at least $1-\left( 3M\right)^{1-A^{2}\mu ^{2}/8}$, we have
\begin{align*}
\left\Vert \widehat{f}-f_{0}\right\Vert _{n}
& \leq \Big(6\lambda X_{\max }\alpha _{\max }\mathcal{M}%
(\alpha _{0})+2\mu \lambda X_{\max }\left\vert \delta _{0}\right\vert _{1} \Big)^{1/2}.
%&<  \Big(8 X_{\max } \alpha _{\max } \lambda \mathcal{M}(\alpha _{0})\Big)^{1/2}.
\end{align*}
\end{lem}

\begin{proof}[\textbf{Proof of Lemma \protect\ref{lemma2}}]
Note that
\begin{equation*}
R_{n}=2n^{-1}\sum_{i=1}^{n}U_{i}X_{i}^{\prime }\delta _{0}\left\{ 1(Q_{i}<%
\widehat{\tau })-1(Q_{i}<\tau _{0})\right\} .
\end{equation*}%
Then on the event $\mathbb{B}$,
\begin{align}
\begin{split}
|R_{n}|& \leq 2\mu \lambda \sum_{j=1}^{M}\left\Vert X^{(j)}\right\Vert
_{n}|\delta _{0}^{(j)}| \\
& \leq 2\mu \lambda X_{\max }\left\vert \delta _{0}\right\vert _{1}.
\end{split}
\label{rough-ineq3}
\end{align}%
Then, conditional on $\mathbb{A}\cap \mathbb{B}$, combining \eqref{rough-ineq3} with \eqref{imp-ineq-lem} gives
\begin{align}\label{lem2-conc}
\left\Vert \widehat{f}-f_{0}\right\Vert _{n}^{2}+\left( 1-\mu \right)
\lambda \left\vert \widehat{\mathbf{D}}(\widehat{\alpha }-\alpha
_{0})\right\vert _{1}\leq 6\lambda X_{\max }\alpha _{\max }\mathcal{M}%
(\alpha _{0})+2\mu \lambda X_{\max }\left\vert \delta _{0}\right\vert _{1}
\end{align}
since
\begin{align*}
\left\vert \widehat{\mathbf{D}}(\widehat{\alpha }-\alpha
_{0})_{J_{0}}\right\vert _{1}& \leq 2X_{\max }\alpha _{\max }\mathcal{M}%
(\alpha _{0}), \\
\left\vert \left\vert \widehat{\mathbf{D}}{\alpha _{0}}\right\vert
_{1}-\left\vert \mathbf{D}{\alpha _{0}}\right\vert _{1}\right\vert & \leq
2X_{\max }\left\vert \alpha _{0}\right\vert _{1}.
\end{align*}
Using the bound that
$ 2\Phi \left( -x\right)  \leq \exp \left( -x^{2}/2\right)$ for $x \geq (2/\pi)^{1/2}$
as in equation (B.4) of \citet{Bickel-et-al:09},
Lemma \ref{lemma5} with $\lambda$ given by \eqref{lambda-form} implies that the event $\mathbb{A}\cap \mathbb{B}$
occurs with probability at least $1-\left( 3M\right)^{1-A^{2}\mu ^{2}/8}$.
Then the lemma follows from \eqref{lem2-conc}.
\end{proof}

\begin{proof}[Proof of Theorem \ref{theorem-main-1}]
The proof follows immediately 
from combining Assumption \ref{assumption-main-1} with 
Lemma \ref{lemma2}.
In particular, for the value of $K_1$, note that since $\left\vert  \delta _{0}\right\vert _{1} \leq \alpha _{\max }\mathcal{M}(\alpha _{0})$,
\begin{align*}
6\lambda X_{\max }\alpha _{\max }\mathcal{M}(\alpha _{0})+2\mu \lambda
X_{\max }\left\vert \delta _{0}\right\vert _{1}
&\leq \lambda \mathcal{M}%
\left( \alpha _{0}\right) \left( 6X_{\max }\alpha _{\max }+2\mu X_{\max
}\alpha _{\max }\right) \\
&\leq \lambda \mathcal{M}%
\left( \alpha _{0}\right) \left( 6 C_2 C_1 +2 \mu C_2 C_1\right),
\end{align*}%
where the last inequality follows from 
Assumption \ref{assumption-main-1}.
Therefore, we set 
\begin{align*}
K_{1} &:=  \sqrt{ 2 C_{1} C_{2}(3+\mu)}.
\end{align*}%
\end{proof}

%\begin{rem}
%Regarding consistency of the Lasso, see, among others, Corollary 6.1 of \citet{BvdG:11} for high-dimensional linear models and Lemma 6.7 of \citet{BvdG:11} for general convex loss functions. If $r_n$ is bounded away from zero, then our result in Theorem \ref{lemma2} coincides with those of \citet{BvdG:11}.
%\end{rem}

\section{Proofs for Section \ref{sec:oracle1}}\label{sec:proofs:oracle1}

We first provide a lemma to derive an oracle inequality regarding the sparsity of the Lasso
estimator $\widehat{\alpha}$.

\begin{lem}[Sparsity of the Lasso]
\label{lemma4}
Conditional on the event $\mathbb{A}\bigcap \mathbb{%
B}$, we have
\begin{equation}
\mathcal{M}(\widehat{\alpha })\leq \frac{4\phi _{\max }}{\left( 1-\mu
\right) ^{2}\lambda ^{2}X_{\min }^{2}}\left\Vert \widehat{f}%
-f_{0}\right\Vert _{n}^{2}.  \label{lemma4ineq}
\end{equation}
\end{lem}

%Lemma \ref{lemma4}, combined with Lemma \ref{lemma3}, implies that $\mathcal{M}(\widehat{\alpha })$ is a just constant multiple of $\mathcal{M}(\alpha _{0})$ when $\delta_0=0$, where the constant depends on $L_1, X_{\max}, X_{\min}, \phi_{\max}, \mu$ and $\kappa$ (independent of $\lambda$).

\begin{proof}[\textbf{Proof of Lemma \ref{lemma4}}]
As in (B.6) of \citet{Bickel-et-al:09}, for each $\tau $, the necessary and sufficient condition
for $\widehat{\alpha }(\tau )$ to be the Lasso solution can be written in
the form
\begin{align*}
\frac{2}{n}[X^{(j)}]^{\prime }(\mathbf{y}-\mathbf{X}(\tau )\widehat{\alpha }%
(\tau ))& =\lambda \left\Vert X^{(j)}\right\Vert _{n}\text{sign}(\widehat{%
\beta }^{(j)}(\tau )) & \text{if $\widehat{\beta }^{(j)}(\tau )$}& \text{$%
\neq 0$} \\
\left\vert \frac{2}{n}[X^{(j)}]^{\prime }(\mathbf{y}-\mathbf{X}(\tau )%
\widehat{\alpha }(\tau ))\right\vert & \leq \lambda \left\Vert
X^{(j)}\right\Vert _{n} & \text{if $\widehat{\beta }^{(j)}(\tau )$}& \text{$%
=0$} \\
\frac{2}{n}[X^{(j)}(\tau )]^{\prime }(\mathbf{y}-\mathbf{X}(\tau )\widehat{%
\alpha }(\tau ))& =\lambda \left\Vert X^{(j)}(\tau )\right\Vert _{n}\text{%
sign}(\widehat{\delta }^{(j)}(\tau )) & \text{if $\widehat{\delta }%
^{(j)}(\tau )$}& \text{$\neq 0$} \\
\left\vert \frac{2}{n}[X^{(j)}(\tau )]^{\prime }(\mathbf{y}-\mathbf{X}(\tau )%
\widehat{\alpha }(\tau ))\right\vert & \leq \lambda \left\Vert X^{(j)}(\tau
)\right\Vert _{n} & \text{if $\widehat{\delta }^{(j)}(\tau )$}& \text{$=0$},
\end{align*}%
where $j=1,\ldots ,M$.

Note that conditional on events $\mathbb{A}$ and $\mathbb{B}$,
\begin{align*}
\left\vert \frac{2}{n}\sum_{i=1}^{n}U_{i}X_{i}^{(j)}\right\vert & \leq \mu \lambda \left\Vert
X^{(j)}\right\Vert _{n} \\
\left\vert\frac{2}{n}\sum_{i=1}^{n}U_{i}X_{i}^{(j)}1\{Q_{i}<\tau \}\right\vert & \leq \mu \lambda
\left\Vert X^{(j)}(\tau )\right\Vert _{n}
\end{align*}%
for any $\tau $, where $j=1,\ldots ,M$. Therefore,
\begin{align*}
\left\vert \frac{2}{n}[X^{(j)}]^{\prime }(\mathbf{X}(\tau _{0})\alpha _{0}-%
\mathbf{X}(\tau )\widehat{\alpha }(\tau ))\right\vert & \geq \left( 1-\mu
\right) \lambda \left\Vert X^{(j)}\right\Vert _{n} & \text{if $\widehat{%
\beta }^{(j)}(\tau )$}& \text{$\neq 0$} \\
\left\vert \frac{2}{n}[X^{(j)}(\tau )]^{\prime }(\mathbf{X}(\tau _{0})\alpha
_{0}-\mathbf{X}(\tau )\widehat{\alpha }(\tau ))\right\vert & \geq \left(
1-\mu \right) \lambda \left\Vert X^{(j)}(\tau )\right\Vert _{n} & \text{if $%
\widehat{\delta }^{(j)}(\tau )$}& \text{$\neq 0$}.
\end{align*}%
Using inequalities above, write
\begin{align*}
\lefteqn{\frac{1}{n^{2}}\left[ \mathbf{X}(\tau _{0})\alpha _{0}-\mathbf{X}(%
\widehat{\tau })\widehat{\alpha }\right] ^{\prime }\mathbf{X}(\widehat{\tau }%
)\mathbf{X}(\widehat{\tau })^{\prime }\left[ \mathbf{X}(\tau _{0})\alpha
_{0}-\mathbf{X}(\widehat{\tau })\widehat{\alpha }\right] } \\
& =\frac{1}{n^{2}}\sum_{j=1}^{M}\left\{ [X^{(j)}]^{\prime }[\mathbf{X}(\tau
_{0})\alpha _{0}-\mathbf{X}(\widehat{\tau })\widehat{\alpha }]\right\} ^{2}+%
\frac{1}{n^{2}}\sum_{j=1}^{M}\left\{ [X^{(j)}(\widehat{\tau })]^{\prime }[%
\mathbf{X}(\tau _{0})\alpha _{0}-\mathbf{X}(\widehat{\tau })\widehat{\alpha }%
]\right\} ^{2} \\
& \geq \frac{1}{n^{2}}\sum_{j:\widehat{\beta }^{(j)}\neq 0}\left\{
[X^{(j)}]^{\prime }[\mathbf{X}(\tau _{0})\alpha _{0}-\mathbf{X}(\widehat{%
\tau })\widehat{\alpha }]\right\} ^{2}+\frac{1}{n^{2}}\sum_{j:\widehat{%
\delta }^{(j)}\neq 0}\left\{ [X^{(j)}(\widehat{\tau })]^{\prime }[\mathbf{X}%
(\tau _{0})\alpha _{0}-\mathbf{X}(\widehat{\tau })\widehat{\alpha }]\right\}
^{2} \\
& \geq \frac{\left( 1-\mu \right) ^{2}\lambda ^{2}}{4}\left( \sum_{j:%
\widehat{\beta }^{(j)}\neq 0}\left\Vert X^{(j)}\right\Vert _{n}^{2}+\sum_{j:%
\widehat{\delta }^{(j)}\neq 0}\left\Vert X^{(j)}(\widehat{\tau })\right\Vert
_{n}^{2}\right) \\
& \geq \frac{\left( 1-\mu \right) ^{2}\lambda ^{2}}{4}X_{\min }^{2}\mathcal{M%
}\left( \hat{\alpha}\right).
\end{align*}%
To complete the proof, note that
\begin{align*}
\frac{1}{n^{2}}& \left[ \mathbf{X}(\tau _{0})\alpha _{0}-\mathbf{X}(\widehat{%
\tau })\widehat{\alpha }\right] ^{\prime }\mathbf{X}(\widehat{\tau })\mathbf{%
X}(\widehat{\tau })^{\prime }\left[ \mathbf{X}(\tau _{0})\alpha _{0}-\mathbf{%
X}(\widehat{\tau })\widehat{\alpha }\right] \\
& \leq \text{maxeig}(\mathbf{X}(\widehat{\tau })\mathbf{X}(\widehat{\tau }%
)^{\prime }/n)\left\Vert \widehat{f}-f_{0}\right\Vert _{n}^{2} \\
& \leq \phi _{\max }\left\Vert \widehat{f}-f_{0}\right\Vert _{n}^{2},
\end{align*}%
where $\text{maxeig}(\mathbf{X}(\widehat{\tau })\mathbf{X}(\widehat{\tau }%
)^{\prime }/n)$ denotes the largest eigenvalue of $\mathbf{X}(\widehat{\tau }%
)\mathbf{X}(\widehat{\tau })^{\prime }/n$.
\end{proof}

\begin{lem}\label{thm-case1}
Suppose that $\delta _{0}=0$. Let Assumption \ref{re-assump} hold with $\kappa
=\kappa (s,\frac{1+\mu}{1-\mu },\mathbb{T})$ for $\mu <1$, and $\mathcal{M}%
(\alpha _{0})\leq s\leq M$.
 %Let $U_{i}$ follow $N(0,\sigma ^{2})\ $and
Let $(\widehat{\alpha },\widehat{\tau })$ be the Lasso estimator defined by \eqref{joint-max} with $\lambda$ given by \eqref{lambda-form}.
%\begin{equation*}
%\lambda =A\sigma \Big(\frac{\log 3M}{nr_{n}}\Big)^{1/2}
%\end{equation*}%
%and $A>2\sqrt{2}/\mu .$
Then, with probability at least $1-\left( 3M\right)
^{1-A^{2}\mu ^{2}/8},$ we have
\begin{align*}
\left\Vert \widehat{f}-f_{0}\right\Vert _{n} &\leq \frac{2 A\sigma  X_{\max }}{\kappa }\left( \frac{\log 3M}{nr_{n}} s\right)
^{1/2}, \\
\left\vert \widehat{\alpha }-\alpha _{0}\right\vert _{1} &\leq \frac{%
4 A\sigma }{\left( 1-\mu \right) \kappa ^{2}}%
\frac{X_{\max }^{2}}{X_{\min }}\left(\frac{\log 3M}{nr_{n}}\right)^{1/2}s, \\
\mathcal{M}(\widehat{\alpha }) &\leq \frac{16 \phi _{\max } }{\left( 1-\mu \right) ^{2}\kappa ^{2}}\frac{X_{\max }^{2}}{X_{\min
}^{2}}s.
\end{align*}
\end{lem}

\begin{proof}[\textbf{Proof of Lemma \ref{thm-case1}}]
Note that $\delta_0 = 0$ implies $\left\Vert f_{(\alpha_0,\widehat{\tau})} - f_0 \right\Vert ^2=0$. Combining this with \eqref{imp-ineq-2-lem}, we have
\begin{align}
\left\Vert \widehat{f}-f_{0}\right\Vert _{n}^{2}+\left( 1-\mu \right)
\lambda \left\vert \widehat{\mathbf{D}}(\widehat{\alpha }-\alpha
_{0})\right\vert _{1}\leq 2 \lambda \left\vert \widehat{%
\mathbf{D}}(\widehat{\alpha }-\alpha _{0})_{J_{0}}\right\vert _{1},
\label{l3*1}
\end{align}
which implies that
\begin{equation*}
\left\vert \widehat{\mathbf{D}}(\widehat{\alpha }-\alpha
_{0})_{J_{0}^{c}}\right\vert _{1}\leq \frac{1+\mu}{1-\mu }\left\vert
\widehat{\mathbf{D}}(\widehat{\alpha }-\alpha _{0})_{J_{0}}\right\vert _{1}.
\end{equation*}%
This in turn allows us to apply Assumption \ref{re-assump}, specifically URE$(s,\frac{1+\mu}{1-\mu },\mathbb{T})$, to yield
\begin{equation}
\begin{split}
\kappa ^{2} \left\vert \widehat{\mathbf{D}}(\widehat{\alpha }-\alpha
_{0})_{J_{0}}\right\vert _{2}^{2}&  \leq \frac{1}{n}|\mathbf{X}(\widehat{\tau})
\widehat{\mathbf{D}}(\widehat{\alpha }-\alpha _{0})|_{2}^{2} \\
& =\frac{1}{n}(\widehat{\alpha }-\alpha _{0})^{\prime }\widehat{\mathbf{D}}
\mathbf{X}(\widehat{\tau})^{\prime }\mathbf{X}(\widehat{\tau})\widehat{\mathbf{D}}(%
\widehat{\alpha }-\alpha _{0}) \\
& \leq \frac{\max (\widehat{\mathbf{D}})^{2}}{n}(\widehat{\alpha }-\alpha
_{0})^{\prime }\mathbf{X}(\widehat{\tau})^{\prime }\mathbf{X}(\widehat{\tau})(\widehat{\alpha }-\alpha _{0}) \\
& =\max (\widehat{\mathbf{D}})^{2}  \left\Vert \widehat{f}-f_{0}\right\Vert _{n}^{2},
\end{split}
\label{l3*2}
\end{equation}%
where $\kappa =\kappa (s,\frac{1+\mu}{1-\mu },\mathbb{T})$ and the last equality is due to the assumption that $\delta_0 = 0$.

Combining \eqref{l3*1} with \eqref{l3*2} yields
\begin{align*}
\left\Vert \widehat{f}-f_{0}\right\Vert _{n}^{2}& \leq  2
\lambda \left\vert \widehat{\mathbf{D}}(\widehat{\alpha }-\alpha
_{0})_{J_{0}}\right\vert _{1} \\
& \leq  2  \lambda \sqrt{s }\left\vert \widehat{\mathbf{D}}(\widehat{\alpha }-\alpha
_{0})_{J_{0}}\right\vert _{2} \\
& \leq \frac{2 \lambda }{\kappa }\sqrt{s }\max (\widehat{\mathbf{D}}) \left\Vert \widehat{f}-f_{0}\right\Vert _{n}.
\end{align*}%
Then the first conclusion of the lemma follows immediately.

In addition, combining the arguments above with the first conclusion of the
lemma yields
\begin{align}\label{alphahat-derivation}
\begin{split}
\left\vert \widehat{\mathbf{D}}\left( \widehat{\alpha }-\alpha _{0}\right)\right\vert _{1}& =\left\vert \widehat{\mathbf{D}}(\widehat{\alpha }-\alpha_{0})_{J_{0}}\right\vert _{1}+\left\vert \widehat{\mathbf{D}}(\widehat{\alpha }-\alpha _{0})_{J_{0}^{c}}\right\vert _{1} \\
& \leq 2 \left( 1-\mu \right) ^{-1}\left\vert \widehat{\mathbf{D}}(\widehat{\alpha }-\alpha _{0})_{J_{0}}\right\vert _{1} \\
& \leq 2 \left( 1-\mu \right) ^{-1}\sqrt{s }\left\vert \widehat{\mathbf{D}}(\widehat{\alpha }-\alpha _{0})_{J_{0}}\right\vert _{2} \\
& \leq \frac{2}{\kappa \left( 1-\mu \right) } \sqrt{s } \max (\widehat{\mathbf{D}})\left\Vert \widehat{f}-f_{0}\right\Vert _{n} \\
& \leq \frac{4\lambda }{\left( 1-\mu \right)\kappa ^{2}}{s} X_{\max }^{2},
\end{split}
\end{align}
which proves the second conclusion of the lemma since
\begin{align}\label{min-eig-D}
\left\vert \widehat{\mathbf{D}}\left( \widehat{\alpha }-\alpha _{0}\right)
\right\vert _{1}\geq \min (\widehat{\mathbf{D}})\left\vert  \widehat{\alpha }-\alpha _{0} \right\vert _{1}.
\end{align}
Finally, the lemma follows by Lemma \ref{lemma4}
with the bound on  $\mathbb{P} (\mathbb{A}\cap \mathbb{B})$
as in the proof of Lemma \protect\ref{lemma2}.
\end{proof}

\begin{proof}[Proof of Theorem \ref{main-thm-case1}]
The proof follows immediately 
from combining Assumption \ref{assumption-main-1} with 
Lemma \ref{thm-case1}. In particular, the constant $K_2$ can be chosen as   
\begin{align*}
K_2 &\equiv  \max \left( 2AC_2, \frac{4A C_2^2}{(1-\mu)C_3},
\frac{16 C_2^2}{(1-\mu)^2 C_3^2} \right).
\end{align*}%
\end{proof}

\section{Proofs for Section \ref{sec:oracle3}}\label{sec:proofs:oracle3}

To simplify notation, in this section, we assume without loss of generality that $Q_{i}=i/n$. Then $\mathbb{T} = [t_0,t_1] \subset (0,1)$.
For some constant $\eta > 0$, define  an event%
\begin{equation*}
\mathbb{C}\left( \eta \right) := \left\{ \sup_{\left\vert \tau -\tau
_{0}\right\vert < \eta }\left\vert \frac{2}{n}\sum_{i=1}^{n}U_{i}X_{i}^{%
\prime }\delta _{0}\left[ 1\left( Q_{i}<\tau _{0}\right) -1\left( Q_{i}<\tau
\right) \right] \right\vert \leq \lambda \sqrt{\eta }\right\}.
\end{equation*}
Recall that $h_{n}\left( \eta \right) :=\left( \left( 2n\eta \right) ^{-1}\sum_{i=%
\max\{1, \left[ n\left( \tau _{0}-\eta \right) \right]\} }^{\min\{ \left[ n\left( \tau
_{0}+\eta \right) \right], n\} }\left( X_{i}^{\prime }\delta _{0}\right)
^{2}\right) ^{1/2}$.

The following lemma gives the lower bound of the probability of the event
$\mathbb{A}\cap \mathbb{B}\cap [\cap_{j=1}^m \mathbb{C}(\eta_j)]$
for a given $m$ and some positive constants $\eta _{1},...,\eta _{m}$.
 To deal with the event $\cap_{j=1}^m \mathbb{C}(\eta_j)$, an extra term is added to the lower bound of the probability, in comparison to Lemma \ref{lemma5}.

\begin{lem}
[Probability of $\mathbb{A}\cap \mathbb{B}\cap \{\cap_{j=1}^m \mathbb{C}(\eta_j)\}$]
\label{prob-ABC}
For a given $m$ and some positive constants $\eta _{1},...,\eta _{m}$ such that $h_{n}\left( \eta_j  \right) > 0$ for each $j=1,\ldots,m$,
\begin{equation*}
\mathbb{P}\left\{ \mathbb{A\bigcap B} \bigcap \left[ \bigcap_{j=1}^{m}\mathbb{C}\left( \eta
_{j}\right) \right] \right\} \geq 1-6M\Phi \left( -\frac{\mu \sqrt{nr_{n}}}{2\sigma }%
\lambda \right) -4\sum_{j=1}^{m}\Phi \left( -\frac{\lambda \sqrt{n}}{2\sqrt{2%
}\sigma h_{n}\left( \eta _{j}\right) }\right).
\end{equation*}%
\end{lem}

\begin{proof}[\textbf{Proof of Lemma \protect\ref{prob-ABC}}]
Given Lemma \ref{lemma5}, it remains to examine the probability of $\mathbb{C%
}\left( \eta _{j}\right) $. As in the proof of Lemma \ref{lemma5}, L\'{e}%
vy's inequality yields that%
\begin{eqnarray*}
\mathbb{P}\left\{ \mathbb{C}\left( \eta _{j}\right) ^{c}\right\}  &\leq &%
\mathbb{P}\left\{ \sup_{\left\vert \tau -\tau _{0}\right\vert \leq \eta
_{j}}\left\vert \frac{2}{n}\sum_{i=1}^{n}U_{i}X_{i}^{\prime }\delta _{0}%
\left[ 1\left( Q_{i}<\tau _{0}\right) -1\left( Q_{i}<\tau \right) \right]
\right\vert >\lambda \sqrt{\eta _{j}}\right\} . \\
&\leq &2\mathbb{P}\left\{ \left\vert \frac{2}{n}\sum_{i=\left[ n\left( \tau
_{0}-\eta _{j}\right) \right] }^{\left[ n\left( \tau _{0}+\eta _{j}\right) %
\right] }U_{i}X_{i}^{\prime }\delta _{0}\right\vert >\lambda \sqrt{\eta _{j}}%
\right\}  \\
\, &\leq &4\Phi \left( -\frac{\lambda \sqrt{n}}{2\sqrt{2}\sigma h_{n}\left(
\eta _{j}\right) }\right) .
\end{eqnarray*}%
Hence, we have proved the lemma since $\mathbb{P}\left\{ \mathbb{A\bigcap B}%
\bigcap \left[ \bigcap_{j=1}^{m}\mathbb{C}\left( \eta _{j}\right) \right] \right\} \geq 1-\mathbb{P}%
\{\mathbb{A}^{c}\}-\mathbb{P}\{\mathbb{B}^{c}\}-\sum_{j=1}^{m}\mathbb{P}\{%
\mathbb{C}\left( \eta _{j}\right) ^{c}\}$.
\end{proof}

The following lemma gives an upper bound of $\left\vert \widehat{\tau}-\tau _{0}\right\vert$ using only Assumption \ref{A-discontinuity}, conditional on the events $\mathbb{A}$ and $\mathbb{B}$.

\begin{lem}\label{lem-claim2}
Suppose that Assumption  \ref{A-discontinuity} holds.
Let 
\begin{align*}
\eta^\ast = \max \left\{ \min_{i}\left\vert Q_{i}-\tau_0\right\vert, c^{-1}\lambda\left( 6 X_{\max }\alpha _{\max }\mathcal{M}
(\alpha _{0})+  2\mu  X_{\max }\left\vert \delta _{0}\right\vert _{1}\right) \right\},
\end{align*}
where  $c$ is  the constant defined in Assumption \ref{A-discontinuity}.
Then conditional on the events $\mathbb{A}$ and $\mathbb{B}$,
\begin{align*}
\left\vert \widehat{\tau}-\tau _{0}\right\vert \leq \eta^\ast.
\end{align*}%
\end{lem}

\begin{proof}[\textbf{Proof of Lemma  \ref{lem-claim2}}]
As in the proof of Lemma \ref{lemma1}, we have,  on the events $\mathbb{A}$ and $\mathbb{B}$,
\begin{align}\label{pf-lem-claim2-a}
\begin{split}
\lefteqn{\widehat{S}_{n}-S_{n}(\alpha_0 ,\tau_0 )} \\
& =
\left\Vert \widehat{f}-f_{0}\right\Vert _{n}^{2}
- 2n^{-1}\sum_{i=1}^{n}U_{i}X_{i}^{\prime }(\widehat{\beta }-\beta_0) - 2n^{-1}\sum_{i=1}^{n}U_{i}X_{i}^{\prime }(\widehat{\delta }-\delta_0)1(Q_{i}<\widehat{\tau }) - R_n \\
&\geq
\left\Vert \widehat{f}-f_{0}\right\Vert _{n}^{2}
- \mu \lambda \left\vert \widehat{\mathbf{D}}(\widehat{\alpha }-\alpha_0 )\right\vert _{1}  - R_n.
\end{split}
\end{align}
Then using \eqref{sparsity-useful-fact}, on the events $\mathbb{A}$ and $\mathbb{B}$,
\begin{align}\label{contra-ineq}
\begin{split}
\lefteqn{\left[ \widehat{S}_{n} + \lambda \left\vert \widehat{\mathbf{D}}\widehat{\alpha }\right\vert _{1} \right] - \left[ S_{n}(\alpha_0 ,\tau_0 ) + \lambda \left\vert \mathbf{D}\alpha_0 \right\vert _{1} \right] } \\
&\geq
\left\Vert \widehat{f}-f_{0}\right\Vert _{n}^{2}
- \mu \lambda \left\vert \widehat{\mathbf{D}}(\widehat{\alpha }-\alpha_0 )\right\vert _{1}
- \lambda \left[ \left\vert \mathbf{D}\alpha_0 \right\vert _{1} - \left\vert \widehat{\mathbf{D}}\widehat{\alpha }\right\vert _{1} \right]  - R_n \\
&\geq
\left\Vert \widehat{f}-f_{0}\right\Vert _{n}^{2}
- 2 \lambda \left\vert \widehat{\mathbf{D}}(\widehat{\alpha }-\alpha_0 )_{J_0} \right\vert _{1}
- \lambda \left[ \left\vert \mathbf{D}\alpha_0 \right\vert _{1} - \left\vert \widehat{\mathbf{D}}\alpha_0 \right\vert _{1} \right]
 - R_n \\
&\geq
\left\Vert \widehat{f}-f_{0}\right\Vert _{n}^{2}
- \left[ 6\lambda X_{\max }\alpha _{\max }\mathcal{M}%
(\alpha _{0})+2\mu \lambda X_{\max }\left\vert \delta _{0}\right\vert _{1} \right],
\end{split}
\end{align}
where the  last inequality comes from 
\eqref{rough-ineq3} and following bounds:
\begin{align*}
2 \lambda \left\vert \widehat{\mathbf{D}}(\widehat{\alpha }-\alpha_0 )_{J_0} \right\vert _{1}
&\leq 4 \lambda  X_{\max}\alpha _{\max } \mathcal{M}\left( \alpha_0 \right), \\
 \lambda \left| \left\vert \mathbf{D}\alpha_0 \right\vert _{1} - \left\vert \widehat{\mathbf{D}}\alpha_0 \right\vert _{1} \right|
&\leq  2 \lambda X_{\max}\alpha _{\max }  \mathcal{M}\left( \alpha_0 \right).
\end{align*}%

Suppose now that $\left\vert \hat{\tau}-\tau _{0}\right\vert> \eta^\ast$.
Then Assumption \ref{A-discontinuity} and \eqref{contra-ineq} together imply that
\begin{align*}
\left[ \widehat{S}_{n} + \lambda \left\vert \widehat{\mathbf{D}}\widehat{\alpha }\right\vert _{1} \right] - \left[ S_{n}(\alpha_0 ,\tau_0 ) + \lambda \left\vert \mathbf{D}\alpha_0 \right\vert _{1} \right]
 > 0,
\end{align*}
which  leads to contradiction as $\widehat{\tau}$ is the minimizer of the
criterion function as in \eqref{joint-max}.
Therefore, we have proved the lemma.
\end{proof}

\begin{rem}
The nonasymptotic bound in Lemma \ref{lem-claim2} can be translated into the consistency of  $\widehat \tau$, as in Lemma \ref{lemma2}.
That is,  if $n \rightarrow \infty$, $M \rightarrow \infty$, and $\lambda \mathcal{M}(\alpha _{0})\rightarrow 0$,
Lemma \ref{lem-claim2} implies the consistency of $\widehat \tau$,
provided that
$X_{\max }$, $\alpha _{\max }$, and $c^{-1}$ are bounded uniformly in $n$
and $Q_i$ is continuously distributed.
\end{rem}

We now provide a lemma for bounding the prediction risk as well as the $\ell_1$ estimation loss for $\alpha_0$.

\begin{lem}
\label{lem-claim3} Suppose that $\left\vert \hat{\tau}-\tau _{0}\right\vert
\leq c_{\tau }$ and $\left\vert \hat{\alpha}-\alpha _{0}\right\vert _{1}\leq
c_{\alpha }$ for some $(c_{\tau },c_{\alpha })$. Suppose further that
Assumption \ref{A-smoothness} and Assumption \ref{re-assump} hold with $%
\mathbb{S} = \left\{ \left\vert \tau -\tau _{0}\right\vert \leq c_{\tau
}\right\} $, $\kappa =\kappa (s,\frac{2+\mu}{1-\mu },\mathbb{S})$ for $0<\mu <1$
and $\mathcal{M}(\alpha _{0})\leq s\leq M$. Then, conditional on $\mathbb{A}$%
, $\mathbb{B\ }$and $\mathbb{C}(c_{\tau })$, we have
\begin{align*}
\left\Vert \widehat{f}-f_{0}\right\Vert _{n}^{2}& \leq   3\lambda  \left\{\sqrt{c_{\tau }}+\left( 2 X_{\min } \right)^{-1} c_{\tau } C \left\vert \delta _{0}\right\vert _{1} \vee \frac{6X_{\max}^2}{\kappa^2}\lambda s \vee \frac{2  X_{\max}}{\kappa}  \left(     c_{\alpha}c_{\tau}C \deltanorm  s \right)^{1/2} \right\}, \\
\alphahatnorm& \leq  \frac{3}{(1-\mu)X_{\min}}  \left\{  \sqrt{c_{\tau }}+\left( 2 X_{\min } \right)^{-1} c_{\tau } C \deltanorm
\vee \frac{6X_{\max}^2}{\kappa^2}  \lambda s \vee  \frac{2X_{\max}}{\kappa} \left( c_{\alpha} c_{\tau}C \deltanorm s\right)^{1/2} \right\}.
\end{align*}
\end{lem}

Lemma \ref{lem-claim3} states the bounds for both $\left\Vert \widehat{f}-f_{0}\right\Vert _{n}$
and $\alphahatnorm$ may become smaller as $c_\tau$ gets smaller.
This is because decreasing $c_\tau$ reduces the first and third terms in the bounds directly,
and also because decreasing $c_\tau$ reduces the second term in the bound indirectly by allowing for a possibly larger $\kappa$ since $\mathbb{S}$ gets smaller.

\begin{proof}[\textbf{Proof of Lemma \protect\ref{lem-claim3}}]
Note that on $\mathbb{C}$,
\begin{eqnarray*}
\left\vert R_{n}\right\vert &=&\left\vert
2n^{-1}\sum_{i=1}^{n}U_{i}X_{i}^{\prime }\delta _{0}\left\{ 1(Q_{i}<\widehat{%
\tau })-1(Q_{i}<\tau _{0})\right\} \right\vert \\
&\leq &\lambda \sqrt{c_{\tau }}.
\end{eqnarray*}%
The triangular inequality, the mean value theorem (applied to $f\left( x\right) =\sqrt{x}$),
and Assumption \ref{A-smoothness} imply that
\begin{align}\label{Dhat-D}
\begin{split}
 \left\vert \left\vert \widehat{\mathbf{D}}{\alpha _{0}}\right\vert
_{1}-\left\vert \mathbf{D}{\alpha _{0}}\right\vert _{1}\right\vert
& =\left\vert \sum_{j=1}^{M}\left( \left\Vert X^{\left( j\right) }\left( \hat{\tau}\right)
\right\Vert _{n}-\left\Vert X^{\left( j\right)}\left( \tau _{0}\right)
\right\Vert _{n}\right) \left\vert \delta _{0}^{\left(
j\right) }\right\vert \right\vert  \\
& \leq \sum_{j=1}^{M}
\left(2 \left\Vert X^{\left( j\right) }\left( t_0 \right)\right\Vert _{n}\right)^{-1}
\left\vert \delta _{0}^{\left(
j\right) }\right\vert \frac{1}{n}\sum_{i=1}^{n}\left\vert X_{i}^{\left(
j\right) }\right\vert ^{2}\left\vert 1\left\{ Q_{i}<\hat{\tau}\right\}
-1\left\{ Q_{i}<\tau _{0}\right\} \right\vert   \\
& \leq \left( 2 X_{\min } \right)^{-1} c_{\tau } C \left\vert \delta _{0}\right\vert _{1}.
\end{split}
\end{align}
We now consider two cases: (i) $\left\vert \widehat{\mathbf{D}}(\widehat{%
\alpha }-\alpha _{0})_{J_{0}}\right\vert _{1}> \sqrt{c_{\tau }}+\left( 2 X_{\min } \right)^{-1} c_{\tau } C \left\vert \delta _{0}\right\vert _{1}$
and \\(ii) $\left\vert \widehat{\mathbf{D}}(\widehat{\alpha }-\alpha _{0})_{J_{0}}\right\vert _{1}\leq \sqrt{c_{\tau }}+\left( 2 X_{\min } \right)^{-1} c_{\tau } C \left\vert \delta _{0}\right\vert _{1}$.

\bigskip

\noindent \textbf{Case (i)}: In this case, note that
\begin{align*}
\lambda \left\vert \left\vert \widehat{\mathbf{D}}\alpha_0\right\vert_1 - \left\vert \mathbf{D}\alpha_0 \right\vert_1\right\vert +R_n &< \lambda (2X_{\min})^{-1}c_{\tau}C\deltanorm + \lambda \sqrt{c_{\tau}}\\
&=\lambda \left(\sqrt{c_{\tau}}+ (2X_{\min})^{-1}c_{\tau}C\deltanorm  \right)\\
&<\lambda \left\vert \widehat{\mathbf{D}}\left(\widehat{\alpha}-\alpha_0\right)_{J_0} \right\vert_1.
\end{align*}
Combining this result with \eqref{imp-ineq-lem}, we have
\begin{align}
\left\Vert \widehat{f}-f_{0}\right\Vert _{n}^{2}+\left( 1-\mu \right)
\lambda \left\vert \widehat{\mathbf{D}}(\widehat{\alpha }-\alpha
_{0})\right\vert _{1}& \leq 3\lambda \left\vert \widehat{\mathbf{D}}(%
\widehat{\alpha }-\alpha _{0})_{J_{0}}\right\vert _{1}, \label{lem7-eq1}
\end{align}
which implies
\begin{align*}
\left\vert \widehat{\mathbf{D}}(\widehat{\alpha }-\alpha _{0})_{J_{0}^c}\right\vert _{1} \leq \frac{2+\mu}{1-\mu} \left\vert \widehat{\mathbf{D}}(\widehat{\alpha }-\alpha _{0})_{J_{0}}\right\vert _{1}.
\end{align*}
Then, we apply Assumption \ref{re-assump} with URE$(s,\frac{2+\mu}{1-\mu },\mathbb{S})$. Note
that since it is assumed that $|\widehat\tau - \tau_0| \leq c_\tau$,  Assumption \ref{re-assump} only
needs to hold with $\mathbb{S}$ in the $c_\tau$ neighborhood of $\tau
_{0}$. Since $\delta_0 \neq 0$, \eqref{l3*2} now has an extra term
\begin{align*}
\lefteqn{ \kappa ^{2}\left\vert \widehat{\mathbf{D}}(\widehat{\alpha }-\alpha
_{0})_{J_{0}}\right\vert _{2}^{2} } \\
&\leq \max (\widehat{\mathbf{D}})^{2}\left\Vert \hat{f}-f_{0}\right\Vert
_{n}^{2} \\
&\;\;\; +\max (\widehat{\mathbf{D}})^{2}\frac{1}{n}\sum_{i=1}^{n} \bigg\{2\Big(  \mathbf{X}_i(\widehat{\tau})'\widehat{\alpha}  - \mathbf{X}_i(\widehat{\tau})'\alpha_0 \Big)
\Big( X_i'\delta_0  \big[ 1(Q_i<\tau_0) -1(Q_i<\widehat{\tau}) \big] \Big) \bigg\}\\
&\leq  \max (\widehat{\mathbf{D}})^{2} \left( \left\Vert \hat{f}-f_{0}\right\Vert
_{n}^{2} + 2 c_{\alpha} \left\vert  \delta_0\right\vert_1  \sup_{j}\Pn  \left\vert{X}_i^{(j)}\right\vert^2 \left\vert 1(Q_i<\tau_0) -1(Q_i<\widehat{\tau})  \right\vert \right) \\
%&\leq \max (\widehat{\mathbf{D}})^{2}\left( \left\Vert \hat{f}-f_{0}\right\Vert _{n}^{2}+\left( 1+2X_{\max }c_{\alpha }\right) C_{\ast}c_{\tau }\right)\\
&\leq  X_{\max}^2 \bigg(\left\Vert \widehat{f}-f_0\right\Vert_n^2 + 2c_{\alpha}c_{\tau}C \deltanorm \bigg),
\end{align*}%
where the last inequality is due to Assumption \ref{A-smoothness}. Combining this result with \eqref{lem7-eq1}, we have
\begin{align*}
\left\Vert \widehat{f}-f_0 \right\Vert_n^2
&\leq 3 \lambda \left\vert \widehat{\mathbf{D}}\left( \widehat{\alpha}-\alpha_0\right)_{J_0}\right\vert_1\\
& \leq 3 \lambda \sqrt{s}\left\vert \widehat{\mathbf{D}}\left( \widehat{\alpha}-\alpha_0\right)_{J_0}\right\vert_2\\
& \leq 3 \lambda \sqrt{s} \left(   \kappa^{-2}X_{\max}^2 \left(\left\Vert \widehat{f}-f_0\right\Vert_n^2 + 2 c_{\alpha}c_{\tau}C \deltanorm \right)  \right)^{1/2}.
\end{align*}
Applying $a+b \leq 2a \vee 2b$, we get the upper bound of $\fhatnorm$ on $\mathbb{A}$ and $\mathbb{B}$, as
\begin{align}\label{result1-fhatnorm}
\fhatnorm^2 \leq  \frac{18 X_{\max}^2}{\kappa^2} \lambda^2 s \vee   \frac{6  X_{\max}}{\kappa} \lambda\left(     c_{\alpha}c_{\tau}C \deltanorm  s \right)^{1/2}.
\end{align}

To derive the upper bound for $\left\vert \widehat{\alpha}-\alpha_0 \right\vert_1$, note that using the same arguments as in \eqref{alphahat-derivation},
\begin{align*}
\left\vert \widehat{D} (\widehat{\alpha}-\alpha_0 ) \right\vert_1&\leq \frac{3}{1-\mu}\left\vert \widehat{D}\left( \widehat{\alpha}-\alpha_0\right)_{J_0} \right\vert_1\\
& \leq
 \frac{3}{1-\mu}
\sqrt{s}\left\vert \widehat{D}\left( \widehat{\alpha}-\alpha_0\right)_{J_0}\right\vert_2\\
& \leq
 \frac{3}{1-\mu}
 \sqrt{s} \left(   \kappa^{-2}X_{\max}^2 \left(\left\Vert \widehat{f}-f_0\right\Vert_n^2 + 2  c_{\alpha}c_{\tau}C \deltanorm \right) \ \right)^{1/2}\\
&\leq \frac{3\sqrt{s}}{(1-\mu)\kappa} X_{\max} \left( \fhatnorm^2 +  2  c_{\alpha} c_{\tau}C \deltanorm \right)^{1/2}.
\end{align*}
Then combining the fact that $a+b \leq 2a \vee 2b$ with \eqref{min-eig-D} and \eqref{result1-fhatnorm} yields
\begin{align*}
\alphahatnorm &\leq \frac{18}{(1-\mu)\kappa^2} \frac{X_{\max}^2}{X_{\min}} \lambda s \vee \frac{6}{(1-\mu)\kappa} \frac{X_{\max}}{X_{\min}} \left( c_{\alpha} c_{\tau}C \deltanorm s\right)^{1/2}.
\end{align*}

\noindent \textbf{Case (ii)}: In this case, it follows directly from \eqref{imp-ineq-lem} that
\begin{align*}
\left\Vert \widehat{f}-f_{0}\right\Vert _{n}^{2}& \leq 3\lambda \left( \sqrt{c_{\tau }}+\left( 2 X_{\min } \right)^{-1} c_{\tau } C \left\vert \delta _{0}\right\vert _{1}\right) , \\
\left\vert \widehat{\alpha }-\alpha _{0} \right\vert
_{1}& \leq \frac{3}{(1-\mu
)X_{\min }}\left( \sqrt{c_{\tau }}+\left( 2 X_{\min } \right)^{-1} c_{\tau } C \left\vert \delta _{0}\right\vert _{1}\right),
\end{align*}%
which establishes the desired result.
\end{proof}

The following lemma shows that the bound for $\left\vert \hat{\tau}-\tau _{0}\right\vert$ can be further tightened if we combine results obtained in
Lemmas \ref{lem-claim2} and \ref{lem-claim3}.

\begin{lem}
\label{lem-claim4} Suppose that $\left\vert \hat{\tau}-\tau _{0}\right\vert
\leq c_{\tau }$ and $\left\vert \hat{\alpha}-\alpha _{0}\right\vert _{1}\leq
c_{\alpha }$ for some $(c_{\tau },c_{\alpha })$. Let $\tilde{\eta}%
:=c^{-1}\lambda \left( \left( 1+\mu \right) X_{\max }c_{\alpha }+\sqrt{c_{\tau }}
+(2X_{\min})^{-1}c_{\tau }C \left\vert \delta_0\right\vert_1\right)  $. If Assumption \ref{A-discontinuity}
holds, then conditional on the events $\mathbb{A}$, $\mathbb{B}$, and $%
\mathbb{C}(c_{\tau })$,
\begin{equation*}
\left\vert \hat{\tau}-\tau _{0}\right\vert \leq \tilde{\eta}.
\end{equation*}
\end{lem}

\begin{proof}[\textbf{Proof of Lemma \protect\ref{lem-claim4}}]
Note that on $\mathbb{A}$, $\mathbb{B\ }$and $\mathbb{C}$,
\begin{eqnarray*}
&&\left\vert \frac{2}{n}\sum_{i=1}^{n}\left[ U_{i}X_{i}^{\prime }\left( \hat{%
\beta}-\beta _{0}\right) +U_{i}X_{i}^{\prime }1\left( Q_{i}<\hat{\tau}%
\right) \left( \hat{\delta}-\delta _{0}\right) \right] \right\vert \\
&\leq &\mu \lambda X_{\max }\left\vert \hat{\alpha}-\alpha
_{0}\right\vert_{1}\leq \mu \lambda X_{\max }c_{\alpha }
\end{eqnarray*}%
and
\begin{equation*}
\left\vert \frac{2}{n}\sum_{i=1}^{n}U_{i}X_{i}^{\prime }\delta _{0}\left[
1\left( Q_{i}<\hat{\tau}\right) -1\left( Q_{i}<\tau _{0}\right) \right]
\right\vert \leq \lambda \sqrt{c_{\tau }}.
\end{equation*}

Suppose $\tilde{\eta}<\left\vert \hat{\tau}-\tau _{0}\right\vert <c_{\tau }.$
Then, as in \eqref{pf-lem-claim2-a},
\begin{equation*}
\widehat{S}_{n}-S_{n}(\alpha _{0},\tau _{0})\geq \left\Vert \widehat{f}%
-f_{0}\right\Vert _{n}^{2}-\mu \lambda X_{\max }c_{\alpha }-\lambda \sqrt{%
c_{\tau }}.
\end{equation*}%
Furthermore,  we obtain
\begin{align*}
\lefteqn{\left[ \widehat{S}_{n}+\lambda \left\vert \widehat{\mathbf{D}}
\widehat{\alpha }\right\vert _{1}\right] -\left[ S_{n}(\alpha _{0},\tau
_{0})+\lambda \left\vert \mathbf{D}\alpha _{0}\right\vert _{1}\right] }\\
& \geq \left\Vert \widehat{f}-f_{0}\right\Vert _{n}^{2}-\mu \lambda X_{\max
}c_{\alpha }-\lambda \sqrt{c_{\tau }} \\
& -\lambda \left(  \left\vert\widehat{\mathbf{D}}(\widehat{\alpha}  -\alpha_0 )\right\vert_1
+\left\vert (\widehat{\mathbf{D}}-  \mathbf{D})\alpha_0 \right\vert_1    \right) \\
& >c\tilde{\eta}-\left( \left( 1+\mu \right) X_{\max }c_{\alpha }+\sqrt{%
c_{\tau }} + (2X_{\min})^{-1}c_{\tau }C \left\vert\delta_0\right\vert_1\right) \lambda,
\end{align*}%
where the last inequality is due to Assumption \ref{A-discontinuity} and \eqref{Dhat-D}.

Since $c\tilde{\eta}=\left( \left( 1+\mu \right) X_{\max }c_{\alpha }+%
\sqrt{c_{\tau }}+(2X_{\min})^{-1}c_{\tau }C \left\vert\delta_0\right\vert_1\right) \lambda $, we again use the
contradiction argument as in the proof of Lemma \ref{lem-claim2} to establish the result.
\end{proof}

Lemma \ref{lem-claim3} provides us with
three different bounds for $\left\vert \hat{\alpha}-\alpha _{0}\right\vert
_{1}$ and the two of them are functions of $c_{\tau }$ and $c_{\alpha }.$
This leads us to apply Lemmas \ref{lem-claim3} and \ref{lem-claim4}
iteratively to tighten up the bounds. Furthermore, when the sample size is
large and thus $\lambda $ in \eqref{lambda-form} is small enough, we
show that the consequence of this chaining argument is that the bound for $%
\left\vert \hat{\alpha}-\alpha _{0}\right\vert $ is dominated by the middle
term in Lemma \ref{lem-claim3}. We give exact conditions for this on $%
\lambda $ and thus on the sample size $n.$ To do so, we first define
some constants:
\begin{align*}
A_{1 \ast}  := \frac{3\left( 1+\mu \right) X_{\max }}{\left( 1-\mu \right)X_{\min }}+1, %\\
A_{2 \ast}  := \frac{ C}{2cX_{\min }}, %\\
A_{3 \ast}  := \frac{6cX_{\max }^{2}}{\kappa ^{2}}, \ \ \text{and} \ \
A_{4 \ast}  := \frac{36 \left( 1+\mu \right) X_{\max }^{3}}{(1-\mu )^{2}  X_{\min}}.
\end{align*}

\begin{assm}[Inequality Conditions]
\label{tech-cond} The following inequalities hold:
\begin{align}
A_{1 \ast} A_{2 \ast} \lambda \left\vert \delta _{0}\right\vert_{1} &< 1,  \label{case1a} \\
\frac{A_{1 \ast} }{\left( 1-A_{1 \ast} A_{2 \ast} \lambda \left\vert \delta _{0}\right\vert_{1}\right) ^{2}} & < A_{3 \ast} s,  \label{case1b} \\
\left( 2\kappa ^{-2} A_{4 \ast} s +1\right) A_{2 \ast} \lambda \left\vert \delta _{0}\right\vert _{1} &< 1,  \label{case2a} \\
  \frac{ A_{2 \ast} \lambda \left\vert \delta _{0}\right\vert_{1} }{\left[ 1-\left( 2\kappa ^{-2} A_{4 \ast} s +1\right) A_{2 \ast} \lambda \left\vert \delta _{0}\right\vert _{1}
\right] ^{2}} &< \frac{(1-\mu)c}{4},  \label{case2b} \\
  \left[ 1-\left( 2\kappa ^{-2} A_{4 \ast} s +1\right) A_{2 \ast} \lambda \left\vert \delta _{0}\right\vert _{1}
\right]^{-2} &< A_{1 \ast}A_{3 \ast} s.  \label{case2c}
\end{align}
\end{assm}

\begin{rem}\label{tech-remark}
It would be easier to satisfy Assumption \ref{tech-cond} when the sample size $n$ is large. To appreciate Assumption \ref{tech-cond} in a setup when $n$ is large, suppose that
(1) $n \rightarrow \infty$, $M \rightarrow \infty$, $s \rightarrow \infty$, and $\lambda \rightarrow 0$; (2) $\left\vert \delta _{0}\right\vert_{1}$ may or may not diverge to infinity;
(3)   $X_{\min}$, $X_{\max }$,  $\kappa$,  $c$, $C$, and $\mu$ are independent of $n$. Then  conditions in Assumption \ref{tech-cond} can hold simultaneously for all sufficiently large $n$, provided that $s \lambda  \left\vert \delta _{0}\right\vert_{1} \rightarrow 0$.
\end{rem}

We now give the main result of this section.

\begin{lem}
\label{main-thm-fixed-threshold} Suppose that Assumption \ref{re-assump}
hold with $\mathbb{S} = \left\{ \left\vert \tau -\tau _{0}\right\vert \leq
\eta ^{\ast }\right\} $, $\kappa =\kappa (s,\frac{2+\mu}{1-\mu },\mathbb{S})$
for $0<\mu <1,$ and $\mathcal{M}(\alpha _{0})\leq s\leq M$. In addition,  Assumptions \ref{A-discontinuity}, \ref{A-smoothness}, and \ref{tech-cond} hold. %Let $U_{i}$ follow $N(0,\sigma ^{2})\ $and
Let $(\widehat{\alpha },\widehat{\tau })$ be the Lasso estimator defined by \eqref{joint-max} with $\lambda$ given by \eqref{lambda-form}.
%\begin{equation*}
%\lambda =A\sigma \Big(\frac{\log 3M}{nr_{n}}\Big)^{1/2}
%\end{equation*}%
%and $A>2\sqrt{2}/\mu $.
Then, there exists a sequence of constants $\eta
_{1},...,\eta _{m^{\ast }}$ for some finite $m^{\ast }$ such that
 $h_{n}\left( \eta_j  \right) > 0$ for each $j=1,\ldots,m^{\ast}$,
with
probability at least $1-\left( 3M\right) ^{1-A^{2}\mu
^{2}/8}-4\sum_{j=1}^{m^{\ast }}\left( 3M\right) ^{-A^{2}/(16r_{n}h_{n}\left(
\eta _{j}\right) )},$ we have
\begin{align*}
\left\Vert \widehat{f}-f_{0}\right\Vert _{n}& \leq \frac{3A\sigma X_{\max }}{%
\kappa }\left( \frac{2\log 3M}{nr_{n}}s\right) ^{1/2}, \\
\left\vert \widehat{\alpha }-\alpha _{0}\right\vert _{1}& \leq \frac{%
18A\sigma }{\left( 1-\mu \right) \kappa ^{2}}\frac{X_{\max }^{2}}{X_{\min }}%
\left( {\frac{\log 3M}{nr_{n}}}\right) ^{1/2}s, \\
\left\vert \hat{\tau}-\tau _{0}\right\vert &\leq \left( \frac{3\left( 1+\mu
\right) X_{\max }}{(1-\mu )X_{\min }}+1\right) \frac{6X_{\max
}^{2}A^{2}\sigma ^{2}}{c\kappa ^{2}}\frac{\log 3M}{nr_{n}}s, \\
\mathcal{M}\left( \hat{\alpha}\right) &\leq \frac{36\phi _{\max }}{%
\left( 1-\mu \right) ^{2}\kappa ^{2}}\frac{X_{\max }^{2}}{X_{\min }^{2}}s.
\end{align*}
\end{lem}

\begin{rem}
It is interesting to compare the URE$(s,c_0,\mathbb{S})$ condition assumed in Lemma \ref{main-thm-fixed-threshold}
with that in Lemma \ref{thm-case1}.
For the latter, the entire parameter space $\mathbb{T}$ is taken to be $\mathbb{S}$ but with a smaller constant $c_0 = (1+\mu)/(1-\mu)$.
Hence, strictly speaking, it is undetermined which URE$(s,c_0,\mathbb{S})$ condition is less stringent. It is possible to reduce $c_0$ in Lemma \ref{main-thm-fixed-threshold}
to a smaller constant but larger than $(1+\mu)/(1-\mu)$ by considering a more general form,  e.g. $c_0 = (1+\mu + \nu)/(1-\mu)$ for a positive constant $\nu$, but we have chosen $\nu = 1$ here
for readability.
\end{rem}

\begin{proof}[\textbf{Proof of Lemma \protect\ref{main-thm-fixed-threshold}}]
Here we use a chaining argument by iteratively applying Lemmas \ref%
{lem-claim3} and \ref{lem-claim4} to tighten the bounds for the prediction
risk and the estimation errors in $\hat{\alpha}$ and $\hat{\tau}.$

Let $c_{\alpha }^{\ast }$ and $c_{\tau }^{\ast }$ denote the bounds given in
the statement of the lemma for $\left\vert \hat{\alpha}-\alpha _{0}\right\vert _{1}$
and $\left\vert \hat{\tau}-\tau _{0}\right\vert ,$ respectively. Suppose
that
\begin{equation}
\sqrt{c_{\tau }}+\left( 2X_{\min }\right) ^{-1}c_{\tau }C\left\vert \delta
_{0}\right\vert _{1}\vee \frac{6X_{\max }^{2}}{\kappa ^{2}}\lambda s\vee
\frac{2X_{\max }}{\kappa }\left( c_{\alpha }c_{\tau }C\left\vert \delta
_{0}\right\vert _{1}s\right) ^{1/2}=\frac{6X_{\max }^{2}}{\kappa ^{2}}%
\lambda s.  \label{12-1}
\end{equation}%
This implies due to Lemma \ref{lem-claim3} that $\left\vert \hat{\alpha}%
-\alpha _{0}\right\vert _{1}$ is bounded by $c_{\alpha }^{\ast }$ and thus
achieves the bounds in the lemma given the choice of $\lambda $. The same
argument applies for $\left\Vert \hat{f}-f_{0}\right\Vert _{n}^{2}.$ The
equation \eqref{12-1} also implies in conjunction with Lemma %
\ref{lem-claim4} with $c_{\alpha }=c_{\alpha }^{\ast }$ that
\begin{eqnarray}
\left\vert \hat{\tau}-\tau _{0}\right\vert &\leq &c^{-1}\lambda \left(
\left( 1+\mu \right) X_{\max }c_{\alpha }^{\ast }+\sqrt{c_{\tau }}+(2X_{\min
})^{-1}c_{\tau }C\left\vert \delta _{0}\right\vert _{1}\right)  \notag \\
&\leq &\left( \frac{3\left( 1+\mu \right) X_{\max }}{(1-\mu )X_{\min }}%
+1\right) \frac{6X_{\max }^{2}}{c\kappa ^{2}}\lambda ^{2}s,  \label{taust}
\end{eqnarray}%
which is $c_{\tau }^{\ast }$. Thus, it remains to show that there is
convergence in the iterated applications of Lemmas \ref{lem-claim3} and \ref%
{lem-claim4} toward the desired bounds when \eqref{12-1} does not hold and
the number of iteration is finite.

Let $c_{\tau }^{\left( m\right) }$ and $c_{\alpha }^{\left( m\right) }$, respectively, denote the bounds for $\left\vert \hat{\alpha}%
-\alpha _{0}\right\vert _{1}$ and $\left\vert \hat{\tau}-\tau
_{0}\right\vert$ in the $m$-th iteration.
In view of  \eqref{lem2-conc} and Lemma \ref{lem-claim2}, we start the
iteration with
\begin{align*}
c_{\alpha }^{\left( 1\right) } &:=\frac{8X_{\max }\alpha _{\max }}{\left(
1-\mu \right) X_{\min }}s, \\
c_{\tau }^{\left( 1\right) } &:=c^{-1}8X_{\max }\alpha _{\max }\lambda s.
\end{align*}%
If the starting values $c_{\alpha }^{\left( 1\right) }$ and $c_{\tau }^{\left( 1\right) }$ are smaller than the desired bounds, we do not start the iteration.
Otherwise, we stop the iteration as soon as updated bounds are smaller than the desired bounds.

Since Lemma \ref%
{lem-claim3} provides us with two types of bounds for $c_{\alpha }$ when $%
\left( \ref{12-1}\right) $ is not met, we evaluate each case below.

\noindent \textbf{Case (i)}:
\begin{equation*}
c_{\alpha }^{\left( m\right) }=\frac{3}{(1-\mu )X_{\min }}\left( \sqrt{%
c_{\tau }^{\left( m-1\right) }}+\left( 2X_{\min }\right) ^{-1}c_{\tau
}^{\left( m-1\right) }C\left\vert \delta _{0}\right\vert _{1}\right) .
\end{equation*}%
This implies by Lemma \ref{lem-claim4} that
\begin{align*}
c_{\tau }^{\left( m\right) } &= c^{-1}\lambda \left( \left( 1+\mu \right)
X_{\max }c_{\alpha }^{\left( m\right) }+\sqrt{c_{\tau }^{\left( m-1\right) }}%
+(2X_{\min })^{-1}c_{\tau }^{\left( m-1\right) }C\left\vert \delta
_{0}\right\vert _{1}\right) \\
&= c^{-1}\lambda \left( \frac{3\left( 1+\mu \right) X_{\max }}{\left( 1-\mu
\right) X_{\min }}+1\right) \left( \sqrt{c_{\tau }^{\left( m-1\right) }}%
+\left( 2X_{\min }\right) ^{-1}C\left\vert \delta _{0}\right\vert
_{1}c_{\tau }^{\left( m-1\right) }\right) \\
&=: A_{1}\sqrt{c_{\tau }^{\left( m-1\right) }}+A_{2}c_{\tau }^{\left(
m-1\right) },
\end{align*}%
where $A_{1}$ and $A_{2}$ are defined accordingly. This system has one
converging fixed point other than zero if $A_{2}<1,$ which is the case under
\eqref{case1a}. Note also that all the terms here are
positive. After some algebra, we get the fixed point%
\begin{eqnarray*}
c_{\tau }^{\infty } &=&\left( \frac{A_{1}}{1-A_{2}}\right) ^{2} \\
&=&\left( \frac{c^{-1}\lambda \left( \frac{3\left( 1+\mu \right) X_{\max }}{%
\left( 1-\mu \right) X_{\min }}+1\right) }{1-c^{-1}\lambda \left( \frac{%
3\left( 1+\mu \right) X_{\max }}{\left( 1-\mu \right) X_{\min }}+1\right)
\left( 2X_{\min }\right) ^{-1}C\left\vert \delta _{0}\right\vert _{1}}%
\right) ^{2}.
\end{eqnarray*}%
Furthermore, \eqref{case1b} implies that%
\begin{equation*}
\sqrt{c_{\tau }^{\infty }}+\left( 2X_{\min }\right) ^{-1}c_{\tau }^{\infty
}C\left\vert \delta _{0}\right\vert _{1}<\frac{6X_{\max }^{2}}{\kappa ^{2}}%
\lambda s,
\end{equation*}%
which in turn yields that
\begin{equation*}
c_{\alpha }^{\infty }=\frac{3}{(1-\mu )X_{\min }}\left( \sqrt{c_{\tau
}^{\infty }}+\left( 2X_{\min }\right) ^{-1}c_{\tau }^{\infty }C\left\vert
\delta _{0}\right\vert _{1}\right) <c_{\alpha }^{\ast },
\end{equation*}%
and that $c_{\tau }^{\infty }<c_{\tau }^{\ast }$ by construction of $c_{\tau
}^{\ast }$ in \eqref{taust}.

\noindent \textbf{Case (ii)}:
Consider the case that
\begin{align*}
c_{\alpha }^{\left( m\right) } &= \frac{6X_{\max }}{(1-\mu )X_{\min }\kappa }%
\left( c_{\alpha }^{\left( m-1\right) }c_{\tau }^{\left( m-1\right)
}C\left\vert \delta _{0}\right\vert _{1}s\right) ^{1/2}=: B_{1}\sqrt{%
c_{\alpha }^{\left( m-1\right) }}\sqrt{c_{\tau }^{\left( m-1\right) }}.
\end{align*}
where $B_{1}$ is defined accordingly.
Again, by Lemma \ref{lem-claim4}, we have that
\begin{align*}
c_{\tau }^{\left( m\right) } &= c^{-1}\lambda \left( \left( 1+\mu \right)
X_{\max }c_{\alpha }^{\left( m\right) }+\sqrt{c_{\tau }^{\left( m-1\right) }}%
+(2X_{\min })^{-1}c_{\tau }^{\left( m-1\right) }C\left\vert \delta
_{0}\right\vert _{1}\right) \\
&=\left( \frac{\lambda \left( 1+\mu \right) 6X_{\max }^{2}\left(
C\left\vert \delta _{0}\right\vert _{1}s\right) ^{1/2}}{c(1-\mu )X_{\min
}\kappa }\sqrt{c_{\alpha }^{\left( m-1\right) }}+\frac{\lambda }{c}\right)
\sqrt{c_{\tau }^{\left( m-1\right) }}+\frac{\lambda C\left\vert \delta
_{0}\right\vert _{1}}{c2X_{\min }}c_{\tau }^{\left( m-1\right) } \\
&=:\left( B_{2}\sqrt{c_{\alpha }^{\left( m-1\right) }}+B_{3}\right) \sqrt{%
c_{\tau }^{\left( m-1\right) }}+B_{4}c_{\tau }^{\left( m-1\right) },
\end{align*}
by defining $B_{2},B_{3}$, and $B_{4}$ accordingly. As above this system has one
fixed point%
\begin{eqnarray*}
c_{\tau }^{\infty } &=&\left( \frac{B_{3}}{1-B_{1}B_{2}-B_{4}}\right) ^{2} \\
&=&\left( \frac{\lambda /c}{1-\left( \frac{\left( 1+\mu \right) 72X_{\max
}^{3}}{(1-\mu )^{2}X_{\min }\kappa ^{2}}s+1\right) \frac{C\left\vert \delta
_{0}\right\vert _{1}}{c2X_{\min }}\lambda }\right) ^{2}
\end{eqnarray*}%
and
\begin{equation*}
c_{\alpha }^{\infty }=B_{1}^{2}c_{\tau }^{\infty }=\left( \frac{6X_{\max }}{%
(1-\mu )X_{\min }\kappa }\right) ^{2}C\left\vert \delta _{0}\right\vert
_{1}sc_{\tau }^{\infty },
\end{equation*}%
provided that $B_{1}B_{2}+B_{4}<1$, which is true under
under \eqref{case2a}.
Furthermore, the fixed points $c_{\alpha
}^{\infty }$ and $c_{\tau }^{\infty }$ of this system is strictly smaller
than $c_{\alpha }^{\ast }$ and $c_{\tau }^{\ast },$ respectively, under \
\eqref{case2b} and \eqref{case2c}.

Since we have shown that $c_{\tau }^{\infty }<c_{\tau }^{\ast }$ and $%
c_{\alpha }^{\infty }<c_{\alpha }^{\ast }$ in both cases and $c_{\tau
}^{\left( m\right) }$ and $c_{\alpha }^{\left( m\right) }$ are strictly
decreasing as $m$ increases, the bound in the lemma is reached within
a finite number, say $m^{\ast },$ of iterative applications of Lemma \ref%
{lem-claim3} and \ref{lem-claim4}. Therefore, for each case, we have shown
that $\left\vert \widehat{\alpha }-\alpha _{0}\right\vert _{1}\leq c_{\alpha
}^{\ast }$ and $\left\vert \hat{\tau}-\tau _{0}\right\vert \leq c_{\tau
}^{\ast }$. The bound for the prediction risk can be obtained similarly, and
then the bound for the sparsity of the Lasso estimator follows from Lemma %
\ref{lemma4}.
Finally, each application of Lemmas \ref{lem-claim3} and \ref{lem-claim4} in
the chaining argument requires conditioning on $\mathbb{C}\left( \eta
_{j}\right) $, $j=1,...,m^{\ast }.$
\end{proof}

\begin{proof}[Proof of Theorem \ref{main-text-thm-fixed-threshold}]
The proof follows immediately 
from combining Assumptions \ref{assumption-main-1} 
and \ref{h-regular} 
with 
Lemma \ref{main-thm-fixed-threshold}.
In particular, the constants $C_4$, $C_5$ and $K_3$   can be chosen as   
\begin{align*}
C_4 &\equiv 4m^\ast, \\
C_5 &\equiv \frac{A^2}{16 \max_{j=1,\ldots,m^\ast} h_n(\eta_j)}, \\
K_3 &\equiv  \max \left( 
3 \sqrt{2} AC_2, 
\frac{18A C_2^2}{(1-\mu)C_3},
\left\{ \frac{3 ( 1+\mu) C_2 }{(1-\mu )C_3}+1\right\} 
\frac{6C_2^{2}A^{2}}{c},
\frac{36 C_2^2}{(1-\mu)^2 C_3^2} \right). 
\end{align*}
\end{proof}

\section{Additional Numerical Results}\label{sec:mc:extra}

In  
Table \ref{tb:resultM2}, we report additional 
empirical results with  $Q_{i}$ being the literacy rate.
The model selection and estimation results are similar to the case when
$Q_i$ is the initial GDP.

In this section, we  also consider different simulation designs allowing correlation between covariates.
The $M$-dimensional vector $X_i$ is now generated from a multivariate normal $N(0,\Sigma)$ with $(\Sigma)_{i,j}=\rho^{|i-j|}$, where $(\Sigma)_{i,j}$ denotes the (i,j) element of the $M \times M$ covariance matrix $\Sigma$. 
All other random variables are the same as above. 
We have conducted the simulation studies for both $\rho=0.1$ and $0.3$; however,  Table \ref{tb:M50D} and Figures \ref{fg:depX1}--\ref{fg:depX2} report only the results of $\rho=0.3$  since the results with $\rho = 0.1$ are similar.
They show very similar results as those in Section \ref{sec:MC}.

Figure \ref{fg:cov} shows frequencies of selecting true parameters when both $\rho=0$ and $\rho=0.3$.
When $\rho = 0$, the probability that the Lasso estimates include the true nonzero parameters is very high. In most cases, the probability is 100\%,  and even the lowest probability is as high as 98.25\%.
When $\rho = 0.3$, the corresponding probability is somewhat lower than the no-correlation case, but it is still high and the lowest value is 80.75\%.

\begin{table}[htbp]
\small
\begin{center}
\caption{Model Selection and Estimation Results with $Q={lr}$}
\label{tb:resultM2}
\begin{tabular}{cccccc}
\\
\hline \hline
 && \multirow{2}{*}{Linear Model}  && \multicolumn{2}{c}{Threshold Model}\\
 && & &\multicolumn{2}{c}{$\widehat{\tau}=82$}\\
\hline
 && & & $ \widehat{\beta}$ & $ \widehat{\delta}$ \\
\cline{5-6}
\textit{const.} 			&& -0.1086 	&& -0.0151 & - \\
\textit{lgdp60} 		&& -0.0159 	&& -0.0099 & - \\
$\textit{ls}_k$			&& 0.0038 	&& 0.0046 & - \\
\textit{syrm60}		&& 0.0069	&& - & - \\
\textit{hyrm60}		&& 0.0188	&& 0.0101 & - \\
\textit{prim60}		&&-0.0001	&& -0.0001 & - \\
\textit{pricm60}		&& 0.0002	&& 0.0001 & 0.0001 \\
\textit{seccm60}		&& 0.0004	&& - & 0.0018\\
\textit{llife}			&& 0.0674	&& 0.0335&- \\
\textit{lfert}			&&-0.0098	&& -0.0069& - \\
\textit{edu/gdp}		&& -0.0547	&& - & - \\
\textit{gcon/gdp}		&&-0.0588	&& -0.0593& - \\
\textit{revol}			&&-0.0299	&& - & - \\
\textit{revcoup} 		&&0.0215	&& - & - \\
\textit{wardum}		&&-0.0017	&& - & - \\
\textit{wartime}		&&-0.0090	&& -0.0231 & -  \\
\textit{lbmp}			&&-0.0161	&& -0.0142 & - \\
\textit{tot}			&&0.1333	&& 0.0846 & - \\
$\textit{lgdp60} \times\textit{hyrf60}$ 	&& -0.0014 && - &  -0.0053\\
$\textit{lgdp60} \times\textit{nof60}$		&& $1.49\times 10^{-5}$ && - & - \\
$\textit{lgdp60} \times\textit{prif60}$		&&$-1.06 \times 10^{-5}$   && - & $-2.66 \times 10^{-6}$\\
$\textit{lgdp60} \times\textit{seccf60}$	&& -0.0001   && - & - \\
\hline
$\lambda$ & & $0.0011$  & & \multicolumn{2}{c}{$0.0044$}\\
$\mathcal{M}(\widehat{\alpha})$ && 22 && \multicolumn{2}{c}{16}\\

$\#\ \textit{of covariates}$ & &47 & & \multicolumn{2}{c}{94}\\
$\#\  \textit{of observations}$ && 70 && \multicolumn{2}{c}{70}\\
%$R^2$ && 0.82 && \multicolumn{2}{c}{0.80}\\
%$\widetilde{R}^2$ && 0.86 && \multicolumn{2}{c}{0.85}\\
%$adj.~R^2$ && 0.75 && \multicolumn{2}{c}{0.74}\\
\hline
\\
\multicolumn{6}{p{.8\textwidth}}{\footnotesize \emph{Note:} The regularization parameter $\lambda$ is chosen by the `leave-one-out'  cross validation method. $\mathcal{M}(\widehat{\alpha})$ denotes the number of covariates to be selected by the Lasso estimator, and `-' indicates that the regressor is not selected. Recall that $\widehat{\beta}$ is the coefficient when $Q \ge \widehat{\gamma}$ and that $\widehat{\delta}$ is the change of the coefficient value when $Q < \widehat{\gamma}$.}
\end{tabular}
\end{center}
\end{table}

% Dependent Tables rho=0.3

%\begin{sidewaystable}[htbp]
\begin{table}[htbp]
\footnotesize

\begin{center}
\caption{Simulation Results with $M=50$ and $\rho=0.3$}
\label{tb:M50D}
\begin{tabular}{lp{2cm}crrrrrr}
\hline \hline
{Threshold}                 &  {Estimation}   &Constant          &   \multicolumn{3}{c}{Prediction Error (PE)}  & \multirow{2}{*}{$\mathbb{E}\left[\mathcal{M}\left(\widehat{\alpha}\right)\right]$}&   \multirow{2}{*}{$\mathbb{E}\left|\widehat{\alpha}-\alpha_0\right|_1$} &    \multirow{2}{*}{$\mathbb{E}\left|\widehat{\tau}-\tau_0\right|_1$}  \\
\cline{4-6}
Parameter &  Method & for $\lambda$ & Mean & Median & SD & & \\
\hline
& & & & & & &\\
\multicolumn{9}{c}{\underline{Jump Scale: $c=1$}} \\
& & & & & & &\\
\multirow{7}{*}{$\tau_0=0.5$} & {Least Squares}  & None & 0.283  &  0.273  &  0.075  &  100.00   &  7.718 & 0.010 \\
 & \multirow{4}{*}{Lasso}   &$A=2.8$ & 0.075  &  0.043  &  0.087   &12.99   &0.650  &0.041 \\
 & & $A=3.2$ & 0.108  &  0.059  &  0.115   &  10.98  & 0.737  & 0.071\\
 & & $A=3.6$ & 0.160  &  0.099  &  0.137  &  9.74   &0.913  &0.119 \\
 & & $A=4.0$ & 0.208  &  0.181  &  0.143  &   8.72  & 1.084  &0.166 \\
 & Oracle 1  & None & 0.013  &  0.006  &  0.017  &  4.00  & 0.169 &0.005 \\
 & Oracle 2  & None & 0.005  &  0.004  &  0.004  & 4.00  & 0.163 &0.000 \\
& & & & & & &\\
\hline
& & & & & & &\\
\multirow{7}{*}{$\tau_0=0.4$} & {Least Squares}  & None & 0.317  &  0.297  &  0.099   &  100.00   &7.696  & 0.010 \\
 & \multirow{4}{*}{Lasso} &$A=2.8$ & 0.118  &  0.063  &  0.123   & 13.89   &0.855  & 0.094\\
 & & $A=3.2$ & 0.155  &  0.090  &  0.139   & 11.69  &0.962  &0.138 \\
 & & $A=3.6$ & 0.207  &  0.201  &  0.143   & 10.47   &1.150  &0.204 \\
 & & $A=4.0$ & 0.258  &  0.301  &  0.138   & 9.64  &1.333  & 0.266\\
 & Oracle 1  & None & 0.013  &  0.007  &  0.016   &4.00  &0.168  &0.006 \\
 & Oracle 2  & None & 0.005  &  0.004  &  0.004   &4.00  & 0.163 &0.000 \\
& & & & & & &\\
\hline
& & & & & & &\\
\multirow{7}{*}{$\tau_0=0.3$} & {Least Squares}  & None & 1.639  &  0.487  &  7.710   & 100.00   & 12.224 & 0.015\\
 & \multirow{4}{*}{Lasso}  &$A=2.8$ & 0.149  &  0.080  &  0.136  & 14.65    &1.135  & 0.184\\
 & & $A=3.2$ & 0.200  &  0.233  &  0.138   & 12.71  &1.346  & 0.272\\
 & & $A=3.6$ & 0.246  &  0.284  &  0.127   & 11.29  &1.548  &0.354 \\
 & & $A=4.0$ & 0.277  &  0.306  &  0.116   &  10.02 & 1.673 &0.408\\
 & Oracle 1  & None & 0.013  &  0.006  &  0.017 &4.00  & 0.182  &0.005 \\
 & Oracle 2  & None & 0.005  &  0.004  &  0.004  &4.00  &0.176 &0.000 \\
& & & & & & &\\
\hline
& & & & & & &\\
\multicolumn{9}{c}{\underline{Jump Scale: $c=0$}} \\
& & & & & & &\\
\multirow{6}{*}{N/A} & {Least Squares}  & None & 6.939  &  0.437  &  42.698 & 100.00  & 23.146 &\multirow{6}{*}{N/A} \\
 & \multirow{4}{*}{Lasso}  &$A=2.8$ & 0.012  &  0.011  &  0.007   & 9.02& 0.248& \\
 & & $A=3.2$ & 0.013  &  0.011  &  0.008     &6.54 &0.214& \\
 & & $A=3.6$ & 0.014  &  0.013  &  0.009    & 5.00 &0.196 & \\
 & & $A=4.0$ & 0.016  &  0.014  &  0.010  &  3.83 & 0.191& \\
 & Oracle 1 \& 2 & None & 0.002  &  0.002  &  0.003    &  2.00 & 0.054 & \\
% & Oracle 2  & None & 0.002  &  0.002  &  0.003    & &  &\\
& & & & & & &\\
\hline
\\
\multicolumn{9}{p{\textwidth}}{\footnotesize \emph{Note: }$M$ denotes the column size of $X_i$ and $\tau$ denotes the threshold parameter. Oracle 1 \& 2 are estimated by the least squares when sparsity is known and when sparsity and $\tau_0$ are known, respectively. All simulations are based on 400 replications of a sample with 200 observations. }
\end{tabular}
\end{center}
%\end{sidewaystable}
\end{table}

\clearpage

\begin{figure}[tbp]
\caption{Mean Prediction Errors and Mean $\mathcal{M}(\widehat{\alpha})$ when $\rho=0.3$}
\begin{center}

%\includegraphics[width = 2.8in]{plot_PE_M50d.pdf}
%\includegraphics[width = 2.8in]{plot_Malpha_M50d.pdf} \\[0pt]
%$M=50$\\[0pt]
\includegraphics[width = 2.8in]{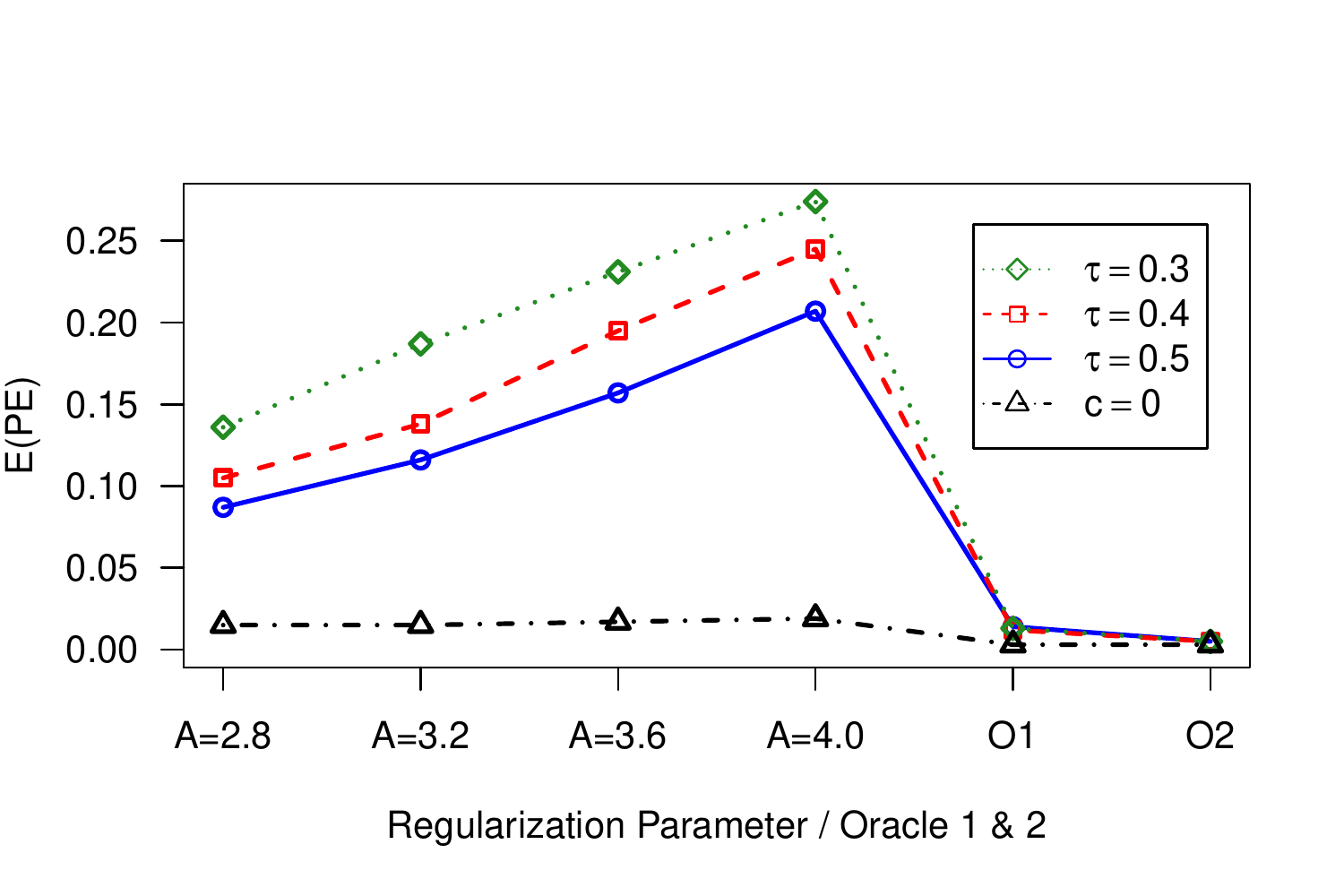}
\includegraphics[width = 2.8in]{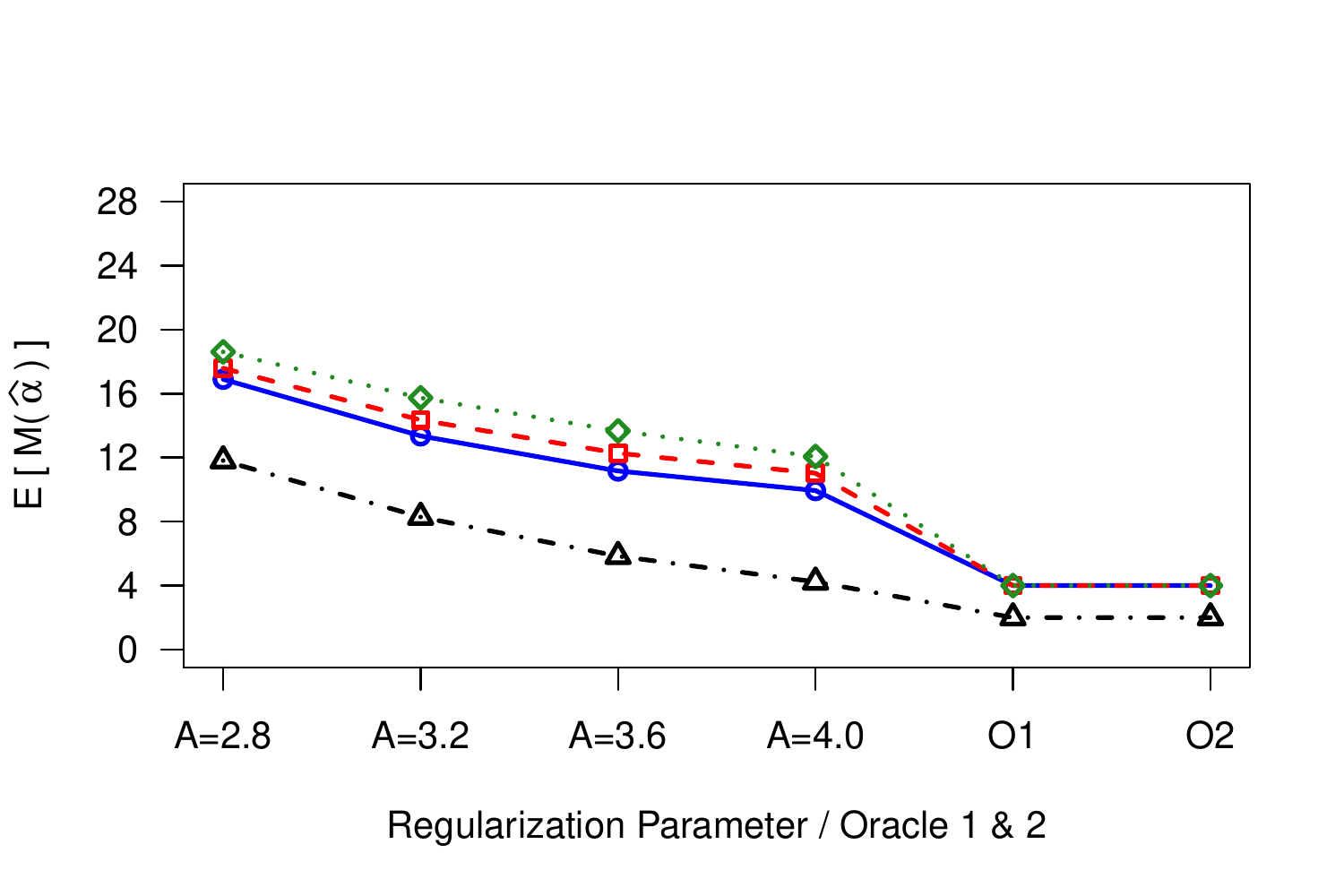} \\[0pt]
$M=100$\\[0pt]
\includegraphics[width = 2.8in]{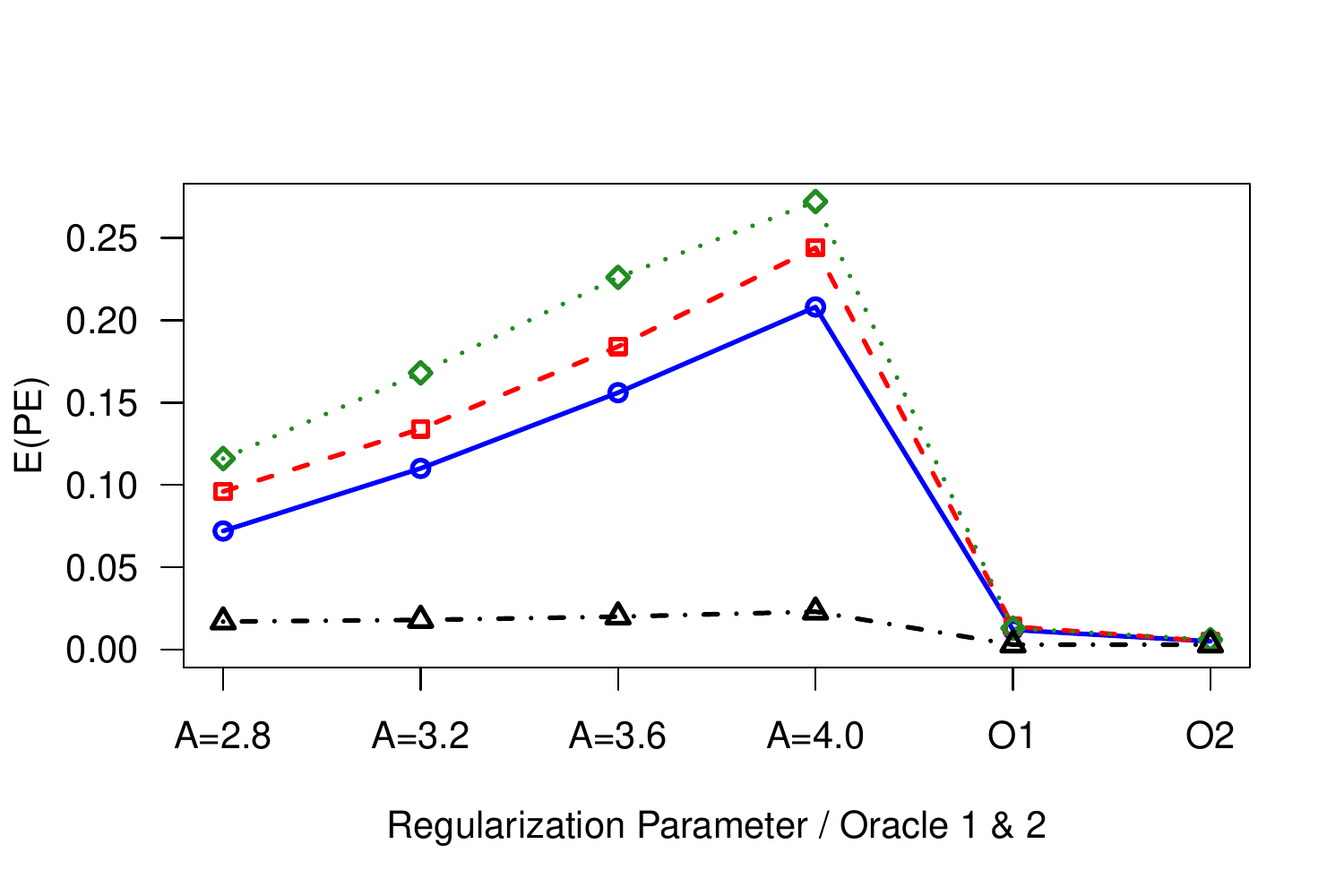}
\includegraphics[width = 2.8in]{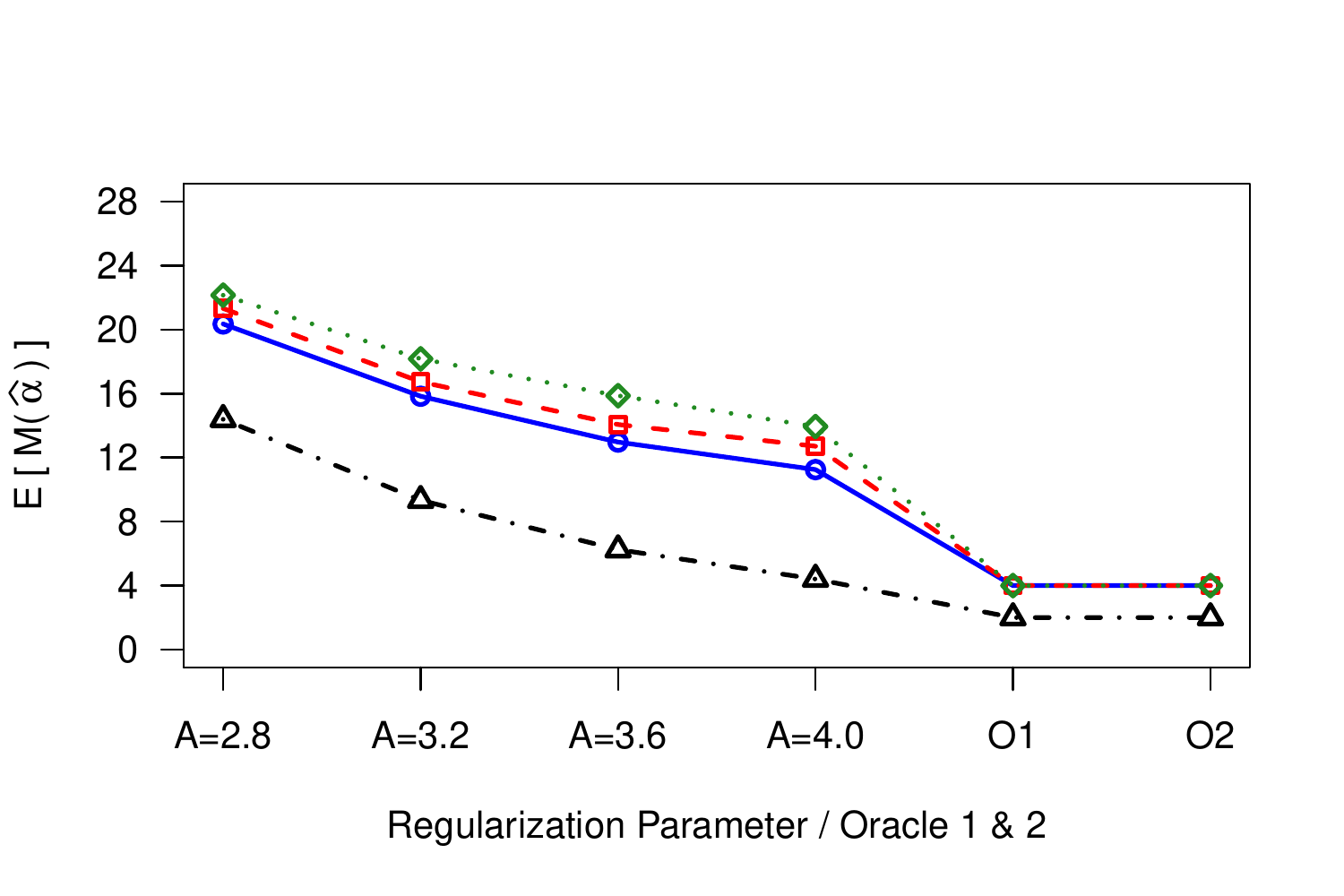} \\[0pt]
$M=200$\\[0pt]
\includegraphics[width = 2.8in]{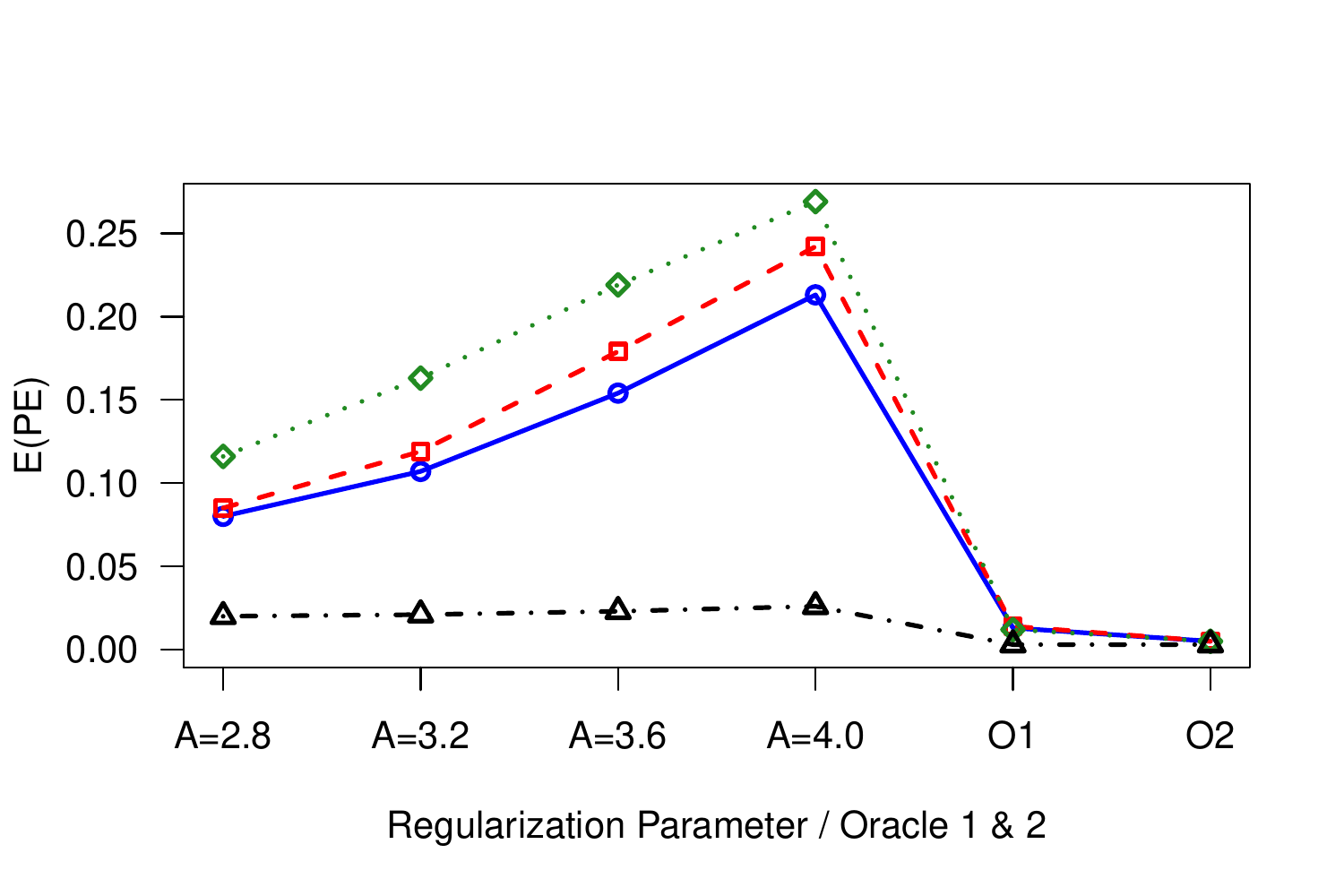}
\includegraphics[width = 2.8in]{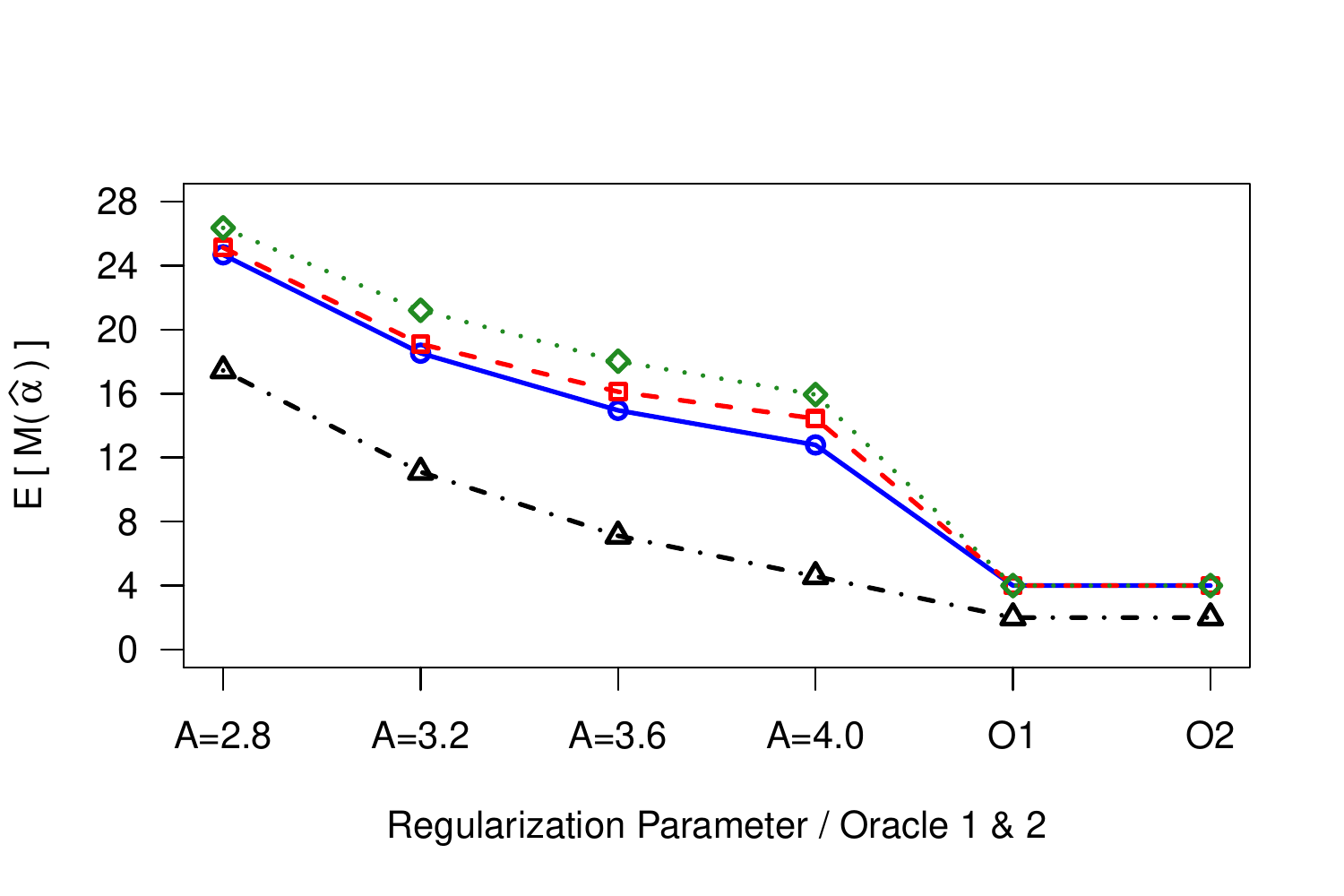} \\[0pt]
$M=400$\\[0pt]
\end{center}
\par
\label{fg:depX1}
\end{figure}

\begin{figure}[tbp]
\caption{Mean $\ell_1$-Errors for $\alpha$ and $\tau$ when $\rho=0.3$}
\begin{center}

%\includegraphics[width = 2.8in]{plot_l1_M50d.pdf}
%\includegraphics[width = 2.8in]{plot_l1_tau_M50d.pdf} \\[0pt]
%$M=50$\\[0pt]
\includegraphics[width = 2.8in]{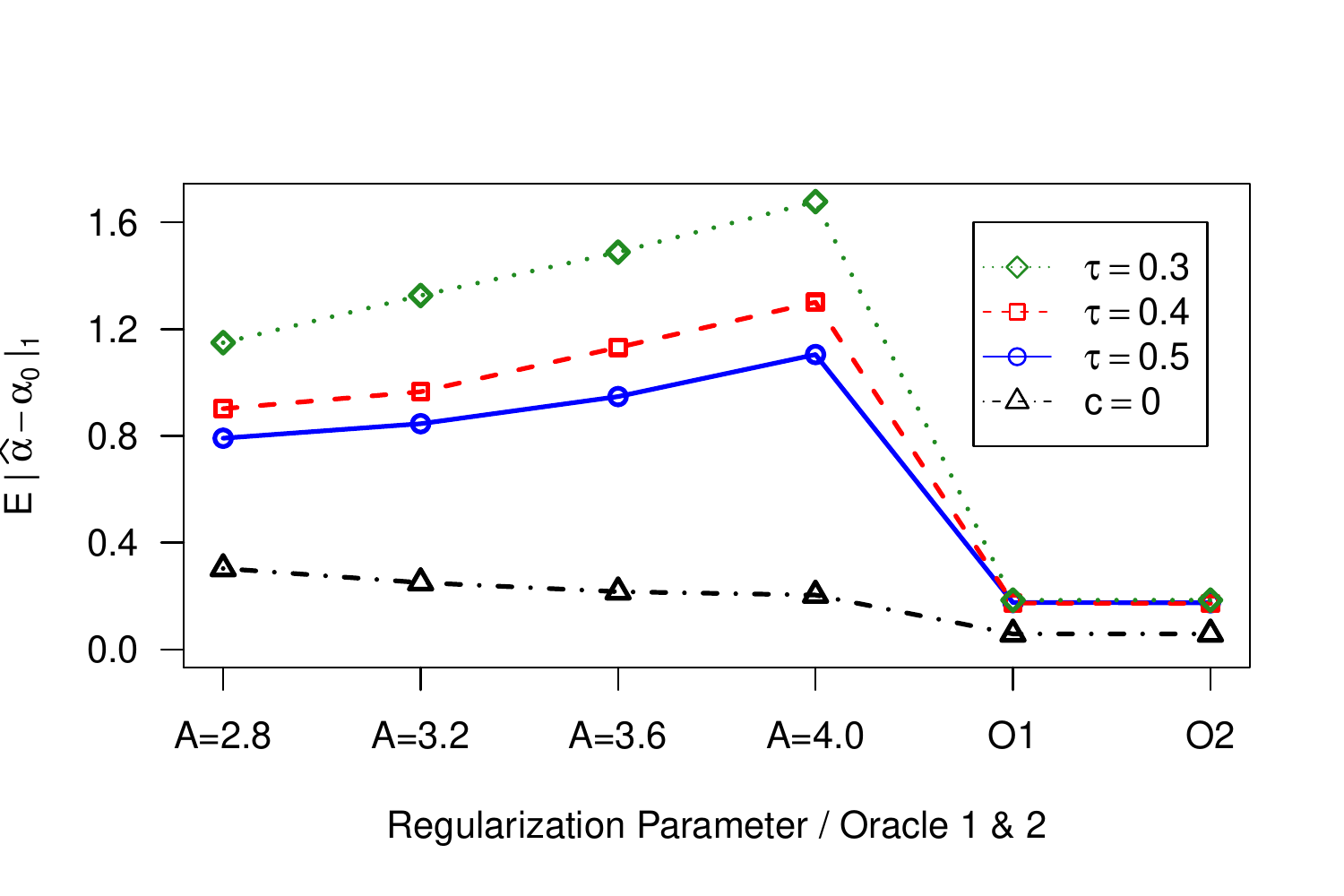}
\includegraphics[width = 2.8in]{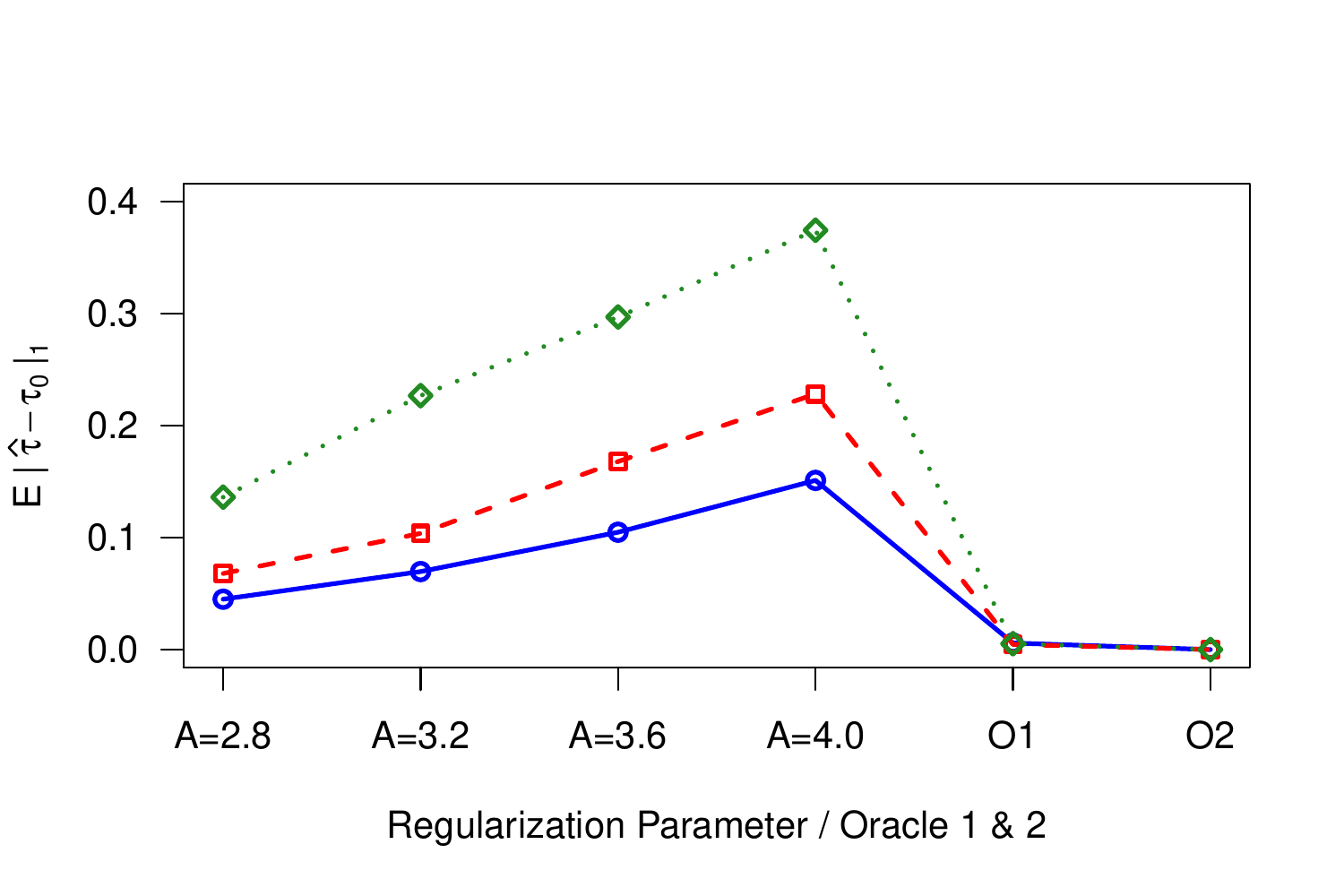} \\[0pt]
$M=100$\\[0pt]
\includegraphics[width = 2.8in]{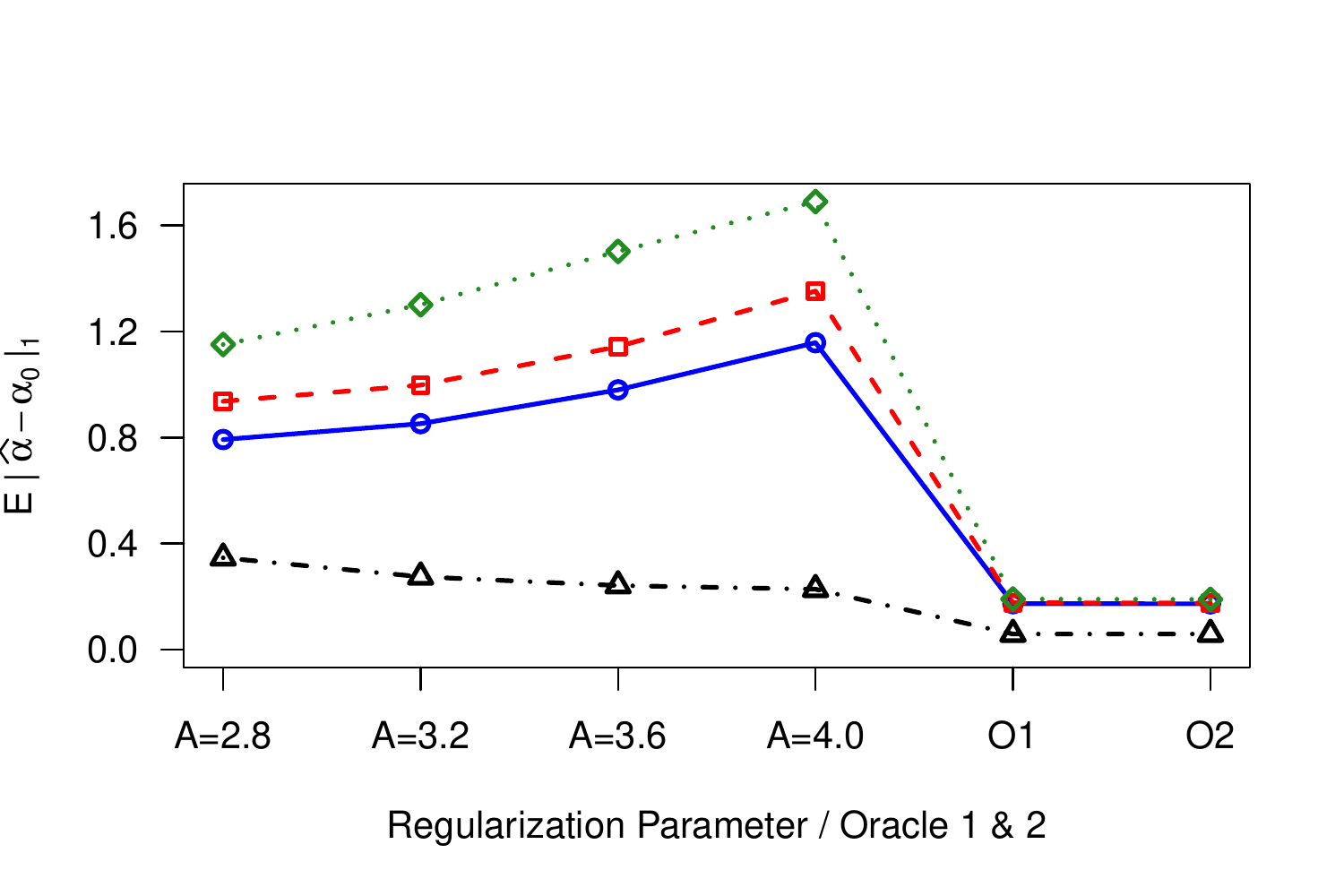}
\includegraphics[width = 2.8in]{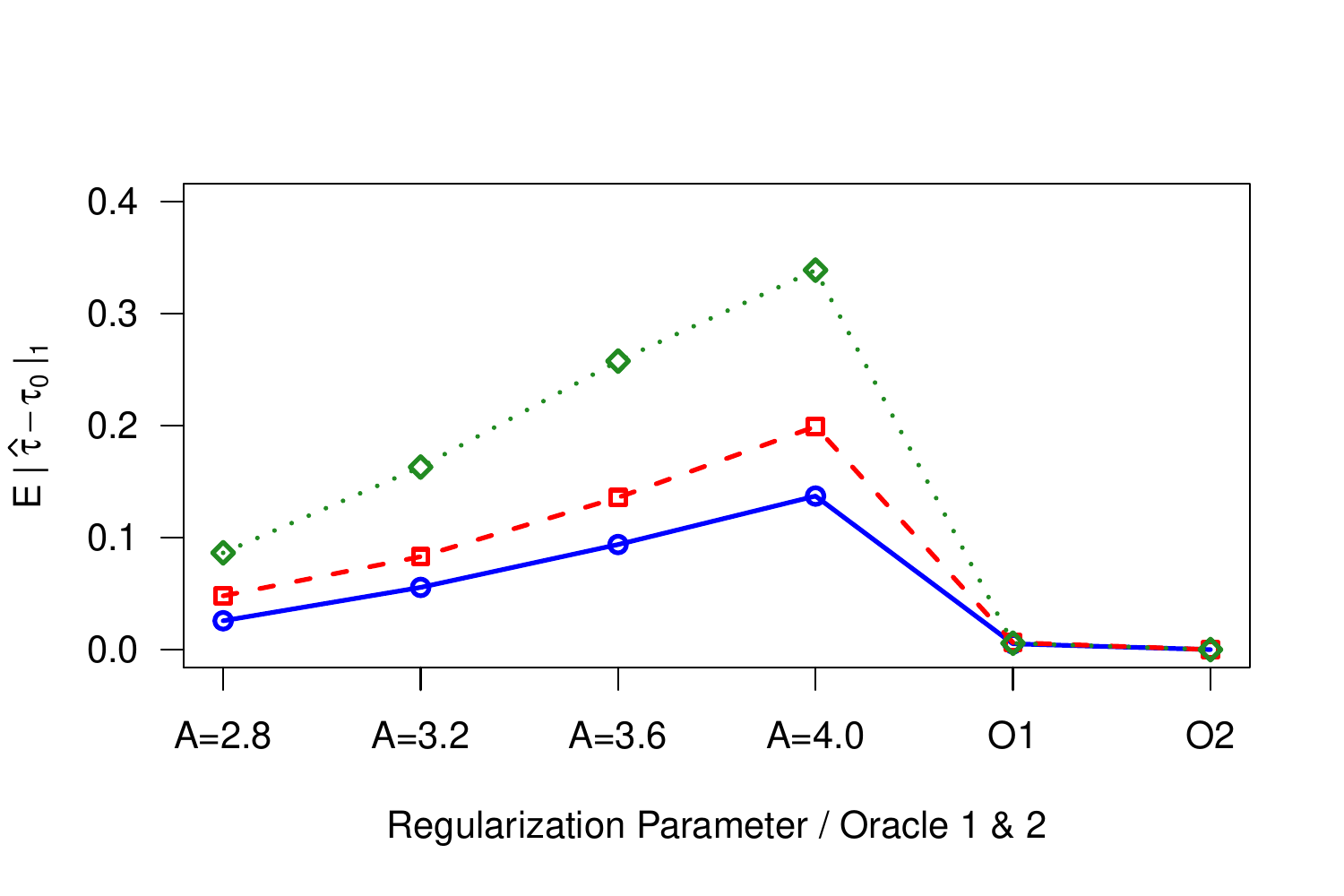} \\[0pt]
$M=200$\\[0pt]
\includegraphics[width = 2.8in]{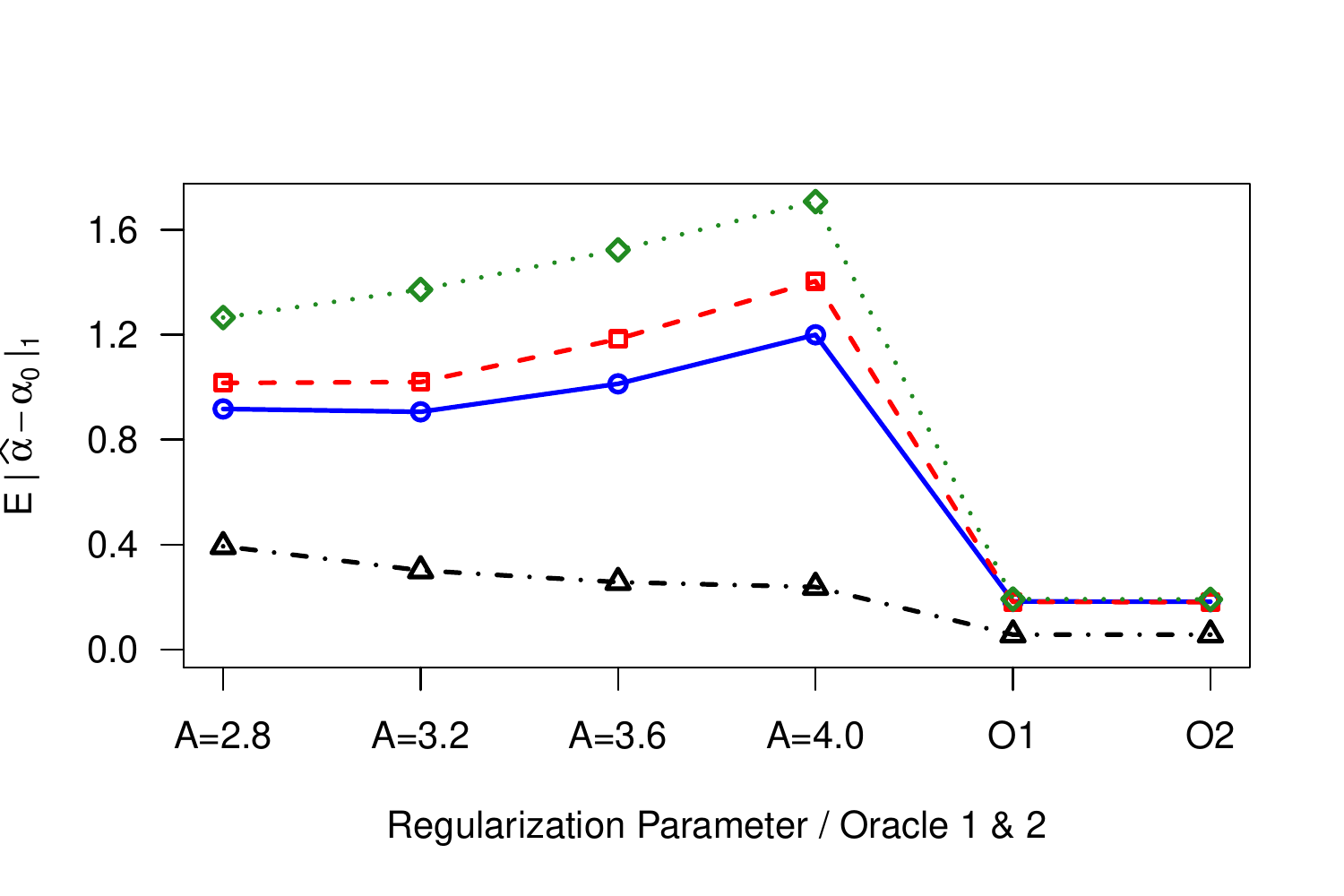}
\includegraphics[width = 2.8in]{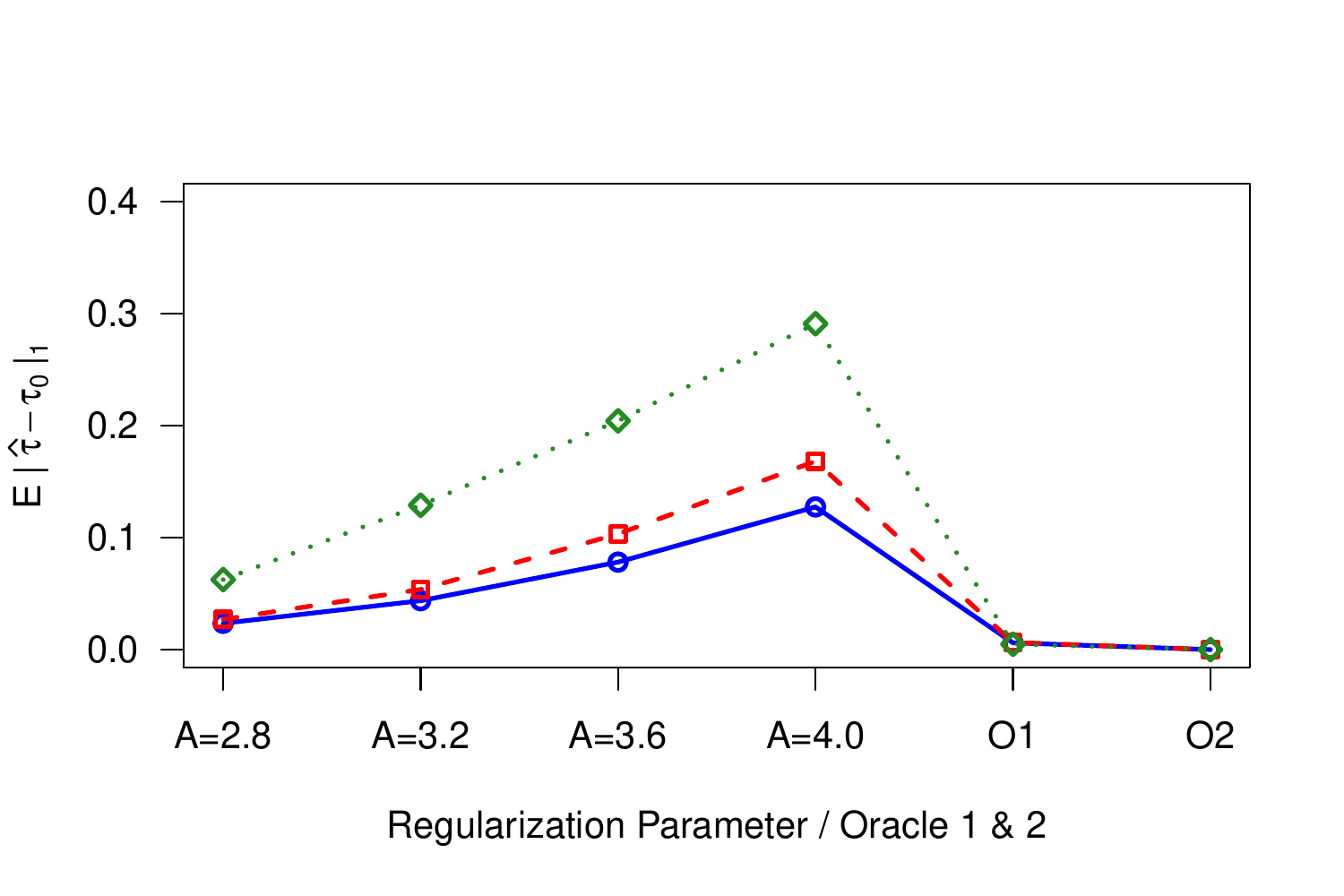} \\[0pt]
$M=400$\\[0pt]
\end{center}
\par
\label{fg:depX2}
\end{figure}

\begin{figure}[tbp]
\caption{Probability of Selecting True Parameters when $\rho=0$ (left panel) 
 and   $\rho=0.3$ (right panel)}
\begin{center}

\includegraphics[width = 2.8in]{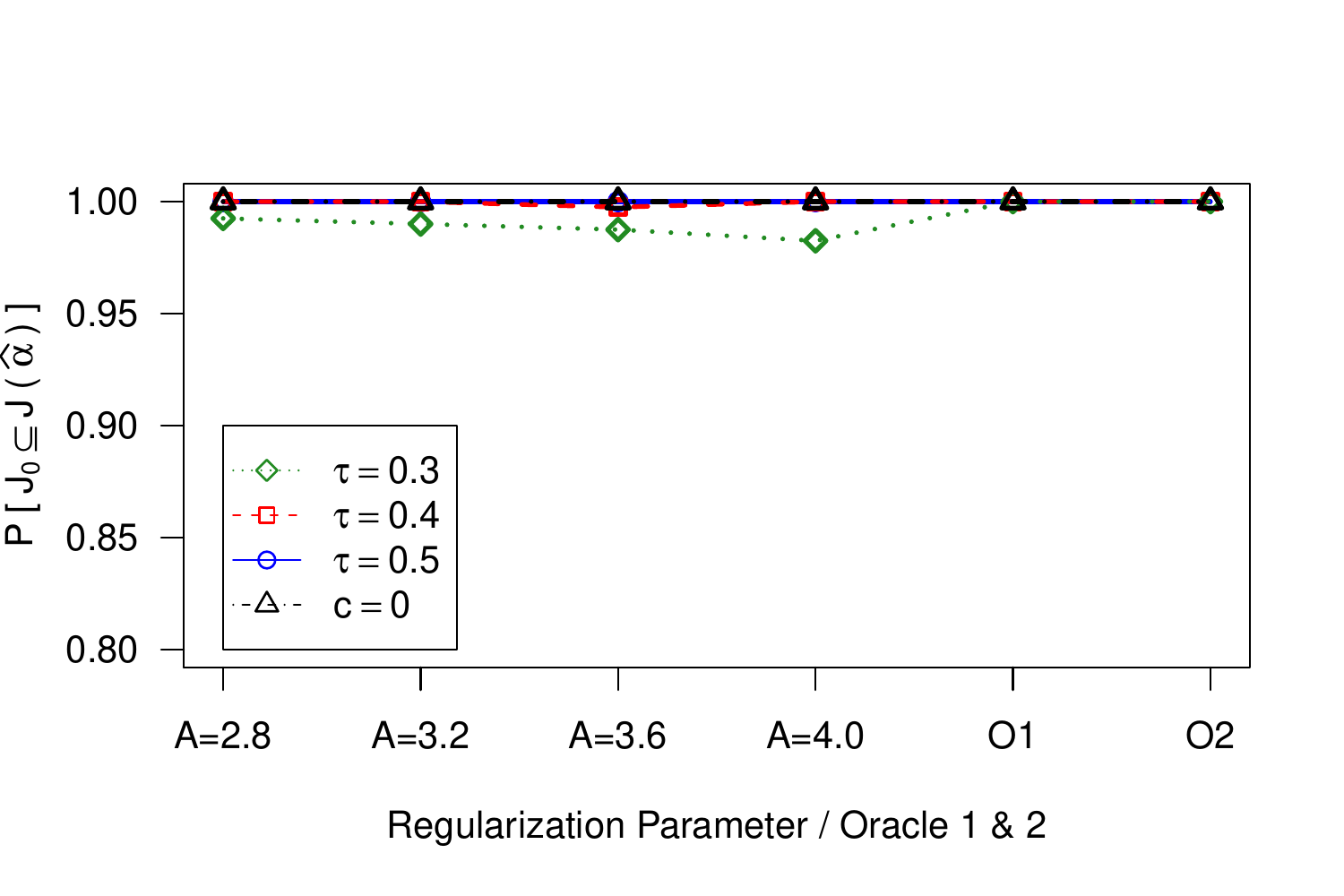}
\includegraphics[width = 2.8in]{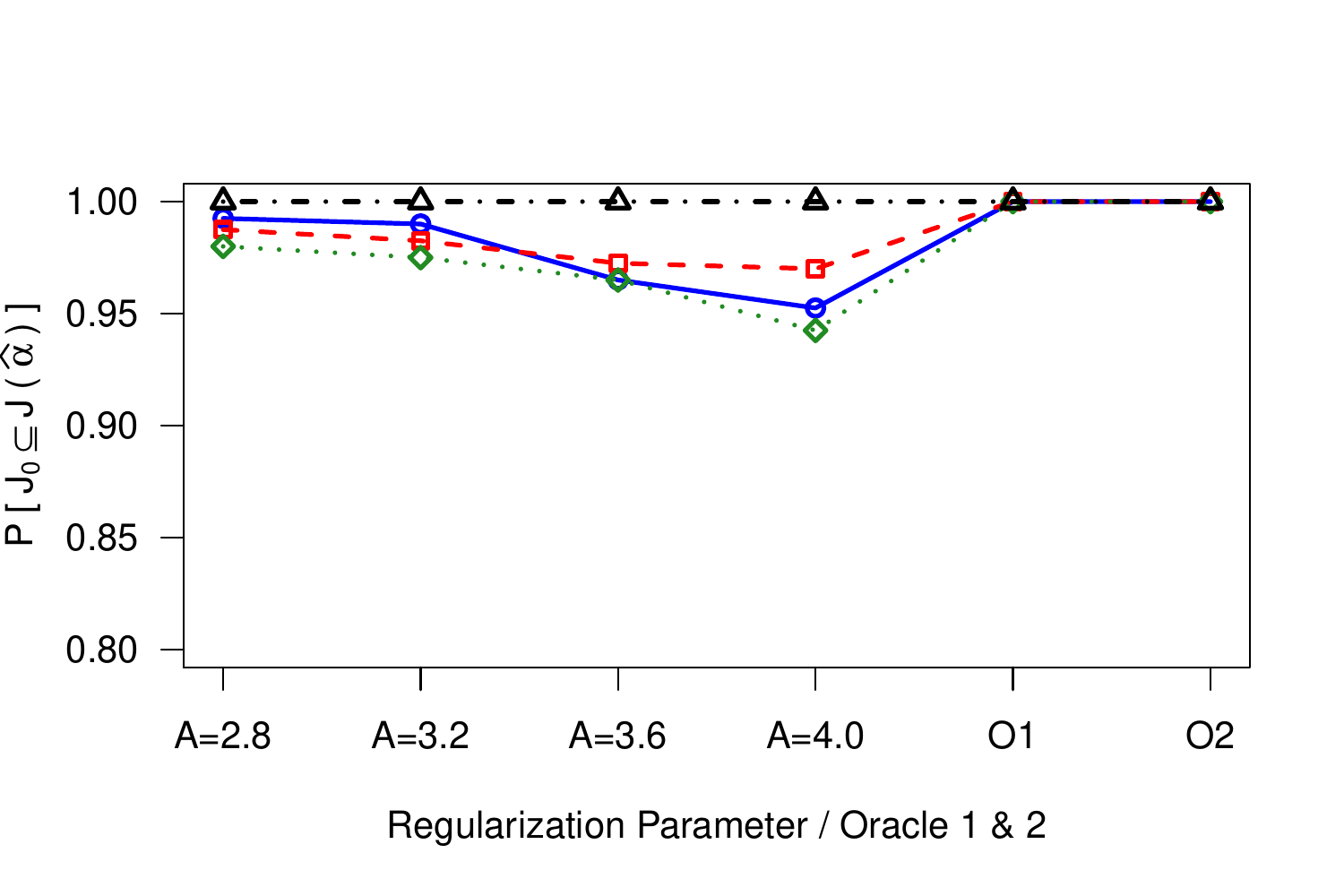}  \\[0pt]
$M=50$\\[0pt]
\includegraphics[width = 2.8in]{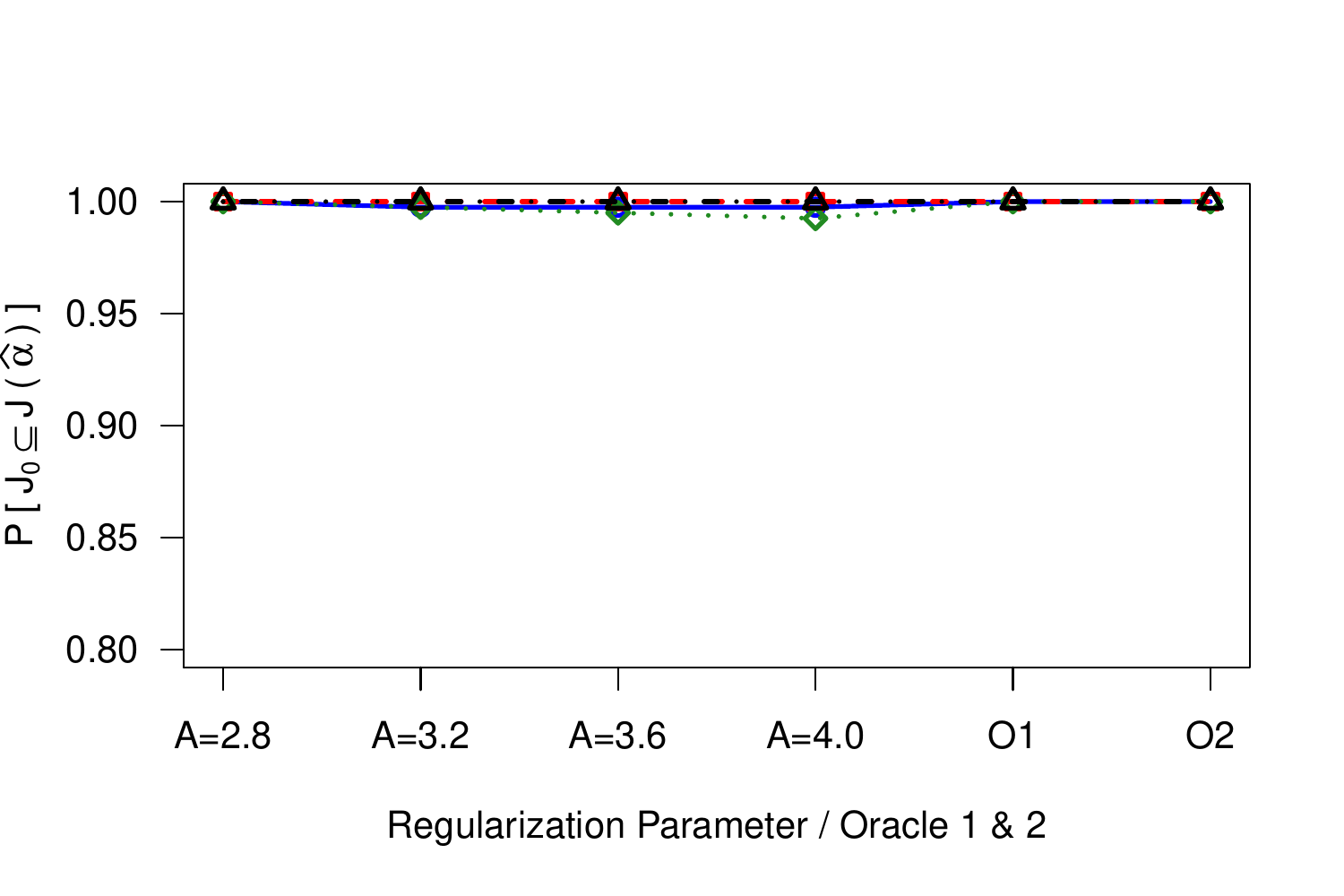}
\includegraphics[width = 2.8in]{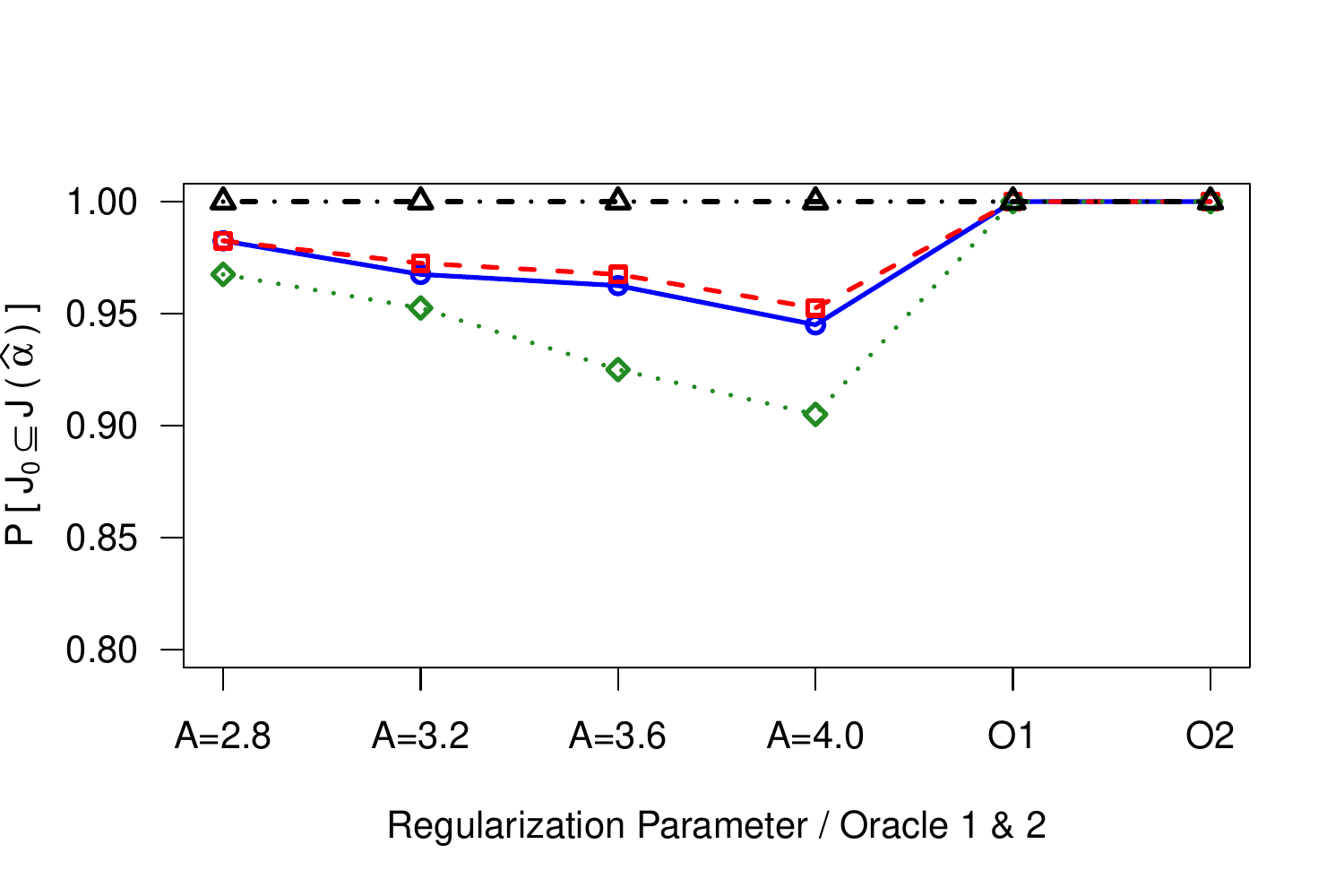}  \\[0pt]
$M=100$\\[0pt]
\includegraphics[width = 2.8in]{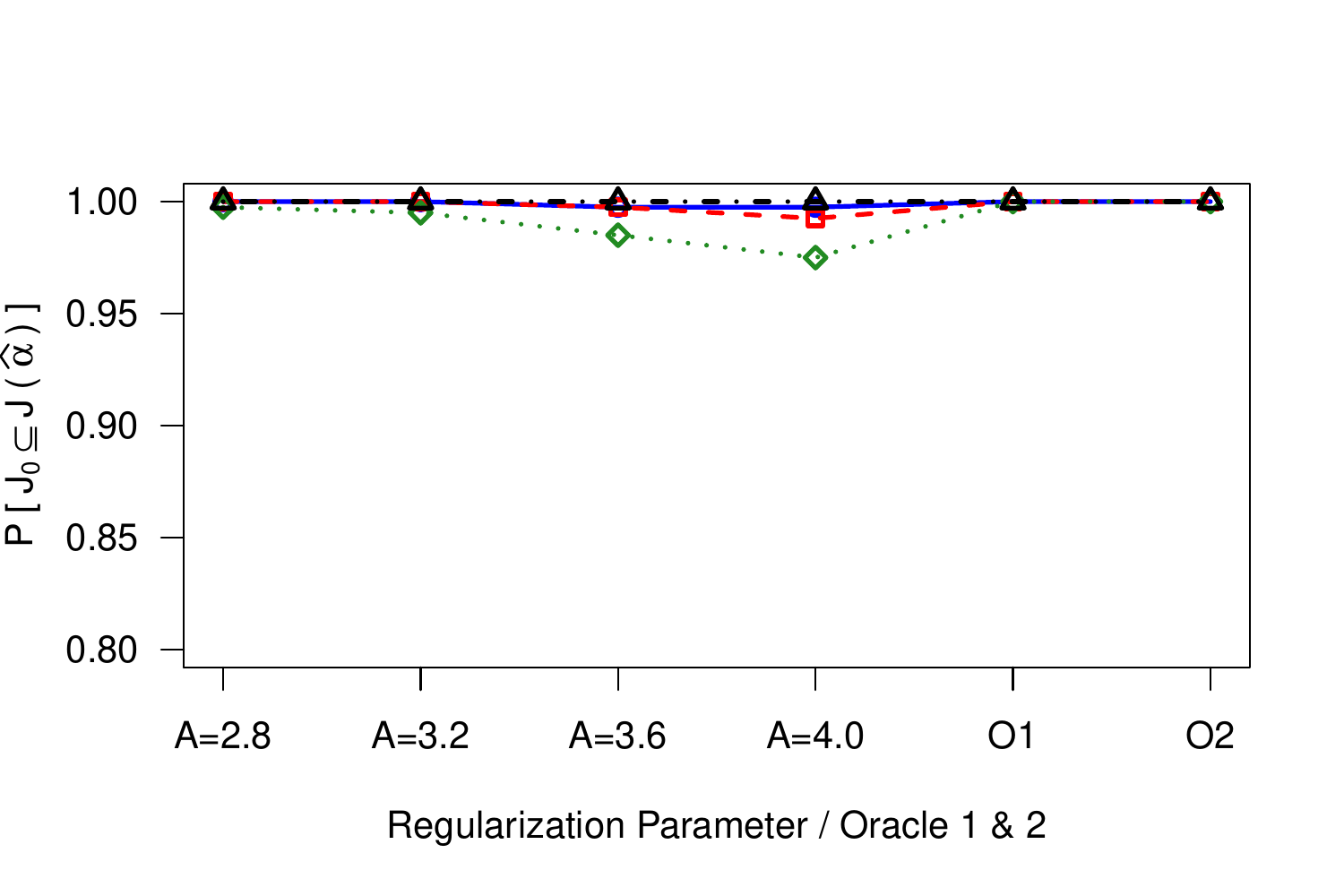}
\includegraphics[width = 2.8in]{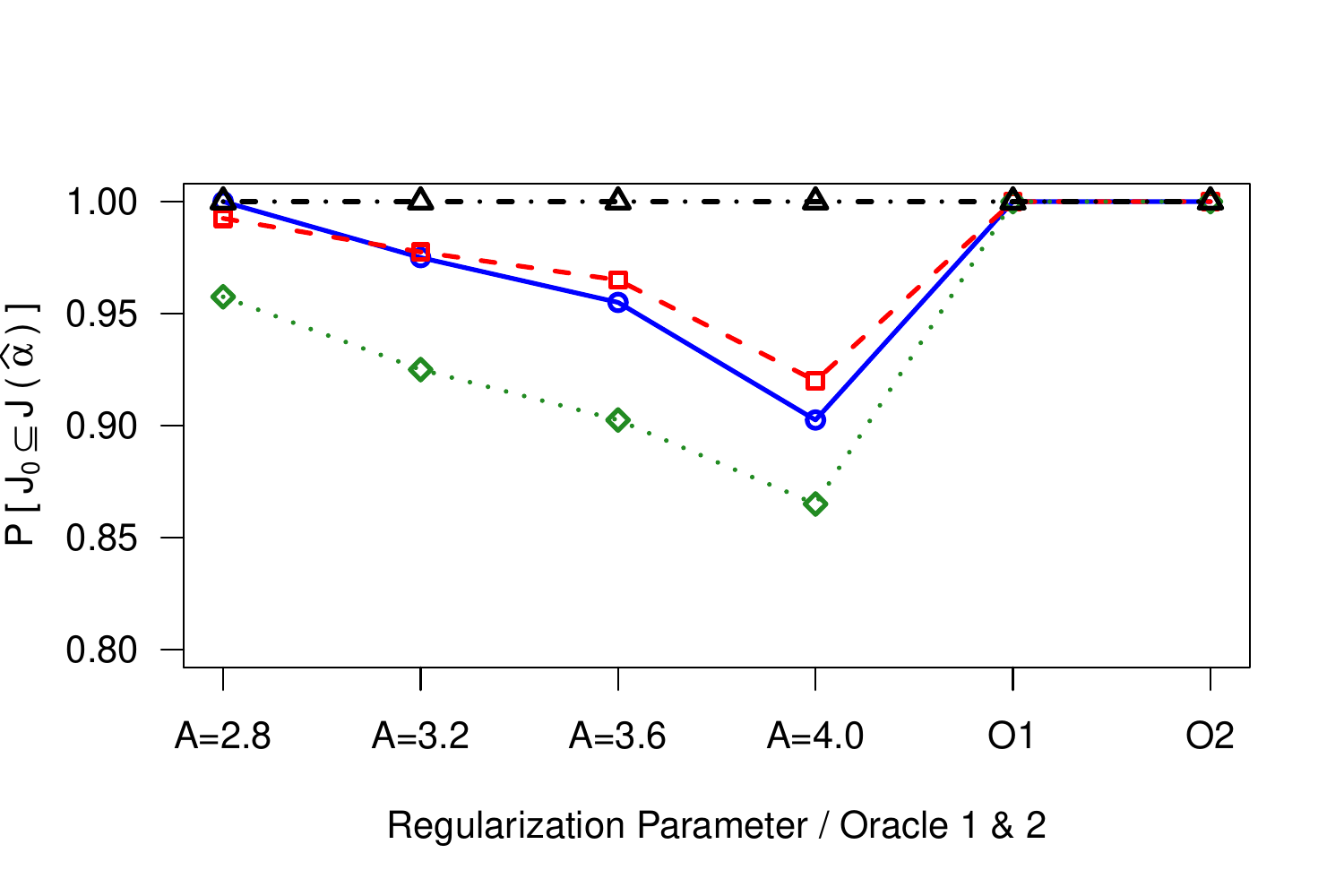}  \\[0pt]
$M=200$\\[0pt]
\includegraphics[width = 2.8in]{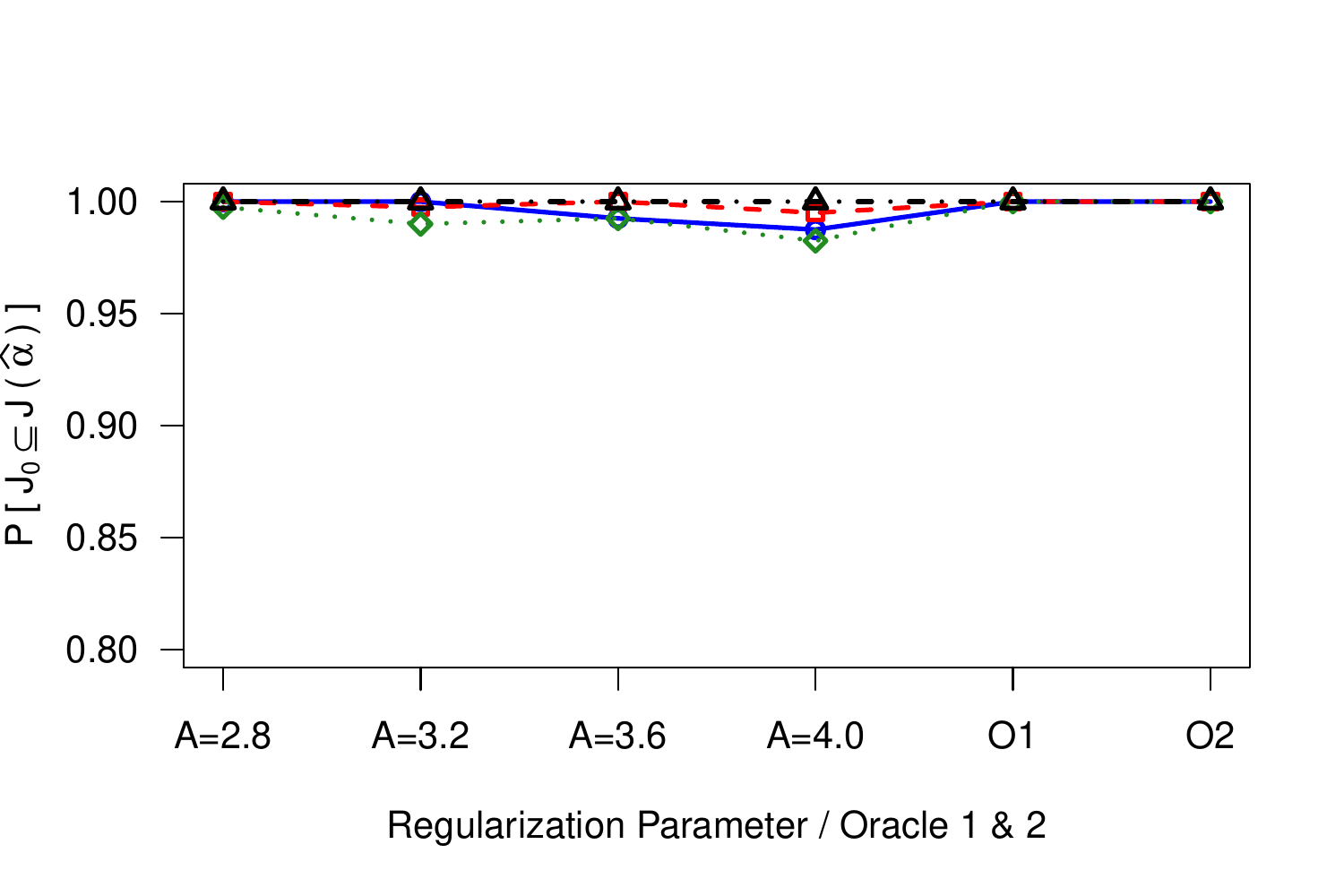}
\includegraphics[width = 2.8in]{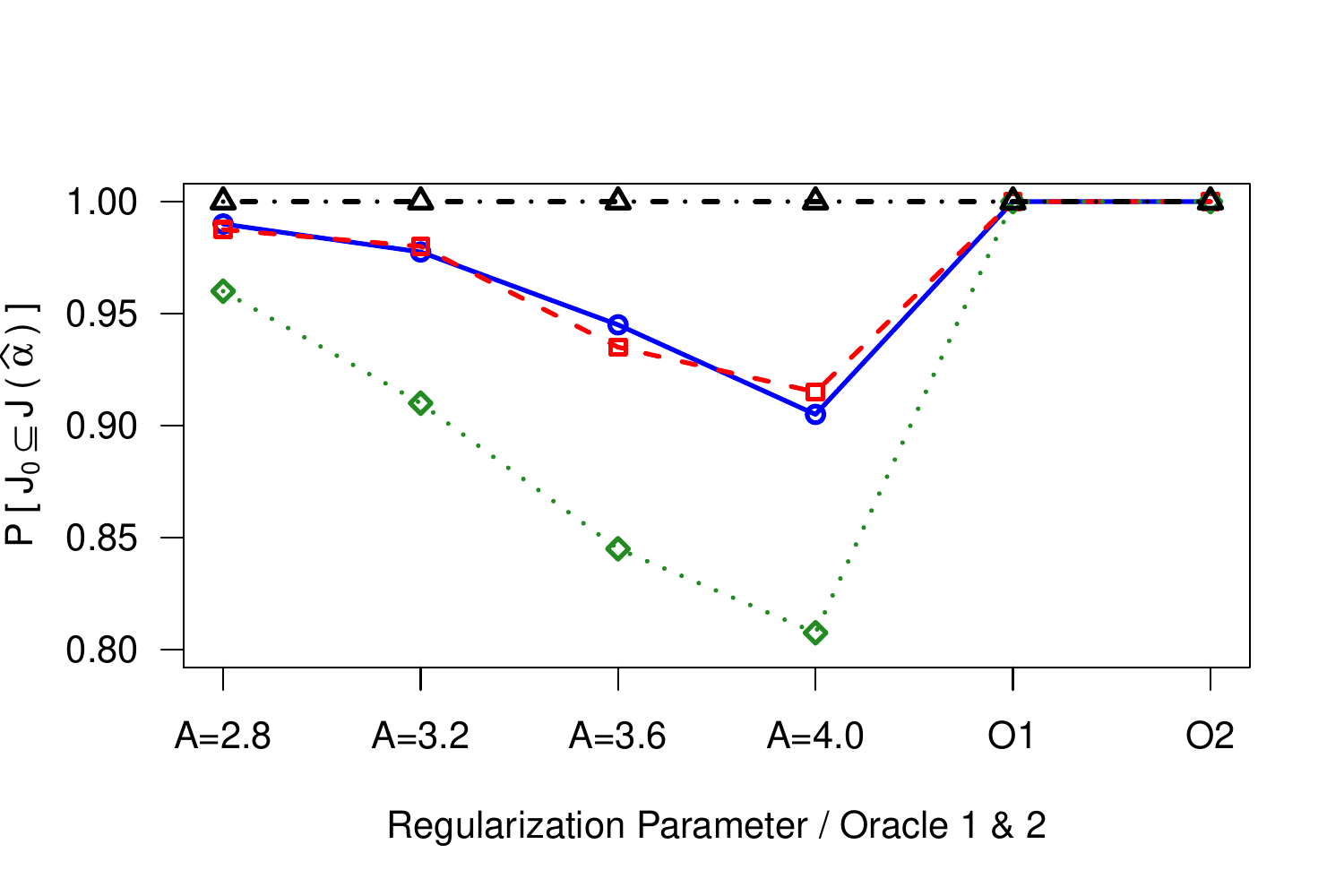}  \\[0pt]
$M=400$\\[0pt]
\end{center}
\par
\label{fg:cov}
\end{figure}

\clearpage

\bibliographystyle{chicago}
\bibliography{LeeSeoShin_13Mar2014}

\begin{thebibliography}{}

\bibitem[\protect\citeauthoryear{Barro and Lee}{Barro and
  Lee}{1994}]{Barro:Lee:94}
Barro, R. and J.~Lee (1994).
\newblock {\em Data set for a panel of 139 countries}.
\newblock NBER.
\newblock Available at {http://admin.nber.org/pub/barro.lee/}.

\bibitem[\protect\citeauthoryear{Barro and {Sala-i-Martin}}{Barro and
  {Sala-i-Martin}}{1995}]{barrosala}
Barro, R. and X.~{Sala-i-Martin} (1995).
\newblock {\em Economic Growth}.
\newblock McGraw-Hill.
\newblock New York.

\bibitem[\protect\citeauthoryear{Belloni and Chernozhukov}{Belloni and
  Chernozhukov}{2011a}]{Belloni:Chernozhukov:11}
Belloni, A. and V.~Chernozhukov (2011a).
\newblock {$\ell_1$}-penalized quantile regression in high-dimensional sparse
  models.
\newblock {\em Ann. Statist.\/}~{\em 39\/}(1), 82--130.

\bibitem[\protect\citeauthoryear{Belloni and Chernozhukov}{Belloni and
  Chernozhukov}{2011b}]{belloni2011high}
Belloni, A. and V.~Chernozhukov (2011b).
\newblock High dimensional sparse econometric models: An introduction.
\newblock In P.~Alquier, E.~Gautier, and G.~Stoltz (Eds.), {\em Inverse
  Problems and High-Dimensional Estimation}, Volume 203 of {\em Lecture Notes
  in Statistics}, pp.\  121--156. Springer Berlin Heidelberg.

\bibitem[\protect\citeauthoryear{Bickel, Ritov, and Tsybakov}{Bickel
  et~al.}{2009}]{Bickel-et-al:09}
Bickel, P.~J., Y.~Ritov, and A.~B. Tsybakov (2009).
\newblock Simultaneous analysis of {L}asso and {D}antzig selector.
\newblock {\em Ann. Statist.\/}~{\em 37\/}(4), 1705--1732.

\bibitem[\protect\citeauthoryear{Bradic, Fan, and Jiang}{Bradic
  et~al.}{2012}]{bradic2012regularization}
Bradic, J., J.~Fan, and J.~Jiang (2012).
\newblock Regularization for {Cox}'s proportional hazards model with
  {NP}-dimensionality.
\newblock {\em Annals of Statistics\/}~{\em 39\/}(6), 3092--3120.

\bibitem[\protect\citeauthoryear{Bradic, Fan, and Wang}{Bradic
  et~al.}{2011}]{Bradic:Fan:Wang:12}
Bradic, J., J.~Fan, and W.~Wang (2011).
\newblock Penalized composite quasi-likelihood for ultrahigh dimensional
  variable selection.
\newblock {\em Journal of the Royal Statistical Society: Series B (Statistical
  Methodology)\/}~{\em 73\/}(3), 325--349.

\bibitem[\protect\citeauthoryear{B\"{u}hlmann and van~de Geer}{B\"{u}hlmann and
  van~de Geer}{2011}]{BvdG:11}
B\"{u}hlmann, P. and S.~van~de Geer (2011).
\newblock {\em Statistics for High-dimensional Data: Methods, Theory and
  Applications}.
\newblock New York: Springer.

\bibitem[\protect\citeauthoryear{Bunea, Tsybakov, and Wegkamp}{Bunea
  et~al.}{2007}]{Bunea-et-al:07b}
Bunea, F., A.~Tsybakov, and M.~Wegkamp (2007).
\newblock Sparsity oracle inequalities for the {L}asso.
\newblock {\em Electronic Journal of Statistics\/}~{\em 1}, 169--194.

\bibitem[\protect\citeauthoryear{Cand{\`e}s and Tao}{Cand{\`e}s and
  Tao}{2007}]{Candes:Tao:07}
Cand{\`e}s, E. and T.~Tao (2007).
\newblock The {D}antzig selector: statistical estimation when {$p$} is much
  larger than {$n$}.
\newblock {\em Ann. Statist.\/}~{\em 35\/}(6), 2313--2351.

\bibitem[\protect\citeauthoryear{Card, Mas, and Rothstein}{Card
  et~al.}{2008}]{Card:Mas:Rothstein:09}
Card, D., A.~Mas, and J.~Rothstein (2008).
\newblock Tipping and the dynamics of segregation.
\newblock {\em Quarterly Journal of Economics\/}~{\em 123\/}(1), 177--218.

\bibitem[\protect\citeauthoryear{Chan}{Chan}{1993}]{chan:93}
Chan, K.~S. (1993).
\newblock Consistency and limiting distribution of the least squares estimator
  of a threshold autoregressive model.
\newblock {\em Annals of Statistics\/}~{\em 21}, 520--533.

\bibitem[\protect\citeauthoryear{Ciuperca}{Ciuperca}{2012}]{Ciuperca:12}
Ciuperca, G. (2012).
\newblock Model selection by {LASSO} methods in a change-point model.
\newblock Working Paper arXiv:1107.0865v2, Institut Camille Jordan.
\newblock available at \url{http://arxiv.org/abs/1107.0865v2}.

\bibitem[\protect\citeauthoryear{Durlauf, Johnson, and Temple}{Durlauf
  et~al.}{2005}]{durlauf2005growth}
Durlauf, S., P.~Johnson, and J.~Temple (2005).
\newblock Growth econometrics.
\newblock {\em Handbook of economic growth\/}~{\em 1}, 555--677.

\bibitem[\protect\citeauthoryear{Durlauf and Johnson}{Durlauf and
  Johnson}{1995}]{Durlauf:Johnson:95}
Durlauf, S.~N. and P.~A. Johnson (1995).
\newblock Multiple regimes and cross-country growth behavior.
\newblock {\em Journal of Applied Econometrics\/}~{\em 10\/}(4), 365--384.

\bibitem[\protect\citeauthoryear{Efron, Hastie, Johnstone, and
  Tibshirani}{Efron et~al.}{2004}]{efron2004least}
Efron, B., T.~Hastie, I.~Johnstone, and R.~Tibshirani (2004).
\newblock Least angle regression.
\newblock {\em Annals of statistics\/}~{\em 32\/}(2), 407--499.

\bibitem[\protect\citeauthoryear{Fan and Li}{Fan and Li}{2001}]{Fan:Li:01}
Fan, J. and R.~Li (2001).
\newblock Variable selection via nonconcave penalized likelihood and its oracle
  properties,.
\newblock {\em Journal of the American Statistical Association\/}~{\em 96},
  1348.

\bibitem[\protect\citeauthoryear{Fan and Lv}{Fan and Lv}{2010}]{Fan:Lv:10}
Fan, J. and J.~Lv (2010).
\newblock A selective overview of variable selection in high dimensional
  feature space.
\newblock {\em Statistica Sinica\/}~{\em 20}, 101--148.

\bibitem[\protect\citeauthoryear{Fan and Lv}{Fan and Lv}{2011}]{Fan:Lv:11}
Fan, J. and J.~Lv (2011).
\newblock Nonconcave penalized likelihood with np-dimensionality.
\newblock {\em Information Theory, IEEE Transactions on\/}~{\em 57\/}(8),
  5467--5484.

\bibitem[\protect\citeauthoryear{Fan and Peng}{Fan and
  Peng}{2004}]{Fan:Peng:04}
Fan, J. and H.~Peng (2004).
\newblock Nonconcave penalized likelihood with a diverging number of
  parameters.
\newblock {\em Ann. Statist.\/}~{\em 32\/}(3), 928--961.

\bibitem[\protect\citeauthoryear{Frick, Munk, and Sieling}{Frick
  et~al.}{2014}]{frick2013multiscale}
Frick, K., A.~Munk, and H.~Sieling (2014).
\newblock Multiscale change-point inference.
\newblock {\em Journal of the Royal Statistical Society Series B\/}.
\newblock forthcoming.

\bibitem[\protect\citeauthoryear{Hansen}{Hansen}{2000}]{Hansen:00}
Hansen, B.~E. (2000).
\newblock Sample splitting and threshold estimation.
\newblock {\em Econometrica\/}~{\em 68\/}(3), 575--603.

\bibitem[\protect\citeauthoryear{Harchaoui and L{\'e}vy-Leduc}{Harchaoui and
  L{\'e}vy-Leduc}{2008}]{Harchaoui:Levy-Leduc:08}
Harchaoui, Z. and C.~L{\'e}vy-Leduc (2008).
\newblock Catching change-points with {L}asso.
\newblock In {\em Advances in Neural Information Processing Systems}, Volume
  Vol. 20, Cambridge, MA. MIT Press.

\bibitem[\protect\citeauthoryear{Harchaoui and L{\'e}vy-Leduc}{Harchaoui and
  L{\'e}vy-Leduc}{2010}]{Harchaoui:Levy-Leduc:10}
Harchaoui, Z. and C.~L{\'e}vy-Leduc (2010).
\newblock Multiple change-point estimation with a total variation penalty.
\newblock {\em Journal of the American Statistical Association\/}~{\em
  105\/}(492), 1480--1493.

\bibitem[\protect\citeauthoryear{Huang, Horowitz, and Ma}{Huang
  et~al.}{2008}]{Huang:Horowitz:Ma:08}
Huang, J., J.~L. Horowitz, and M.~S. Ma (2008).
\newblock Asymptotic properties of bridge estimators in sparse high-dimensional
  regression models.
\newblock {\em Ann. Statist.\/}~{\em 36\/}(2), 587--613.

\bibitem[\protect\citeauthoryear{Huang, Ma, and Zhang}{Huang
  et~al.}{2008}]{Huang:Ma:Zhang:08}
Huang, J., S.~G. Ma, and C.-H. Zhang (2008).
\newblock Adaptive lasso for sparse high-dimensional regression models,.
\newblock {\em Statistica Sinica\/}~{\em 18}, 1603.

\bibitem[\protect\citeauthoryear{Kim, Choi, and Oh}{Kim
  et~al.}{2008}]{Kim:Choi:Oh:08}
Kim, Y., H.~Choi, and H.-S. Oh (2008).
\newblock Smoothly clipped absolute deviation on high dimensions,.
\newblock {\em Journal of the American Statistical Association\/}~{\em 103},
  1665.

\bibitem[\protect\citeauthoryear{Lee, Seo, and Shin}{Lee
  et~al.}{2011}]{lee2011testing}
Lee, S., M.~Seo, and Y.~Shin (2011).
\newblock Testing for threshold effects in regression models.
\newblock {\em Journal of the American Statistical Association\/}~{\em
  106\/}(493), 220--231.

\bibitem[\protect\citeauthoryear{Lin and Lv}{Lin and Lv}{2013}]{Lin:Lv:12}
Lin, W. and J.~Lv (2013).
\newblock High-dimensional sparse additive hazards regression.
\newblock {\em Journal of the American Statistical Association\/}~{\em
  108\/}(501), 247--264.

\bibitem[\protect\citeauthoryear{Meinshausen and Yu}{Meinshausen and
  Yu}{2009}]{Meinshausen:Yu:09}
Meinshausen, N. and B.~Yu (2009).
\newblock Lasso-type recovery of sparse representations for high-dimensional
  data.
\newblock {\em Ann. Statist.\/}~{\em 37\/}(1), 246--270.

\bibitem[\protect\citeauthoryear{Pesaran and Pick}{Pesaran and
  Pick}{2007}]{Pes:Pick}
Pesaran, M.~H. and A.~Pick (2007).
\newblock Econometric issues in the analysis of contagion.
\newblock {\em Journal of Economic Dynamics and Control\/}~{\em 31\/}(4),
  1245--1277.

\bibitem[\protect\citeauthoryear{Pollard}{Pollard}{1984}]{Pollard:84}
Pollard, D. (1984).
\newblock {\em Convergence of Stochastic Processes}.
\newblock New York, NY: Springer.

\bibitem[\protect\citeauthoryear{Raskutti, Wainwright, and Yu}{Raskutti
  et~al.}{2010}]{RWY:10}
Raskutti, G., M.~J. Wainwright, and B.~Yu (2010).
\newblock Restricted eigenvalue properties for correlated gaussian designs.
\newblock {\em Journal of Machine Learning Research\/}~{\em 11}, 2241--2259.

\bibitem[\protect\citeauthoryear{Raskutti, Wainwright, and Yu}{Raskutti
  et~al.}{2011}]{raskutti2011minimax}
Raskutti, G., M.~J. Wainwright, and B.~Yu (2011).
\newblock Minimax rates of estimation for high-dimensional linear regression
  over-balls.
\newblock {\em Information Theory, IEEE Transactions on\/}~{\em 57\/}(10),
  6976--6994.

\bibitem[\protect\citeauthoryear{Raskutti, Wainwright, and Yu}{Raskutti
  et~al.}{2012}]{raskutti2012minimax}
Raskutti, G., M.~J. Wainwright, and B.~Yu (2012).
\newblock Minimax-optimal rates for sparse additive models over kernel classes
  via convex programming.
\newblock {\em The Journal of Machine Learning Research\/}~{\em 13}, 389--427.

\bibitem[\protect\citeauthoryear{Seijo and Sen}{Seijo and
  Sen}{2011a}]{Seijo:Sen:11a}
Seijo, E. and B.~Sen (2011a).
\newblock Change-point in stochastic design regression and the bootstrap.
\newblock {\em Ann. Statist.\/}~{\em 39\/}(3), 1580--1607.

\bibitem[\protect\citeauthoryear{Seijo and Sen}{Seijo and
  Sen}{2011b}]{Seijo:Sen:11b}
Seijo, E. and B.~Sen (2011b).
\newblock A continuous mapping theorem for the smallest argmax functional.
\newblock {\em Electron. J. Statist.\/}~{\em 5}, 421--439.

\bibitem[\protect\citeauthoryear{Tibshirani}{Tibshirani}{1996}]{Tibshirani:96}
Tibshirani, R. (1996).
\newblock Regression shrinkage and selection via the lasso.
\newblock {\em J. Roy. Statist. Soc. Ser. B\/}~{\em 58\/}(1), 267--288.

\bibitem[\protect\citeauthoryear{Tibshirani}{Tibshirani}{2011}]{Tibshirani:11}
Tibshirani, R. (2011).
\newblock Regression shrinkage and selection via the lasso: a retrospective.
\newblock {\em J. Roy. Statist. Soc. Ser. B\/}~{\em 73\/}(3), 273--282.

\bibitem[\protect\citeauthoryear{Tong}{Tong}{1990}]{Tong:90}
Tong, H. (1990).
\newblock {\em Non-linear Time Series: A Dynamical System Approach}.
\newblock New York: Oxford University Press.

\bibitem[\protect\citeauthoryear{van~de Geer}{van~de Geer}{2008}]{vdGeer:08}
van~de Geer, S.~A. (2008).
\newblock High-dimensional generalized linear models and the lasso.
\newblock {\em Annals of Statistics\/}~{\em 36\/}(2), 614--645.

\bibitem[\protect\citeauthoryear{van~de Geer and B{\"u}hlmann}{van~de Geer and
  B{\"u}hlmann}{2009}]{vdGeer:Buhlmann:09}
van~de Geer, S.~A. and P.~B{\"u}hlmann (2009).
\newblock On the conditions used to prove oracle results for the {L}asso.
\newblock {\em Electron. J. Stat.\/}~{\em 3}, 1360--1392.

\bibitem[\protect\citeauthoryear{van~der Vaart and Wellner}{van~der Vaart and
  Wellner}{1996}]{VW}
van~der Vaart, A.~W. and J.~A. Wellner (1996).
\newblock {\em Weak Convergence and Empirical Process}.
\newblock Springer, New York.

\bibitem[\protect\citeauthoryear{Wang, Wu, and Li}{Wang
  et~al.}{2012}]{Wang:Wu:Li:12}
Wang, L., Y.~Wu, and R.~Li (2012).
\newblock Quantile regression for analyzing heterogeneity in ultra-high
  dimension.
\newblock {\em Journal of the American Statistical Association\/}~{\em
  107\/}(497), 214--222.

\bibitem[\protect\citeauthoryear{Wu}{Wu}{2008}]{Wu:08}
Wu, Y. (2008).
\newblock Simultaneous change point analysis and variable selection in a
  regression problem.
\newblock {\em Journal of Multivariate Analysis\/}~{\em 99\/}(9), 2154 -- 2171.

\bibitem[\protect\citeauthoryear{Zhang and Siegmund}{Zhang and
  Siegmund}{2012}]{Zhang:Siegmund:12}
Zhang, N.~R. and D.~O. Siegmund (2012).
\newblock Model selection for high dimensional multi-sequence change-point
  problems.
\newblock {\em Statistica Sinica\/}~{\em 22}, 1507--1538.

\bibitem[\protect\citeauthoryear{Zou}{Zou}{2006}]{Zou:06}
Zou, H. (2006).
\newblock The adaptive lasso and its oracle properties.
\newblock {\em J. Amer. Statist. Assoc.\/}~{\em 101\/}(476), 1418--1429.

\end{thebibliography}

%\clearpage

\end{document}